\documentclass[12pt,a4paper]{amsart}
\usepackage{amssymb}
\usepackage[cmtip,all]{xy}
\usepackage{hyperref}

\calclayout

\newcommand{\+}{\nobreakdash-}
\renewcommand{\:}{\colon}

\newcommand{\rarrow}{\longrightarrow}

\newcommand{\ot}{\otimes}

\newcommand{\lrarrow}{\mskip.5\thinmuskip\relbar\joinrel\relbar
   \joinrel\rightarrow\mskip.5\thinmuskip\relax}

\DeclareFontFamily{U}{mathx}{\hyphenchar\font45}
\DeclareFontShape{U}{mathx}{m}{n}{
      <5> <6> <7> <8> <9> <10>
      <10.95> <12> <14.4> <17.28> <20.74> <24.88>
      mathx10
      }{}
\DeclareSymbolFont{mathx}{U}{mathx}{m}{n}
\DeclareFontSubstitution{U}{mathx}{m}{n}
\DeclareMathAccent{\widecheck}{0}{mathx}{"71}

\DeclareMathOperator{\Hom}{Hom}
\DeclareMathOperator{\Ext}{Ext}
\DeclareMathOperator{\Spec}{Spec}

\newcommand{\Modl}{{\operatorname{\mathsf{--Mod}}}}
\newcommand{\Qcoh}{{\operatorname{\mathsf{--Qcoh}}}}
\newcommand{\Ctrh}{{\operatorname{\mathsf{--Ctrh}}}}
\newcommand{\Lcth}{{\operatorname{\mathsf{--Lcth}}}}
\newcommand{\Sh}{{\operatorname{\mathsf{--Sh}}}}
\newcommand{\Cosh}{{\operatorname{\mathsf{--Cosh}}}}

\newcommand{\ctra}{{\operatorname{\mathsf{-ctra}}}}

\newcommand{\Ab}{\mathsf{Ab}}

\newcommand{\sop}{\mathsf{op}}

\newcommand{\prj}{\mathsf{prj}}
\newcommand{\inj}{\mathsf{inj}}
\newcommand{\fl}{\mathsf{fl}}
\renewcommand{\cot}{\mathsf{cot}}
\newcommand{\vfl}{\mathsf{vfl}}
\newcommand{\cta}{\mathsf{cta}}
\newcommand{\qs}{\mathsf{qs}}
\newcommand{\lct}{\mathsf{lct}}
\newcommand{\lin}{\mathsf{lin}}
\newcommand{\alf}{\mathsf{alf}}
\newcommand{\cfq}{\mathsf{cfq}}

\newcommand{\bB}{\mathbf B}
\newcommand{\bW}{\mathbf W}
\newcommand{\bT}{\mathbf T}

\newcommand{\cO}{\mathcal O}
\newcommand{\cF}{\mathcal F}
\newcommand{\cI}{\mathcal I}
\newcommand{\cM}{\mathcal M}
\newcommand{\cN}{\mathcal N}

\newcommand{\fE}{\mathfrak E}
\newcommand{\fF}{\mathfrak F}
\newcommand{\fG}{\mathfrak G}
\newcommand{\fH}{\mathfrak H}
\newcommand{\fI}{\mathfrak I}
\newcommand{\fJ}{\mathfrak J}
\newcommand{\fK}{\mathfrak K}
\newcommand{\fL}{\mathfrak L}
\newcommand{\fM}{\mathfrak M}
\newcommand{\fN}{\mathfrak N}
\newcommand{\fP}{\mathfrak P}
\newcommand{\fQ}{\mathfrak Q}
\newcommand{\fR}{\mathfrak R}

\newcommand{\m}{\mathfrak m}
\newcommand{\p}{\mathfrak p}

\newcommand{\sK}{\mathsf K}
\newcommand{\sE}{\mathsf E}
\newcommand{\sF}{\mathsf F}
\newcommand{\sC}{\mathsf C}
\newcommand{\sS}{\mathsf S}

\newcommand{\bC}{\mathbf C}
\newcommand{\bX}{\mathbf X}
\newcommand{\bomega}{\boldsymbol\omega}

\newcommand{\boL}{\mathbb L}

\newcommand{\boQ}{\mathbb Q}
\newcommand{\boZ}{\mathbb Z}

\newcommand{\Section}[1]{\bigskip\section{#1}\medskip}
\theoremstyle{plain}
\newtheorem{thm}{Theorem}[section]
\newtheorem{prop}[thm]{Proposition}
\newtheorem{lem}[thm]{Lemma}
\newtheorem{cor}[thm]{Corollary}
\theoremstyle{definition}
\newtheorem{rem}[thm]{Remark}
\newtheorem{rems}[thm]{Remarks}

\newtheorem{ex}[thm]{Example}

\begin{document}

\title{The contraherent version of the theorem \\
of Sl\'avik and \v S\v tov\'\i\v cek}

\author{Leonid Positselski}

\address{Institute of Mathematics, Czech Academy of Sciences \\
\v Zitn\'a~25, 115~67 Prague~1 \\ Czech Republic} 

\email{positselski@math.cas.cz}

\begin{abstract}
 This is a paper about contraherent cosheaves on non-semi-separated
schemes.
 We prove two theorems, a negative one and a positive one.
 On the negative side, let $X$ be a quasi-compact, quasi-separated
scheme that is not semi-separated.
 We present a locally cotorsion contraherent cosheaf on $X$ that
does not have an admissible monomorphism into any locally injective
locally contraherent cosheaf.
 The construction and proof follow the arguments of Sl\'avik
and \v St\!'ov\'\i\v cek.
 On the positive side, let $X$ be a Noetherian scheme of finite
Krull dimension.
 We prove that every $\bW$\+locally contraherent cosheaf on $X$ has
an admissible monomorphism into a locally cotorsion $\bW$\+locally
contraherent cosheaf.
 Moreover, the cokernel is a flat contraherent cosheaf.
 The proof is based on the theorem of Raynaud--Gruson about
the projective dimensions of flat modules and Enochs' classification
of flat cotorsion modules.
\end{abstract}

\maketitle

\tableofcontents

\section{Introduction}
\medskip

 Various classes of adjusted objects, such the classical flat,
projective, and injective modules, are used in homological algebra
in order to build resolutions and construct derived functors.
 There are always enough projective and injective objects in
the category of modules over an associative ring.
 The situation becomes more complicated when one passes to
algebraic geometry and considers the category of quasi-coherent
sheaves over a scheme.

 For any scheme $X$, the category of quasi-coherent sheaves $X\Qcoh$
is a Grothendieck abelian category~\cite[Section Tag~077K]{SP};
so it has all limits and enough injective objects.
 However, in any abelian category with infinite products and enough
projective objects, the direct product functors are exact.
 The infinite products of quasi-coherent sheaves are usually \emph{not}
exact functors~\cite[Example~4.9]{Kra}, \cite{Kan}; hence the category
$X\Qcoh$ cannot have enough projective objects (and in fact, usually
has no nonzero projective objects at all).  {\hbadness=1600\par}

 The conventional approach is to use flat quasi-coherent sheaves in
lieu of the nonexistent projective ones~\cite{M-th}.
 On a quasi-compact, semi-separated scheme $X$, any quasi-coherent sheaf
is a quotient sheaf of a flat quasi-coherent
sheaf~\cite[Section~2.4]{M-n}, \cite[Section~3.2]{M-th}.
 Another and perhaps clearer proof of this result can be found
in~\cite[Lemma~A.1]{EP}.
 Essentially the same argument as in~\cite{EP} can be used to prove
a stronger assertion that any quasi-coherent sheaf on $X$ is
a quotient sheaf of a \emph{very flat} quasi-coherent
sheaf~\cite[Lemma~4.1.1]{Pcosh}.

 On the other hand, let $X$ be a quasi-compact, quasi-separated scheme
that is \emph{not} semi-separated.
 Then the construction of Sl\'avik and
\v St\!'ov\'\i\v cek~\cite[Theorem~2.2]{SS} provides an example of
a quasi-coherent sheaf on $X$ that is \emph{not} a quotient sheaf of
any flat quasi-coherent sheaf.

 The \emph{contraherent cosheaves} are the dual analogues of
quasi-coherent sheaves~\cite[Preface]{Pcosh}, \cite{Pphil}.
 For any scheme $X$, there is the exact (but not abelian!) category
of contraherent cosheaves $X\Ctrh$.
 Choosing an open covering $\bW$ for the scheme $X$, one also obtains
a larger exact category of \emph{$\bW$\+locally contraherent cosheaves}
$X\Lcth_\bW$.
 The categories $X\Ctrh$ and $X\Lcth_\bW$ have exact functors of
infinite products and, under mild assumptions on the scheme $X$,
enough projective objects~\cite[Section~0.18, Lemma~4.5.1,
and Corollary~6.2.4]{Pcosh}.

 Dual-analogously to the categories of quasi-coherent sheaves not
having enough projective objects, one does not expect the categories
of contraherent or locally contraherent cosheaves to have enough
injective objects (though we are not aware of any specific
counterexamples).
 The \emph{locally injective} contraherent or locally contraherent
cosheaves are the closest dual analogues of flat (or very flat)
quasi-coherent sheaves.

 On a quasi-compact semi-separated scheme $X$, any $\bW$\+locally
contraherent cosheaf is an admissible subobject of a locally injective
$\bW$\+locally contraherent cosheaf~\cite[Lemma~4.3.2(a)]{Pcosh}.
 The proof of this result is dual-analogous to the proofs of
the existence of enough flat or very flat quasi-coherent sheaves
given in~\cite[Lemma~A.1]{EP} and~\cite[Lemma~4.1.1]{Pcosh}
(see Proposition~\ref{qcoh-ssep-enough-lin-prop} in the present paper).

 Now let $X$ be again a quasi-compact, quasi-separated scheme that is
not semi-separated.
 Then very close dual analogues of the construction and arguments of
Sl\'avik and \v St\!'ov\'\i\v cek~\cite[Section~2]{SS} provide
an example of a contraherent cosheaf $\fG$ on $X$ that is \emph{not}
an admissible subobject of any locally injective locally contraherent
cosheaf.
 This is the first main result of the present paper
(Theorem~\ref{non-ssep-not-enough-lin-theorem}).

 The (locally) contraherent cosheaves forming the exact categories
$X\Ctrh$ and $X\Lcth_\bW$ are glued from what we call
\emph{contraadjusted modules} over commutative
rings~\cite[Section~1.1]{Pcosh}, \cite{ST}, \cite{Pcta}; so such
cosheaves are called \emph{locally contraadjusted}.
 This is the widest reasonable class of modules to glue cosheaves from.
 The \emph{locally cotorsion} contraherent or locally contraherent
cosheaves, glued from \emph{cotorsion modules}~\cite{En2},
form intermediate exact categories between the exact categories of
locally contraadjusted and locally injective (locally) contraherent
cosheaves.

 Our contraherent cosheaf $\fG$ mentioned two paragraphs above is,
actually, a locally cotorsion contraherent cosheaf.
 So embedding locally cotorsion contraherent cosheaves into locally
injective ones appears to be no easier task than embedding locally
contraadjusted contraherent cosheaves into locally injective ones.

 Now let $X$ be a Noetherian scheme of finite Krull dimension.
 Notice that such a scheme is always quasi-compact and quasi-separated,
but \emph{need not} be semi-separated.
 Our second main result in this paper claims that any (i.~e., any
locally contraadjusted) $\bW$\+locally contraherent cosheaf on $X$ is
an admissible subobject of a locally cotorsion $\bW$\+locally
contraherent cosheaf.
 See Theorem~\ref{noetherian-findim-enough-lct-theorem}.

 The proof of the latter theorem is more involved.
 It is based on Enochs' classification of flat cotorsion modules over
Noetherian commutative rings~\cite[Section~2]{En2} and on
the theorem of Raynaud--Gruson~\cite[Corollaire~II.3.2.7]{RG},
according to which the projective dimension of a flat module over
a Noetherian commutative ring of Krull dimension~$D$ cannot
exceed~$D$.

 A proof of our second main theorem in a greater generality of
formal schemes can be found in~\cite[Proposition~4.7.11]{Pform}.
 However, \cite{Pform}~is a very long paper with lots of complicated
material (and indeed, formal schemes are harder to contemplate and
work with than schemes).
 For the benefit of the reader, we decided to include an independent
account of the second main result in this shorter paper.

 We hope that this paper will be also useful as a brief introduction
into the theory of contraherent cosheaves that may be more accessible
for some readers than the longer and heavier book-sized
manuscripts~\cite{Pcosh,Pdomc,Pform} (and at the same time contain
more details than the announcement~\cite{Pphil}).

\subsection*{Acknowledgement}
 The author is supported by the GA\v CR project 26-22734S
and the Institute of Mathematics, Czech Academy of Sciences
(research plan RVO:~67985840).

\Section{Category-Theoretic Preliminaries}

 We suggest the survey paper~\cite{Bueh} as the background reference
source on exact categories (in the sense of Quillen).
 The assumption of \emph{weak
idempotent-completeness}~\cite[Section~7]{Bueh} simplifies
the theory of exact categories.
 For the purposes of the present paper, one can safely assume all
the exact categories to satisfy the even stronger
\emph{idempotent-completeness} condition~\cite[Section~6]{Bueh}.

 Given an exact category $\sK$ and a full additive subcategory
$\sE\subset\sK$ closed under extensions in $\sK$, the category $\sE$
is endowed with the \emph{inherited} (or \emph{induced}) exact
category structure.
 This means that the admissible short exact sequences in $\sE$ are
the admissible short exact sequences in $\sK$ with the terms belonging
to~$\sE$.

 The Yoneda Ext groups $\Ext_\sE^n(X,Y)$, \,$X$, $Y\in\sE$,
\,$n\in\boZ_{\ge0}$ in an exact category $\sE$ can be defined in
the way similar to the classical construction for abelian categories.
 In particular, $\Ext_\sE^1(X,Y)$ is the group of equivalence
classes of admissible short exact sequences $0\rarrow Y\rarrow Z
\rarrow X\rarrow0$ in~$\sE$.

 A discussion of projective and injective objects in exact categories
can be found in~\cite[Section~11]{Bueh}.
 We denote the full subcategory of projective objects by
$\sE_\prj\subset\sE$ and the full subcategory of injective objects
by $\sE^\inj\subset\sE$.

\subsection{Resolving and coresolving subcategories}
\label{resolving-subsecn}
 Let $\sE$ be an exact category.
 A full subcategory $\sF\subset\sE$ is said to be \emph{generating}
in $\sE$ if for every object $E\in\sE$ there exists an object
$F\in\sF$ together with an admissible epimorphism $F\rarrow E$.
 Dually, a full subcategory $\sC\subset\sE$ is said to be
\emph{cogenerating} in $\sE$ if for every object $E\in\sE$ there
exists an object $C\in\sC$ together with an admissible monomorphism
$E\rarrow C$.

 A generating full subcategory $\sF\subset\sE$ is said to be
\emph{resolving} if $\sF$ is closed under extensions and kernels of
admissible epimorphisms in~$\sE$.
 Dually, a cogenerating full subcategory $\sC\subset\sE$ is said to be
\emph{coresolving} if $\sC$ is closed under extensions and cokernels
of admissible monomorphisms in~$\sE$.

 For example, the full subcategory of projective objects $\sE_\prj$
is resolving in $\sE$ if and only if it is generating in $\sE$, and
if and only if there are enough projectives in~$\sE$.
 Dually, the full subcategory of injective objects $\sE^\inj$ is
coresolving in $\sE$ if and only if it is cogenerating in $\sE$, and
if and only if there are enough injectives in~$\sE$.
 To give yet another example, for any associative ring $R$,
the full subcategory of flat $R$\+modules $R\Modl_\fl$ is resolving
in the abelian category of left $R$\+modules $R\Modl$.

 We say that an exact category is \emph{split} if all the (admissible)
short exact sequences in it are split.

\begin{lem} \label{split-co-resolving-lemma}
\textup{(a)} Let\/ $\sE$ be an exact category and\/ $\sF\subset\sE$
be a resolving subcategory closed under direct summands.
 Assume that the exact category\/ $\sF$ is split.
 Then there are enough projective objects in\/ $\sE$, and\/ $\sF$ is
the class of all such projective objects, $\sF=\sE_\prj$. \par
\textup{(b)} Let\/ $\sE$ be an exact category and\/ $\sC\subset\sE$
be a coresolving subcategory closed under direct summands.
 Assume that the exact category\/ $\sC$ is split.
 Then there are enough injective objects in\/ $\sE$, and\/ $\sC$ is
the class of all such injective objects, $\sC=\sE^\inj$.
\end{lem}

\begin{proof}
 Let us prove part~(b).
 Let $C\in\sC$ be an object; we need to show that the object $C$ is
injective in~$\sE$.
 Indeed, let $C\rarrow E$ be an admissible monomorphism in~$\sE$.
 As the full subcategory $\sC\subset\sE$ is cogenerating by assumption,
there exists an admissible monomorphism $E\rarrow D$ in $\sE$ with
$D\in\sC$.
 Now the composition $C\rarrow E\rarrow D$ is an admissible
monomorphism in $\sE$ between two objects from~$\sC$.
 As the cokernel of any such morphism $C\rarrow D$ belongs to $\sC$ by
assumption, it follows that $C\rarrow D$ is also an admissible
monomorphism in~$\sC$.
 Hence $C\rarrow D$ is a split monomorphism.
 Let $s\:D\rarrow C$ be its retraction; then the composition
$E\rarrow D\overset s\rarrow C$ is a retraction for the morphism
$C\rarrow E$.
 Thus $C\rarrow E$ is a split monomorphism as well, and we have
proved that $C\in\sE^\inj$.
 Conversely, let $J$ be an injective object in~$\sE$.
 Since the full subcategory $\sC\subset\sE$ is cogenerating,
there exists an admissible monomorphism $J\rarrow D$ in $\sE$
with $D\in\sC$.
 Then $J$ is a direct summand of $D$, and by assumption it follows
that $J\in\sC$.
 The assertion that there are enough injective objects in $\sE$ now
also follows from the assumption that the class $\sC$ is cogenerating
in~$\sE$.
\end{proof}

 Let $\sE$ be a weakly idempotent-complete exact category,
$\sF\subset\sE$ be a resolving subcategory, $E\in\sE$ be an object,
and $d\ge-1$ be an integer.
 One says that the \emph{$\sF$\+resolution dimension} of $E$ does not
exceed~$d$ if there exists an exact sequence $0\rarrow F_d\rarrow
F_{d-1}\rarrow\dotsb\rarrow F_0\rarrow E\rarrow0$ in $\sE$ with
the objects $F_i\in\sF$ for all $0\le i\le d$.
 Dually, given a coresolving subcategory $\sC\subset\sE$, one says that
the \emph{$\sC$\+coresolution dimension} of an object $E\in\sE$ does
not exceed~$d$ if there exists an exact sequence $0\rarrow E\rarrow
C^0\rarrow C^1\rarrow\dotsb\rarrow C^d\rarrow0$ in $\sE$ with
the objects $C^i\in\sC$ for all $0\le i\le d$.

 It is a classical observation in homological algebra that
the projective/injective dimension of an object does not depend on
the choice of a projective/injective (co)resolution.
 Similarly, the flat dimension of an $R$\+module does not depend on
the choice of a flat resolution.
 Such results, stated in the following proposition, are actually valid
for all (co)resolving subcategories in (weakly idempotent-complete)
exact categories.

\begin{prop} \label{co-resolution-dimension-prop}
\textup{(a)} Let\/ $\sE$ be a weakly idempotent-complete exact category
and\/ $\sF\subset\sE$ be a resolving subcategory.
 Let $E\in\sE$ be an object of\/ $\sF$\+resolution dimension\/~$\le d$
and\/ $0\rarrow F\rarrow F_{d-1}\rarrow\dotsb\rarrow F_0\rarrow E
\rarrow0$ be an exact sequence in\/ $\sE$ with the objects $F_i\in\sF$
for all\/ $0\le i\le d-1$.
 Then one has $F\in\sF$. \par
\textup{(b)} Let\/ $\sE$ be a weakly idempotent-complete exact category
and\/ $\sC\subset\sE$ be a coresolving subcategory.
 Let $E\in\sE$ be an object of\/ $\sC$\+coresolution dimension\/~$\le d$
and\/ $0\rarrow E\rarrow C^0\rarrow\dotsb\rarrow C^{d-1}\rarrow C
\rarrow0 $ be an exact sequence in\/ $\sE$ with the objects $C_i\in\sC$
for all\/ $0\le i\le d-1$.
 Then one has $C\in\sC$.
\end{prop}

\begin{proof}
 Various expositions on various generality levels can be found
in~\cite[Lemma~2.1]{Zhu}, \cite[Proposition~2.3(1)]{Sto},
and~\cite[Corollary~A.5.2]{Pcosh}.  \hbadness=1600
\end{proof}

\subsection{Cotorsion pairs}
 The notion of a \emph{cotorsion pair} (also known as a \emph{cotorsion
theory}) goes back to Salce~\cite{Sal}.
 The main result was obtained in the paper of Eklof and
Trlifaj~\cite{ET}.
 The book~\cite[Chapter~6]{GT} is the main reference source in
the module-theoretic context, while the papers~\cite[Sections~8
and~12]{Pcta} and~\cite[Introduction]{Pctrl} can be used for
introductory reading.
 Various category-theoretic expositions can be found in
the papers~\cite{SaoSt,Pal}, the lecture notes~\cite{Sto-ICRA},
the announcement~\cite{Pphil}, and the book
manuscript~\cite[Appendix~B]{Pcosh}.

 Let $\sE$ be an exact category and $\sF$, $\sC\subset\sE$ be two
classes of objects.
 One denotes by $\sF^{\perp_1}\subset\sE$ the class of all objects
$X\in\sE$ such that $\Ext^1_\sE(F,X)=0$ for all $F\in\sF$.
 Dually, the notation ${}^{\perp_1}\sC\subset\sE$ stands for the class
of all objects $Y\in\sE$ such that $\Ext^1_\sE(Y,C)=0$ for all
$C\in\sC$.

 A pair of classes of objects $(\sF,\sC)$ in $\sE$ is called
a \emph{cotorsion pair} if one has $\sC=\sF^{\perp_1}$ and
$\sF={}^{\perp_1}\sC$.
 For any class of objects $\sS\subset\sE$, the pair of classes
$\sC=\sS^{\perp_1}$ and $\sF={}^{\perp_1}\sC$ is a cotorsion pair
in~$\sE$.
 The latter cotorsion pair is said to be \emph{generated by}
the class of objects $\sS\subset\sE$.

 For example, for any exact category $\sE$, the pair of classes of
objects $(\sE_\prj,\sE)$ is a cotorsion pair in $\sE$, called
the \emph{projective cotorsion pair}.
 Dually, the pair of classes of objects $(\sE,\sE^\inj)$ is also
a cotorsion pair in $\sE$, called the \emph{injective cotorsion pair}.

 Notice that there is \emph{no} assumption of existence of enough
projective or injective objects in $\sE$ in the previous paragraph.
 So what does it mean, in application to a cotorsion pair, that
there are enough objects in the classes $\sF$ and~$\sC$\,?
 The rather nontrivial answer to this question is provided by
the following definition of a \emph{complete} cotorsion pair in
an exact category.

 A cotorsion pair $(\sF,\sC)$ in $\sE$ is said to be complete if,
for every object $E\in\sE$, there exist (admissible) short exact
sequences
\begin{alignat}{4}
 0&\lrarrow C'&&\lrarrow F&&\lrarrow E&&\lrarrow0,
\label{special-precover-sequence} \\
 0&\lrarrow E&&\lrarrow C&&\lrarrow F'&&\lrarrow0
\label{special-preenvelope-sequence}
\end{alignat}
with some objects $F$, $F'\in\sF$ and $C$, $C'\in\sC$.

 The short exact sequence~\eqref{special-precover-sequence} is called
a \emph{special precover sequence}.
 The short exact sequence~\eqref{special-preenvelope-sequence} is
called a \emph{special preenvelope sequence}.
 Collectively, the sequences~(\ref{special-precover-sequence}\+-%
\ref{special-preenvelope-sequence}) are referred to as
the \emph{approximation sequences}.

 The following assertion is a suitably generalized version of
the result known as the \emph{Salce lemma}~\cite{Sal},
\cite[Lemma~5.20]{GT}.

\begin{lem} \label{salce-lemma}
 Let\/ $\sE$ be an exact category and\/ $\sF$, $\sC\subset\sE$ be
a pair of classes of objects. \par
\textup{(a)} Assume that a short exact
sequence~\eqref{special-preenvelope-sequence} with objects
$C\in\sC$ and $F'\in\sF$ exists for every object $E\in\sE$.
 Assume further that the class\/ $\sF$ is generating and closed
under extensions in\/~$\sE$.
 Then a short exact sequence~\eqref{special-precover-sequence}
with objects $F\in\sF$ and $C'\in\sC$ exists for every object
$E\in\sE$. \par
\textup{(b)} Assume that a short exact
sequence~\eqref{special-precover-sequence} with objects
$F\in\sF$ and $C'\in\sC$ exists for every object $E\in\sE$.
 Assume further that the class\/ $\sC$ is cogenerating and closed
under extensions in\/~$\sE$.
 Then a short exact sequence~\eqref{special-preenvelope-sequence}
with objects $C\in\sC$ and $F'\in\sF$ exists for every object
$E\in\sE$.
\end{lem}

\begin{proof}
 See, e.~g., \cite[Lemma~B.1.1]{Pcosh}.
\end{proof}

 Given a cotorsion pair $(\sF,\sC)$ in $\sE$, the intersection of
the two classes $\sF\cap\sC$ is called the \emph{kernel} (or in
a different terminology, the \emph{core}) of the cotorsion pair
$(\sF,\sC)$.
 The following lemma is well-known.

\begin{lem} \label{cotorsion-pair-injectives-projectives}
 Let\/ $\sE$ be an exact category and $(\sF,\sC)$ be a complete
cotorsion pair in\/~$\sE$.
 Then \par
\textup{(a)} there are enough injective objects in the exact category\/
$\sF$, and the full subcategory\/ $\sF^\inj\subset\sF$ coincides with
the full subcategory\/ $\sF\cap\sC\subset\sF$; \par
\textup{(b)} there are enough projective objects in the exact category\/
$\sC$, and the full subcategory\/ $\sC_\prj\subset\sC$ coincides with
the full subcategory\/ $\sF\cap\sC\subset\sC$.
\end{lem}

\begin{proof}
 See, e.~g., \cite[Lemma~1.9]{Pfltp} or~\cite[Lemma~A.2.4]{Pform}.
\end{proof}

\Section{Module-Theoretic Preliminaries}

 Given an associative ring $R$, we denote by $R\Modl$ the abelian
category of left $R$\+modules.
 Accordingly, $R\Modl_\prj$ and $R\Modl^\inj\subset R\Modl$ are the full
subcategories of projective and injective $R$\+modules, respectively.
 The notation $R\Modl_\fl\subset R\Modl$ for the full subcategory
of flat $R$\+modules was already mentioned in
Section~\ref{resolving-subsecn}.

\subsection{Cotorsion modules} \label{cotorsion-modules-subsecn}
 A left $R$\+module $F$ is said to be \emph{cotorsion} (in the sense
of Enochs~\cite{En2}) if one has $\Ext_R^1(F,C)=0$ for all flat
left $R$\+modules~$F$.
 This is equivalent to having $\Ext_R^n(F,C)=0$ for all flat left
$R$\+modules $F$ and all integers $n\ge1$ (as one can easily see
using the facts that projective modules are flat and the kernels of
subjective morphisms of flat modules are flat).

 We denote the full subcategory of cotorsion $R$\+modules by
$R\Modl^\cot\subset R\Modl$.
 The full subcategory $R\Modl^\cot$ is closed under extensions,
cokernels of monomorphisms, and infinite products in $R\Modl$,
while the full subcategory $R\Modl_\fl$ is well known to be closed
under extensions, kernels of epimorphisms, and direct limits in
$R\Modl$.
 So both the full subcategories $R\Modl^\cot$ and $R\Modl_\fl$
inherit exact category structures from the abelian exact structure
of $R\Modl$.

\begin{thm} \label{flat-cotorsion-pair}
 For any associative ring $R$, the pair of classes (flat $R$\+modules,
cotorsion $R$\+modules) is a complete cotorsion pair in $R\Modl$.
 In other words: \par
\textup{(a)} for any $R$\+module $M$ there exists a short exact
sequence of $R$\+modules\/ $0\rarrow C\rarrow F\rarrow M\rarrow0$
with a flat $R$\+module $F$ and a cotorsion $R$\+module~$C$; \par
\textup{(b)} for any $R$\+module $M$ there exists a short exact
sequence of $R$\+modules\/ $0\rarrow M\rarrow C\rarrow F\rarrow0$
with a cotorsion $R$\+module $C$ and a flat $R$\+module~$F$.
\end{thm}

\begin{proof}
 The assertion that $(R\Modl_\fl,\allowbreak\>R\Modl^\cot)$ is
\emph{a cotorsion pair} already requires a proof.
 One has $R\Modl^\cot=(R\Modl_\fl)^{\perp_1}$ by the definition; but
why does the equality $R\Modl_\fl={}^{\perp_1}(R\Modl^\cot)$ hold?
 A proof of that can be found in~\cite[Lemma~3.4.1]{Xu}.
 \emph{Completeness} of this cotorsion pair is harder to prove;
the argument, based on~\cite[Theorems~2 and~10]{ET}, appeared
in the paper~\cite[Section~2]{BBE}.
\end{proof}

 The cotorsion pair $(R\Modl_\fl,\allowbreak\>R\Modl^\cot)$ is called
the \emph{flat cotorsion pair} in $R\Modl$.

\begin{lem} \label{flat-cotorsion-injective-tensor-Hom-lemma}
 Let $R$ be a commutative ring. \par
\textup{(a)} For any two flat $R$\+modules $F$ and $G$,
the $R$\+module $F\ot_RG$ is flat. \par
\textup{(b)} For any flat $R$\+module $F$ and any cotorsion
$R$\+module $C$, the $R$\+module\/ $\Hom_R(F,C)$ is cotorsion. \par
\textup{(c)} For any $R$\+module $M$ and any injective
$R$\+module $J$, the $R$\+module\/ $\Hom_R(M,J)$ is cotorsion. \par
\textup{(d)} For any flat $R$\+module $F$ and any injective
$R$\+module $J$, the $R$\+module\/ $\Hom_R(F,J)$ is injective.
\end{lem}

\begin{proof}
 All the assertions are well known.
 Part~(a) is easy.
 Part~(c) can be found in~\cite[Lemma~2.1]{En2}.
 Parts~(b\+-d) can be found in~\cite[Lemma~1.3.2]{Pcosh};
see also the more general~\cite[Lemma~1.5]{Pdomc} with its
slightly different proof.
\end{proof}

 The next lemma is even more well-known and easy.

\begin{lem} \label{injective-change-of-scalars-lemma}
 Let $R\rarrow S$ be a homomorphism of commutative rings. \par
\textup{(a)} If $S$ is a flat $R$\+module, then any injective
$S$\+module is injective as an $R$\+module. \par
\textup{(b)} For any injective $R$\+module $J$, the $S$\+module\/
$\Hom_R(S,J)$ is injective. \qed
\end{lem}

\begin{lem} \label{cotorsion-change-of-scalars-lemma}
 Let $R\rarrow S$ be a homomorphism of commutative rings. \par
\textup{(a)} Any cotorsion $S$\+module is cotorsion as an $R$\+module.
\par
\textup{(b)} If $S$ is a flat $R$\+module, then for any cotorsion
$R$\+module $C$, the $S$\+module\/ $\Hom_R(S,C)$ is cotorsion.
\end{lem}

\begin{proof}
 This is a particular case of~\cite[Lemmas~1.3.4(a) and~1.3.5(a)]{Pcosh}
or~\cite[Lemma~1.6(a,c)]{Pdomc}.
\end{proof}

\subsection{Contraadjusted modules} \label{cta-modules-subsecn}
 Now let $R$ be a commutative ring.
 For any element $s\in R$, we denote by $R[s^{-1}]$ the localization
of the ring $R$ at the element~$s$, i.~e., the ring obtained from $R$
by formally inverting~$s$.
 So, in other words, $R[s^{-1}]=S^{-1}R$, where $S$ is
the multiplicative subset $S=\{1,s,s^2,s^3,\dotsc\}\subset R$.

 An $R$\+module $C$ is said to be
\emph{contraadjusted}~\cite[Section~1.1]{Pcosh}, \cite[Sections~2
and~5]{ST}, \cite[Sections~2 and~8]{Pcta}, \cite[Section~4.3]{Pphil},
\cite[Section~2]{Pal} if $\Ext^1_R(R[s^{-1}],C)=0$ for all
elements $s\in R$.
 It is helpful to keep in mind that the projective dimension of
the $R$\+module $R[s^{-1}]$ never exceeds~$1$ \,\cite[proof of
Lemma~2.2]{ST}, \cite[proof of Lemma~2.1]{Pcta}; that is why no
higher $\Ext$ vanishing condition is needed in this definition.
 Any quotient $R$\+module of a contraadjusted $R$\+module is
contraadjusted.

 An $R$\+module $F$ is said to be \emph{very
flat}~\cite[Section~1.1]{Pcosh}, \cite[Section~2]{ST},
\cite[Section~0.5]{PSl1}, \cite[Section~4.3]{Pphil}
if $\Ext^1_R(F,C)=0$ for all contraadjusted $R$\+modules~$C$.
 The projective dimensions of very flat $R$\+modules do not
exceed~$1$ \,\cite[Section~1.1]{Pcosh}, \cite[Lemma~2.2]{ST},
\cite[Theorem~4.5(d)]{Pphil}.

\begin{ex} \label{open-affine-very-flat-example}
 Let $U$ be an affine scheme and $V\subset U$ be an affine open
subscheme.
 Then the $\cO(U)$\+module $\cO(V)$ is very
flat~\cite[Lemma~1.2.4]{Pcosh}.
 The argument uses the finite \v Cech exact sequence for a finite
covering of the affine open subscheme $V\subset U$ by principal affine
open subschemes of $U$, and the fact that the kernels of surjective
morphisms of very flat $\cO(U)$\+modules are very flat (as
mentioned below).

 This is a special case of a much more general (and much more
difficult) result of~\cite[Main Theorem~1.1]{PSl1}
(see also a discussion in~\cite[Remark~5.1]{Pphil}).
\end{ex}

 We denote the full subcategory of contraadjusted $R$\+modules by
$R\Modl^\cta\subset R\Modl$ and the full subcategory of very flat
$R$\+modules by $R\Modl_\vfl\subset R\Modl$.
 The full subcategory $R\Modl^\cta$ is closed under extensions,
quotients, and infinite products, while the full subcategory
$R\Modl_\vfl$ is closed under extensions, kernels of epimorphisms,
and infinite direct sums in $R\Modl$.
 So both the full subcategories $R\Modl^\cta$ and $R\Modl_\vfl$
inherit exact category structures from the abelian exact structure
of $R\Modl$.

\begin{thm} \label{very-flat-cotorsion-pair}
 For any commutative ring $R$, the pair of classes (very flat
$R$\+modules, contraadjusted $R$\+modules) is a complete cotorsion
pair in $R\Modl$.
 In other words: \par
\textup{(a)} for any $R$\+module $M$ there exists a short exact
sequence of $R$\+modules\/ $0\rarrow C\rarrow F\rarrow M\rarrow0$ with
a very flat $R$\+module $F$ and a contraadjusted $R$\+module~$C$; \par
\textup{(b)} for any $R$\+module $M$ there exists a short exact
sequence of $R$\+modules\/ $0\rarrow M\rarrow C\rarrow F\rarrow0$ with
a contraadjusted $R$\+module $C$ and a very flat $R$\+module~$F$.
\end{thm}

\begin{proof}
 This is a special case of the Eklof--Trlifaj
theorem~\cite[Theorems~2 and~10]{ET}, \cite[Theorem~6.11]{GT}.
 Some details specific to the situation at hand can be found
in~\cite[Theorem~1.1.1]{Pcosh}, \cite{ST}, \cite[Theorem~4.5]{Pphil}.
\end{proof}

 The word ``contraadjusted'' means ``adjusted to contraherent
cosheaves''.
 The cotorsion pair $(R\Modl_\vfl,\allowbreak\>R\Modl^\cta)$ is called
the \emph{very flat cotorsion pair} in $R\Modl$.

 The following lemma is the very flat/contraadjusted version of
Lemma~\ref{flat-cotorsion-injective-tensor-Hom-lemma}.

\begin{lem} \label{vfl-cta-tensor-Hom-lemma}
 Let $R$ be a commutative ring. \par
\textup{(a)} For any two very flat $R$\+modules $F$ and $G$,
the $R$\+module $F\ot_RG$ is very flat. \par
\textup{(b)} For any very flat $R$\+module $F$ and any contraadjusted
$R$\+module $C$, the $R$\+module\/ $\Hom_R(F,C)$ is contraadjusted.
\end{lem}

\begin{proof}
 This is~\cite[Lemma~1.2.1]{Pcosh}.
\end{proof}

\begin{lem} \label{very-flat-change-of-scalars-lemma}
 Let $R\rarrow S$ be a homomorphism of commutative rings. \par
\textup{(a)} If the $R$\+module $S[s^{-1}]$ is very flat for every
element $s\in S$, then any very flat $S$\+module is very flat as
an $R$\+module. \par
\textup{(b)} For any very flat $R$\+module $F$, the $S$\+module
$S\ot_RF$ is very flat.
\end{lem}

\begin{proof}
 This is~\cite[Lemmas~1.2.3(b) and~1.2.2(b)]{Pcosh}.
\end{proof}

 The next lemma is the contraadjusted version of
Lemma~\ref{cotorsion-change-of-scalars-lemma}.

\begin{lem} \label{contraadjusted-change-of-scalars-lemma}
 Let $R\rarrow S$ be a homomorphism of commutative rings. \par
\textup{(a)} Any contraadjusted $S$\+module is also contraadjusted
as an $R$\+module. \par
\textup{(b)} If the $R$\+module $S[s^{-1}]$ is very flat for every
element $s\in S$, then the $S$\+module\/ $\Hom_R(S,C)$ is
contraadjusted for any contraadjusted $R$\+module~$C$.
\end{lem}

\begin{proof}
 This is~\cite[Lemmas~1.2.2(a) and~1.2.3(a)]{Pcosh}.
\end{proof}

\begin{lem} \label{Hom-into-pushout-lemma}
 Let $R$ be an associative ring and
\begin{equation} \label{pushout-diagram-of-modules}
\begin{gathered}
 \xymatrix{
  0 \ar[r] & K \ar[r] \ar[d] & L \ar[r] \ar[d] & M \ar[r] & 0 \\
  0 \ar[r] & K' \ar[r] & L' \ar[ur]
 }
\end{gathered}
\end{equation}
be a pushout diagram of short exact sequences of left $R$\+modules
(so the upper line is a short exact sequence and the square on
the left-hand side is a pushout square).
 In this context: \par
\textup{(a)} If $K$ is a cotorsion $R$\+module, and $F$ is a flat
left $R$\+module, then applying the functor\/ $\Hom_R(F,{-})$ to
the diagram~\eqref{pushout-diagram-of-modules} produces
a pushout diagram of abelian groups. \par
\textup{(b)} If the ring $R$ is commutative, $K$ is a contraadjusted
$R$\+module, and $F$ is a very flat $R$\+module, then applying
the functor\/ $\Hom_R(F,{-})$ to
the diagram~\eqref{pushout-diagram-of-modules} produces
a pushout diagram of $R$\+modules.
\end{lem}

\begin{proof}
 In both parts~(a) and~(b), the point is that the short exact
sequence $0\rarrow K\rarrow L\oplus K'\rarrow L'\rarrow0$ remains
exact after the functor $\Hom_R(F,{-})$ is applied, due to
the adjustedness assumptions imposed on $K$ and~$F$.
 See, e.~g., \cite[Proposition~2.12]{Bueh} for a background
discussion.
\end{proof}

\subsection{Locality lemmas}
 Throughout this section, we consider a commutative ring $R$ and
a finite collection of homomorphisms of commutative rings
$R\rarrow S_\alpha$, \,$1\le\alpha\le N$, such that the related
finite collection of morphisms of affine schemes $\Spec S_\alpha
\rarrow\Spec R$ is an open covering of $\Spec R$.
 So, in particular, for every index~$\alpha$, the morphism
$\Spec S_\alpha\rarrow\Spec R$ is an open immersion of schemes.

\begin{lem} \label{very-flatness-local}
\textup{(a)} An $R$\+module $F$ is flat if and only if
the $S_\alpha$\+module $S_\alpha\ot_RF$ is flat for every
index\/~$\alpha$. \par
\textup{(b)} An $R$\+module $F$ is very flat if and only if
the $S_\alpha$\+module $S_\alpha\ot_RF$ is very flat for every
index\/~$\alpha$.
\end{lem}

\begin{proof}
 Part~(a) is well known, while part~(b) is~\cite[Lemma~1.2.6(a)]{Pcosh}
or~\cite[Example~2.5]{Pal}.
 The ``only if'' assertion of part~(b) holds by
Lemma~\ref{very-flat-change-of-scalars-lemma}(b).
 The ``if'' implications of both parts~(a) and~(b) are provable using
the finite \v Cech coresolution
\begin{multline} \label{module-cech-coresolution}
 0\lrarrow F\lrarrow\bigoplus\nolimits_{\alpha=1}^N
 S_\alpha\ot_RF\lrarrow\bigoplus\nolimits_{1\le\alpha<\beta\le N}
 S_\alpha\ot_RS_\beta\ot_RF \\
 \lrarrow\dotsb\lrarrow S_1\ot_R\dotsb\ot_RS_N\ot_RF\lrarrow0,
\end{multline}
which is an exact sequence of $R$\+modules for any $R$\+module~$F$.
 To prove the ``if'' implication in~(b), one needs to
use~\eqref{module-cech-coresolution} together with
Example~\ref{open-affine-very-flat-example},
Lemma~\ref{very-flat-change-of-scalars-lemma}(a\+-b), and
the fact that the class of very flat $R$\+modules is closed
under kernels of epimorphisms.
 One proceeds by induction, moving from the rightmost to
the leftmost end of the sequence~\eqref{module-cech-coresolution}
and proving that its $R$\+modules of cocycles are very flat.
\end{proof}

\begin{lem} \label{contraadjusted-cech-resolution}
 Let $C$ be a contraadjusted $R$\+module.
 Then the finite \v Cech resolution
\begin{multline} \label{cta-module-cech-resolution}
 0\lrarrow\Hom_R(S_1\ot_R\dotsb\ot_R S_N,\>C)\lrarrow\dotsb \\
 \lrarrow\bigoplus\nolimits_{1\le\alpha<\beta\le N}
 \Hom_R(S_\alpha\ot_RS_\beta,\>C)\lrarrow
 \bigoplus\nolimits_{\alpha=1}^N\Hom_R(S_\alpha,C)
 \lrarrow C\lrarrow 0
\end{multline}
is an exact sequence of $R$\+modules (and in fact, an exact
sequence in the exact category of contraadjusted $R$\+modules
$R\Modl^\cta$).
\end{lem}

\begin{proof}
 This is~\cite[Lemma~1.2.6(b)]{Pcosh}.
 Set $F=R$; then the sequence~\eqref{module-cech-coresolution}
is exact in the exact category $R\Modl_\vfl$ (as per the proof of
Lemma~\ref{very-flatness-local} above).
 Hence, applying the functor $\Hom_R({-},C)$
to~\eqref{module-cech-coresolution} produces an exact sequence
in $R\Modl^\cta$ in view of Lemma~\ref{vfl-cta-tensor-Hom-lemma}(b).
 This is the desired exact sequence~\eqref{cta-module-cech-resolution}.
\end{proof}

 The following lemma is a dual-analogous version of
Lemma~\ref{very-flatness-local}.

\begin{lem} \label{cotorsion-injective-colocal-if-contraadjusted}
 Let $C$ be a contraadjusted $R$\+module. \par
\textup{(a)} The $R$\+module $C$ is cotorsion if and only if
the $S_\alpha$\+module\/ $\Hom_R(S_\alpha,C)$ is cotorsion
for every index\/~$\alpha$. \par
\textup{(b)} The $R$\+module $C$ is injective if and only if
the $S_\alpha$\+module\/ $\Hom_R(S_\alpha,C)$ is injective
for every index\/~$\alpha$.
\end{lem}

\begin{proof}
 Part~(a) is~\cite[Lemma~1.3.6(a)]{Pcosh} or~\cite[Example~3.8]{Pal},
while part~(b) is~\cite[Lemma~1.3.6(b)]{Pcosh}
or~\cite[Example~3.7]{Pal}.
 The ``only if'' implication in part~(a) holds by
Lemma~\ref{cotorsion-change-of-scalars-lemma}(b),
and the ``only if'' implication in part~(b) holds by
Lemma~\ref{injective-change-of-scalars-lemma}(b).

 The ``if'' implications in both parts~(a) and~(b) are provable
using Lemma~\ref{contraadjusted-cech-resolution}.
 Let us sketch the proof of the ``if'' implication in part~(a).
 Using Lemma~\ref{cotorsion-change-of-scalars-lemma}(a\+-b),
one shows that all the terms of the exact
sequence~\eqref{cta-module-cech-resolution}, except perhaps
the rightmost one, are cotorsion $R$\+modules.
 Then one proceeds by induction, moving from the leftmost to
the rightmost end of the sequence~\eqref{cta-module-cech-resolution}
and proving that its $R$\+modules of cycles are cotorsion.
 The fact that the class of cotorsion $R$\+modules is closed under
cokernels of monomorphisms in $R\Modl$ (see
Section~\ref{cotorsion-modules-subsecn}) plays a key role in
this argument.
\end{proof}

\begin{lem} \label{short-exact-sequnces-of-cta-are-colocal}
 Let $C\rarrow D\rarrow E$ be a composable pair of homomorphisms of
contraadjusted $R$\+modules.
 Then\/ $0\rarrow C\rarrow D\rarrow E\rarrow0$ is a short exact
sequence of $R$\+modules if and only if\/ $0\rarrow\Hom_R(S_\alpha,C)
\rarrow\Hom_R(S_\alpha,D)\rarrow\Hom_R(S_\alpha,E)\rarrow0$ is a short
exact sequence of $S_\alpha$\+modules for every index\/~$\alpha$.
\end{lem}

\begin{proof}
 This is~\cite[Lemma~1.4.1(a)]{Pcosh}.
 In particular, the easy ``only if'' assertion follows from
Example~\ref{open-affine-very-flat-example} above.
 (See also~\cite[Lemma~2.17.10(a)]{Pform}.)
\end{proof}

\begin{lem} \label{colocality-of-epimorphisms}
\textup{(a)} Let $D\rarrow E$ be a homomorphism of contraadjusted
$R$\+modules.
 Then $D\rarrow E$ is an admissible epimorphism in $R\Modl^\cta$ if
and only if\/ $\Hom_R(S_\alpha,D)\rarrow\Hom_R(S_\alpha,E)$ is
an admissible epimorphism in $S_\alpha\Modl^\cta$ for every
index\/~$\alpha$. {\hbadness=1350\par}
\textup{(b)} Let $D\rarrow E$ be a homomorphism of cotorsion
$R$\+modules.
 Then $D\rarrow E$ is an admissible epimorphism in $R\Modl^\cot$ if
and only if\/ $\Hom_R(S_\alpha,D)\rarrow\Hom_R(S_\alpha,E)$ is
an admissible epimorphism in $S_\alpha\Modl^\cot$ for every
index\/~$\alpha$. \par
\textup{(c)} Let $D\rarrow E$ be a homomorphism of injective
$R$\+modules.
 Then $D\rarrow E$ is a split epimorphism if and only if\/
$\Hom_R(S_\alpha,D)\rarrow\Hom_R(S_\alpha,E)$ is a split
epimorphism of (injective) $S_\alpha$\+modules for every
index\/~$\alpha$.
\end{lem}

\begin{proof}
 This is~\cite[Lemmas~1.4.1(b), 1.4.2(b), and~1.4.3(b)]{Pcosh}
(see also~\cite[Lemmas~2.17.10(b) and~2.17.11(b)]{Pform}).
\end{proof}

 The analogues of the assertions of
Lemma~\ref{colocality-of-epimorphisms} for admissible/split
monomorphisms instead of epimorphisms are \emph{not} true.
 See~\cite[Example~3.2.1]{Pcosh}.

\subsection{Noetherian commutative rings}
\label{noetherian-commutative-rings-subsecn}
 The aim of this section is to discuss flat contraadjusted and
flat cotorsion modules over Noetherian commutative rings, but some
of the definitions and results are applicable in greater generality.

\begin{prop} \label{flat-contraadjusted-colocalization}
 Let $R\rarrow S$ be a homomorphism of commutative rings such that
the induced morphism of affine schemes\/ $\Spec S\rarrow\Spec R$
is an open immersion.
 Assume that the ring $R$ is coherent.
 Let $G$ be a flat contraadjusted $R$\+module.
 Then\/ $\Hom_R(S,G)$ is a flat contraadjusted $S$\+module.
\end{prop}

\begin{proof}
 This is~\cite[Corollary~1.7.6(a)]{Pcosh}.
\end{proof}

 Let $R$ be a commutative ring and $I\subset R$ be a finitely
generated ideal.
 An $R$\+module $C$ is said to be an \emph{$I$\+contramodule}
(or an \emph{$I$\+contramodule $R$\+module})~\cite[Sections~2
and~9]{Pcta} if $\Hom_R(R[s^{-1}],C)=0=\Ext^1_R(R[s^{-1}],C)$
for all $s\in I$.
 It suffices to check this condition for any chosen set of
generators~$s_j$ of the ideal~$I$ \,\cite[Theorem~5.1]{Pcta}.

 The full subcategory of $I$\+contramodule $R$\+modules
$R\Modl_{I\ctra}\subset R\Modl$ is closed under kernels, cokernels,
extensions, and infinite products in $R\Modl$
\,\cite[Proposition~1.1]{GL}, \cite[Theorem~1.2(a)]{Pcta}.
 Therefore, $R\Modl_{I\ctra}$ is an abelian category.
 The inclusion functor $R\Modl_{I\ctra}\rarrow R\Modl$ is exact
and preserves infinite products.

 We refer to the papers~\cite[Section~5.5]{PSl1}
and~\cite[Section~1]{Pdc} for a discussion of the full subcategory
of \emph{quotseparated} $I$\+contramodule $R$\+modules
$R\Modl_{I\ctra}^\qs\subset R\Modl_{I\ctra}$.
 The full subcategory $R\Modl_{I\ctra}^\qs$ is closed under subobjects,
quotient objects, and infinite products in $R\Modl_{I\ctra}$, but
\emph{not} under extensions.
 So $R\Modl_{I\ctra}^\qs$ is also an abelian category with an exact
inclusion functor $R\Modl_{I\ctra}^\qs\rarrow R\Modl$.

 The category of quotseparated $I$\+contramodule $R$\+modules
$R\Modl_{I\ctra}^\qs$ is naturally equivalent (in fact, isomorphic)
to the category of contramodules over a certain commutative topological
ring~$\widehat R$.
 Specifically, $\widehat R=\varprojlim_{n\ge1}R/I^n$ is the $I$\+adic
completion of the ring $R$, endowed with the $I$\+adic
topology~\cite[Proposition~1.5]{Pdc}.
 This is one reason why quotseparated $I$\+contramodules are often
more convenient to work with than arbitrary $I$\+contramodules.

 Both the abelian categories $R\Modl_{I\ctra}$ and $R\Modl_{I\ctra}^\qs$
are locally $\aleph_1$\+presenable (in the sense
of~\cite[Definition~1.17 and Theorem~1.20]{AR}) and have enough
projective objects~\cite[Section~1]{Pdc}.
 In particular, the projective objects of $R\Modl_{I\ctra}^\qs$ are
the direct summands of the \emph{free quotseparated $I$\+contramodule
$R$\+modules} $\widehat R[[X]]$.
 The latter are just the $I$\+adic completions of the free
$R$\+modules $R[X]=R^{(X)}$; so $\widehat R[[X]]=
\varprojlim_{n\ge1}(R/I)^{(X)}$.
 Here $X$ is an arbitrary set, and $A[X]=A^{(X)}$ is the direct sum
of $X$ copies of an abelian group~$A$.

 When the finitely generated ideal $I\subset R$ is \emph{weakly
proregular} in the sense of~\cite[Section~4]{PSY}, all $I$\+contramodule
$R$\+modules are quotseparated~\cite[Corollary~3.7 and Remark~3.8]{Pdc}.
 In particular, in a Noetherian commutative ring $R$, all finitely
generated ideals are weakly proregular~\cite[Theorem~4.34]{PSY}.

 As we are only interested in contramodules over Noetherian
commutative rings in this paper, we will usually omit the adjective
``quotseparated'' and the superindex $\qs$ from our terminology
and notation.
 In fact, we are only interested in $\m$\+contramodule $R$\+modules
for local Noetherian commutative rings $R$ with the maximal ideal~$\m$;
in this context, the classes of projective and free $\m$\+contramodule
$R$\+modules coincide~\cite[Corollary~10.7]{Pcta}.

\begin{lem} \label{free-contramodules-are-flat-cotorsion}
 Let $R$ be a Noetherian commutative ring and\/ $\m\subset R$ be
a maximal ideal.
 Then \par
\textup{(a)} all\/ $\m$\+contramodule $R$\+modules are cotorsion
$R$\+modules; \par
\textup{(b)} all free/projective\/ $\m$\+contramodule $R$\+modules
are flat $R$\+modules (moreover, an\/ $\m$\+contramodule $R$\+module
is free/projective if and only if it is flat as an $R$\+module).
\end{lem}

\begin{proof}
 Part~(a) is~\cite[Theorem~9.3]{Pcta}.
 Part~(b) is a special case of~\cite[Theorem~10.5]{Pcta}.
 A further discussion can be found in~\cite[Proposition~1.3.7]{Pcosh}
or~\cite[Corollary~2.19.3]{Pform}.
\end{proof}

 Given a commutative ring $R$ and a prime ideal $\p\subset R$, we
denote by $R_\p=(R\setminus\p)^{-1}R$ the localization of $R$ at~$\p$.
 So $R_\p$ is a local Noetherian commutative ring with the maximal
ideal $R_\p\p\subset R_\p$.
 The following classification of flat cotorsion modules over
Noetherian commutative rings is due to Enochs~\cite{En2}.

\begin{prop} \label{flat-cotorsion-over-Noetherian-classified}
 Let $R$ be a Noetherian commutative ring.
 Then an $R$\+module $C$ is flat and cotorsion if and only if $C$ has
the form of an infinite product $C\simeq\prod_{\p\in\Spec R}F_\p$,
where $F_\p$ is a free\/ $(R_\p\p)$\+contramodule $R_\p$\+module for
every prime ideal\/ $\p\subset R$.
\end{prop}

\begin{proof}
 This is~\cite[Theorem in Section~2]{En2}.
 We refer to~\cite[Theorem~1.3.8]{Pcosh}
or~\cite[Proposition~2.19.5]{Pform} for a further discussion.
\end{proof}

 Let $R$ be an associative ring.
 The \emph{cotorsion dimension} of a left $R$\+module $M$ is
defined as the coresolution dimension of the object $M\in R\Modl$
with respect to the coresolving subcategory $R\Modl^\cot\subset
R\Modl$.
 Simply put, the cotorsion dimension of $M$ is the minimal length of
a coresolution of $M$ by cotorsion left $R$\+modules.

\begin{thm} \label{raynaud-gruson-theorem}
 Let $R$ be a Noetherian commutative ring of finite Krull dimension~$D$.
 Then \par
\textup{(a)} the projective dimensions of flat $R$\+modules do not
exceed~$D$; \par
\textup{(b)} the cotorsion dimensions of $R$\+modules do not exceed~$D$.
\end{thm}

\begin{proof}
 Parts~(a) and~(b) are easily seen to be equivalent restatements of
each other~\cite[Lemma~B.1.9]{Pcosh}.
 In fact, for any associative ring $R$, both the supremum of
the projective dimensions of flat left $R$\+modules and the supremum
of the cotorsion dimensions of all left $R$\+modules are equal
to the supremum of all integers~$n\ge0$ for which there exist
a flat left $R$\+module $F$ and a left $R$\+module $M$ with
$\Ext^n_R(F,M)\ne0$.

 Part~(a) is a celebrated result of Raynaud and
Gruson~\cite[Corollaire~II.3.2.7]{RG}.
 According to the previous paragraph, part~(b) follows from part~(a).
 See also~\cite[Section~1.5]{Pcosh} or~\cite[Proposition~2.13.8]{Pform}.
\end{proof}

\Section{Locally Contraherent Cosheaves}

\subsection{Copresheaves and cosheaves}
 Let $(X,\cO_X)$ be a ringed space.
 We skip the well-known definitions of presheaves and sheaves of
$\cO_X$\+modules on~$X$.
 The notation $(X,\cO_X)\Sh$ stands for the Grothendieck abelian
category of sheaves of $\cO_X$\+modules on~$X$.

 Let $\bB$ be a set of open subsets forming a base of the topology
of~$X$.
 We will view $\bB$ as a category: the objects are the open subsets
of $X$ belonging to $\bB$, and the morphisms are the identity
inclusions.

 By a \emph{presheaf of abelian groups on\/~$\bB$} one means
a contravariant functor $\bB^\sop\rarrow\Ab$ (where $\Ab=\boZ\Modl$
is the category of abelian groups).
 Given two open subsets $V\subset U\subset X$, \ $U$, $V\in\bB$,
and a presheaf $\cM$ on $\bB$, the homomorphism of abelian groups
$\cM(U)\rarrow\cM(V)$ assigned to the identity inclusion $V\rarrow U$
by the presheaf $\cM$ is called the \emph{restriction map}.

 A \emph{presheaf of $\cO_X$\+modules $\cM$ on\/~$\bB$} is a presheaf
of abelian groups such that, for every open subset $U\subset X$,
\ $U\in\bB$, the abelian group $\cM(U)$ is endowed with
an $\cO_X(U)$\+module structure, and for every pair of open subsets
$V\subset U\subset X$, \ $U$, $V\in\bB$, the restriction map
$\cM(U)\rarrow\cM(V)$ is a homomorphism of $\cO_X(U)$\+modules.
 Here the $\cO_X(V)$\+module $\cM(V)$ is endowed with
an $\cO_X(U)$\+module structure using the restriction homomorphism
of rings $\cO_X(U)\rarrow\cO_X(V)$.

 A presheaf of abelian groups $\cM$ on $\bB$ is called a \emph{sheaf
of abelian groups} (\emph{on\/~$\bB$}) if the following \emph{sheaf
axiom} is satisfied.
 Let $U\subset X$, \ $U\in\bB$, be an open subset, and let
$U=\bigcup_\alpha V_\alpha$ be an open covering of $U$ such that
$V_\alpha\in\bB$ for all the indices~$\alpha$.
 Furthermore, for every pair of indices $\alpha$ and~$\beta$, let
$V_\alpha\cap V_\beta=\bigcup_\gamma W_{\alpha,\beta,\gamma}$ be
an open covering of the intersection $V_\alpha\cap V_\beta$
such that $W_{\alpha\beta\gamma}\in\bB$ for all~$\gamma$.
 Then the sequence of abelian groups
\begin{equation} \label{sheaf-axiom-for-topology-base}
 0\lrarrow \cM(U)\lrarrow\prod\nolimits_\alpha\cM(V_\alpha)
 \lrarrow\prod\nolimits_{\alpha,\beta,\gamma}\cM(W_{\alpha\beta\gamma})
\end{equation}
should be left exact.

 A presheaf of $\cO_X$\+modules $\cM$ on $\bB$ is called a \emph{sheaf
of $\cO_X$\+modules} (\emph{on\/~$\bB$}) if its underlying presheaf
of abelian groups $\cM$ is a sheaf of abelian groups on~$\bB$.
 We denote the category of sheaves of $\cO_X$\+modules on $\bB$ by
$(\bB,\cO_X)\Sh$.

\begin{prop} \label{sheaves-determined-on-topology-base}
 The functor of restriction of sheaves of $\cO_X$\+modules on $X$
to a topology base\/ $\bB$ is an equivalence of categories
$$
 (X,\cO_X)\Sh\simeq(\bB,\cO_X)\Sh.
$$
 In other words, any sheaf of $\cO_X$\+modules on\/ $\bB$ can be
extended to a sheaf of $\cO_X$\+modules on the whole of~$X$ (i.~e.,
on all the open subsets of~$X$) in a unique way, and any morphism
of sheaves of $\cO_X$\+modules on\/ $\bB$ can be uniquely extended
to a morphism of sheaves of $\cO_X$\+modules on~$X$.
\end{prop}

\begin{proof}
 This is~\cite[Section~0.3.2]{EGAI}, \cite[Lemma Tag~009U]{SP},
\cite[Proposition~2.1.3]{Pcosh}, \cite[Theorem~2.1(a)]{Pdomc},
or~\cite[Section~3.2]{Pform}.
\end{proof}

 Now we turn to copresheaves and cosheaves.
 A \emph{copresheaf of abelian groups} on\/ $\bB$ is a covariant
functor $\bB\rarrow\Ab$.
 We denote the abelian group assigned by a copresheaf $\fP$ to
an open subset $U\subset X$, \ $U\in\bB$, by~$\fP[U]$.
 For any two open subsets $V\subset U\subset X$, \ $U$, $V\in\bB$,
and a copresheaf $\fP$ on $\bB$, the homomorphism of abelian groups
$\fP[V]\rarrow\fP[U]$ assigned to the identity inclusion $V\rarrow U$
by the copresheaf $\fP$ is called the \emph{corestriction map}.

 A \emph{copresheaf of $\cO_X$\+modules\/ $\fP$ on\/~$\bB$} is
a copresheaf of abelian groups such that, for every open subset
$U\subset X$, \ $U\in\bB$, the abelian group $\fP[U]$ is endowed with
an $\cO_X(U)$\+module structure, and for every pair of open subsets
$V\subset U\subset X$, \ $U$, $V\in\bB$, the corestriction map
$\fP[V]\rarrow\fP[U]$ is a homomorphism of $\cO_X(U)$\+modules.
 Here, as above, the $\cO_X(V)$\+module $\fP[V]$ is endowed with
an $\cO_X(U)$\+module structure using the restriction homomorphism
of rings $\cO_X(U)\rarrow\cO_X(V)$.

 A copresheaf of abelian groups $\fP$ on $\bB$ is called a \emph{cosheaf
of abelian groups} (\emph{on\/~$\bB$}) if the following \emph{cosheaf
axiom} is satisfied.
 As above, we consider an open subset $U\subset X$, \ $U\in\bB$
and an open covering $U=\bigcup_\alpha V_\alpha$ such that
$V_\alpha\in\bB$ for all the indices~$\alpha$.
 Furthermore, for every pair of indices $\alpha$ and~$\beta$,
we choose an open covering $V_\alpha\cap V_\beta=
\bigcup_\gamma W_{\alpha,\beta,\gamma}$ of the intersection
$V_\alpha\cap V_\beta$ such that $W_{\alpha\beta\gamma}\in\bB$
for all~$\gamma$.
 For any such choices of $U$, $V_\alpha$, and $W_{\alpha\beta\gamma}$,
the sequence of abelian groups
\begin{equation} \label{cosheaf-axiom-for-topology-base}
 \bigoplus\nolimits_{\alpha,\beta,\gamma}\fP[W_{\alpha\beta\gamma}]
 \lrarrow\bigoplus\nolimits_\alpha\fP[V_\alpha]\lrarrow\fP[U]\lrarrow0
\end{equation}
should be right exact.

 A copresheaf of $\cO_X$\+modules $\fP$ on $\bB$ is called
a \emph{cosheaf of $\cO_X$\+modules} (\emph{on\/~$\bB$}) if its
underlying copresheaf of abelian groups $\fP$ is a cosheaf of abelian
groups on~$\bB$.
 We denote the category of cosheaves of $\cO_X$\+modules on $\bB$ by
$(\bB,\cO_X)\Cosh$.

 A \emph{co}(\emph{pre})\emph{sheaf} (of abelian groups or
$\cO_X$\+modules) \emph{on~$X$} is a co(pre)sheaf on the topology base
consisting of all the open subsets of~$X$.
 The category of cosheaves of $\cO_X$\+modules on $X$ will be denoted
by $(X,\cO_X)\Cosh$.

\begin{prop} \label{cosheaves-determined-on-topology-base}
 The functor of restriction of cosheaves of $\cO_X$\+modules on $X$
to a topology base\/ $\bB$ is an equivalence of categories
$$
 (X,\cO_X)\Cosh\simeq(\bB,\cO_X)\Cosh.
$$
 In other words, any cosheaf of $\cO_X$\+modules on\/ $\bB$ can be
extended to a cosheaf of $\cO_X$\+modules on the whole of~$X$ (i.~e.,
on all the open subsets of~$X$) in a unique way, and any morphism of
cosheaves of $\cO_X$\+modules on\/ $\bB$ can be uniquely extended
to a morphism of cosheaves of $\cO_X$\+modules on~$X$.
\end{prop}

\begin{proof}
 This is~\cite[Theorem~2.1.2]{Pcosh}, \cite[Theorem~2.1(b)]{Pdomc},
or~\cite[Section~3.2]{Pform}.
\end{proof}

\begin{rem} \label{one-covering-of-the-intersections-is-enough}
 The sheaf axiom~\eqref{sheaf-axiom-for-topology-base} implies, in
particular, that the map $\cM(U)\rarrow\prod_\alpha\cM(V_\alpha)$
is injective for any open covering $U=\bigcup_\alpha V_\alpha$ with
$U$, $V_\alpha\in\bB$.
 Using this fact, one can show that, once the open covering
$U=\bigcup_\alpha V_\alpha$ is fixed, it is not necessary to check
 the sheaf axiom for \emph{all} coverings $V_\alpha\cap V_\beta=
\bigcup_\gamma W_{\alpha\beta\gamma}$ of the intersections
$V_\alpha\cap V_\beta$ by open subsets $W_{\alpha\beta\gamma}\in\bB$.
 Chosing \emph{one} such covering for every pair of indices $\alpha$,
$\beta$ and checking~\eqref{sheaf-axiom-for-topology-base} for such
chosen coverings of the intersections $V_\alpha\cap V_\beta$ is
enough to establish the validity of the sheaf axiom for~$\cM$.

 Similarly, the cosheaf axiom~\eqref{cosheaf-axiom-for-topology-base}
implies, in particular, that the map $\bigoplus_\alpha\fP[V_\alpha]
\rarrow\fP[U]$ is surjective for any open covering
$U=\bigcup_\alpha V_\alpha$ with $U$, $V_\alpha\in\bB$.
 Using this fact, one can show that it is not necessary to check
the cosheaf axiom for \emph{all} coverings $V_\alpha\cap V_\beta=
\bigcup_\gamma W_{\alpha\beta\gamma}$ of the intersections
$V_\alpha\cap V_\beta$ by open subsets $W_{\alpha\beta\gamma}\in\bB$.
 Chosing \emph{one} such covering for every pair of indices $\alpha$,
$\beta$ and checking~\eqref{cosheaf-axiom-for-topology-base} for such
chosen coverings of the intersections $V_\alpha\cap V_\beta$ is
sufficient.
\end{rem}

\begin{rems} \label{finite-coverings-suffice-remarks}
 Cosheaves are dual-analogous to sheaves (cf.~\cite[Preface]{Pcosh}).
 In what ways can one make this informal statement precise?

\smallskip
 (a)~Let $\fP$ be a cosheaf of $\cO_X$\+modules on $X$ and $A$
be an arbitrary abelian group.
 Comparing the axioms~\eqref{sheaf-axiom-for-topology-base}
and~\eqref{cosheaf-axiom-for-topology-base}, one can see that the rule
$\cM(U)=\Hom_\boZ(\fP[U],A)$ for all open subsets $U\subset X$
defines a sheaf of $\cO_X$\+modules $\cM$ on~$X$.
 In particular, one can define the \emph{character sheaf} $\fP^+$ of
a cosheaf $\fP$ by the rule $\fP^+(U)=\Hom_\boZ(\fP[U],\boQ/\boZ)$
for all open subsets $U\subset X$.

\smallskip
 (b)~Assume that the topological space $X$ has a topology base $\bB$
with the following properties:
\begin{itemize}
\item all open subsets $U\subset X$, \ $U\in\bB$ are quasi-compact
(in the topology on $U$ induced from the topology of~$X$);
\item for any open subset $U\subset X$, \ $U\in\bB$, and any two open
subsets $V$, $W\subset U$, \ $V$, $W\in\bB$, the intersection $V\cap W$
is quasi-compact.
\end{itemize}
 Then, for (co)presheaves on $\bB$, it suffices to check both the sheaf
axiom~\eqref{sheaf-axiom-for-topology-base} and the cosheaf
axiom~\eqref{cosheaf-axiom-for-topology-base} for \emph{finite} open
coverings $U=\bigcup_\alpha V_\alpha$ and $V_\alpha\cap V_\beta=
\bigcup_\gamma W_{\alpha\beta\gamma}$
\emph{only}~\cite[Remark~2.1.4]{Pcosh}, \cite[Remarks~3.2.1]{Pform}.

\smallskip
 (c)~In particular, let $X$ be a scheme.
 Then the topology base $\bB$ consisting of all the quasi-compact,
quasi-separated open subsets $U\subset X$ satisfies both the conditions
of~(b).
 Alternatively, the topology base consisting of all the affine open
subschemes $U\subset X$ also satisfies both the conditions of~(b).

\smallskip
 (d)~Under the assumptions of~(b), let $\cM$ be a sheaf of
$\cO_X$\+modules on $X$ and $J$ be an injective abelian group.
 Comparing the axioms~\eqref{sheaf-axiom-for-topology-base}
and~\eqref{cosheaf-axiom-for-topology-base} for \emph{finite} open
coverings, one can see that the rule $\fP[U]=\Hom_\boZ(\cM(U),J)$ for
all open subsets $U\subset X$, \ $U\in\bB$ defines a cosheaf of
$\cO_X$\+modules $\fP$ on~$\bB$.
 Using Proposition~\ref{cosheaves-determined-on-topology-base}, one
can uniquely extend this cosheaf of $\cO_X$\+modules on $\bB$ to
a cosheaf of $\cO_X$\+modules $\fP$ on~$X$.
 In particular, one can define the \emph{character cosheaf} $\cM^+$ of
a sheaf $\cM$ by the rule $\cM^+[U]=\Hom_\boZ(\cM(U),\boQ/\boZ)$
for all open subsets $U\subset X$, \ $U\in\bB$.

\smallskip
 (e)~It is also clear from the sheaf
axiom~\eqref{sheaf-axiom-for-topology-base} that infinite products
exist in the category $(X,\cO_X)\Sh$, and the functors $\cM\longmapsto
\cM(U)\:(X,\cO_X)\Sh\rarrow\cO_X(U)\Modl$ of sections over all open
subsets $U\subset X$ preserve infinite products.
 Similarly, it is clear from the cosheaf 
axiom~\eqref{cosheaf-axiom-for-topology-base} that infinite coproducts
exist in the category $(X,\cO_X)\Cosh$, and the functors $\fP\longmapsto
\fP[U]\:(X,\cO_X)\Cosh\rarrow\cO_X(U)\Modl$ of cosections over all open
subsets $U\subset X$ preserve infinite coproducts.

\smallskip
 (f)~Now, under the assumptions of~(b), it is clear from the sheaf
axiom~\eqref{sheaf-axiom-for-topology-base} \emph{for finite open
coverings} that infinite coproducts exist in the category
$(\bB,\cO_X)\Sh$, and the functors of sections over all open
subsets $U\subset X$, \ $U\in\bB$ preserve infinite coproducts.
 In view of
Proposition~\ref{sheaves-determined-on-topology-base}, it follows
that infinite coproducts exist in the category $(X,\cO_X)\Sh$,
and the functors of sections over open subsets $U\subset X$, \ $U\in\bB$
preserve infinite coproducts.
 (These are well-known facts for sheaves, of course.)

 Similarly, under the assumptions of~(b), it is clear from the cosheaf
axiom~\eqref{cosheaf-axiom-for-topology-base} \emph{for finite open
coverings} that infinite products exist in the category
$(\bB,\cO_X)\Cosh$, and the functors of cosections over all open
subsets $U\subset X$, \ $U\in\bB$ preserve infinite products.
 In view of
Proposition~\ref{cosheaves-determined-on-topology-base}, it follows
that infinite products exist in the category $(X,\cO_X)\Cosh$,
and the functors of cosections over open subsets $U\subset X$, \
$U\in\bB$ preserve infinite products.
\end{rems}

\subsection{The contraadjustedness and contraherence conditions}
\label{contraadj-contraher-subsecn}
 The definitions of a \emph{contraherent cosheaf} and a \emph{locally
contraherent cosheaf} on a scheme are dual analogues of the description
of quasi-coherent sheaves suggested by Enochs and Estrada
in~\cite{EE}.
 We start with spelling out the relevant formulation of quasi-coherence
before passing to contraherence.

 Let $X$ be a scheme with a chosen open covering~$\bW$.
 We will say that an open subscheme $U\subset X$ is \emph{subordinate
to\/~$\bW$} if there is an open subset $W\in\bW$ such that $U\subset W$.
 Denote by $\bB$ the topology base in $X$ consisting of all the affine
open subschemes $U\subset X$ subordinate to~$\bW$.

 Let $\cM$ be a presheaf of $\cO_X$\+modules on~$\bB$.
 The presheaf $\cM$ is said to be \emph{quasi-coherent} if
the following \emph{quasi-coherence axiom} is satisfied:
\begin{enumerate}
\renewcommand{\theenumi}{\roman{enumi}}
\item for any pair of affine open subschemes $U\subset V$ subordinate
to $\bW$ in $X$, the restriction map of $\cO_X(U)$\+modules
$\cM(U)\rarrow\cM(V)$ induces an isomorphism of $\cO_X(V)$\+modules
$$
 \cO_X(V)\ot_{\cO_X(U)}\cM(U)\simeq\cM(V).
$$
\end{enumerate}

 Any quasi-coherent presheaf $\cM$ on $\bB$ is a sheaf on~$\bB$ (as one
can see from the exactness of the leftmost fragment of the \v Cech sequence~\eqref{module-cech-coresolution}
together with Remarks~\ref{one-covering-of-the-intersections-is-enough}
and~\ref{finite-coverings-suffice-remarks}(b\+-c)).
 So, by Proposition~\ref{sheaves-determined-on-topology-base}, one
can uniquely extend $\cM$ to a sheaf of $\cO_X$\+modules on the whole
of~$X$ (i.~e., on all the open subsets of~$X$).

 Furthermore, the quasi-coherence property of sheaves of
$\cO_X$\+modules is \emph{local} in the following sense.
 Denote by $\bW_X=\{X\}$ the trivial open covering of $X$ and by
$\bB_X$ the topology base consisting of all the affine open
subschemes of~$X$.

 Let $\cM$ be a sheaf of $\cO_X$\+modules on~$X$.
 Then the restriction of $\cM$ to $\bB_X$ is quasi-coherent if and
and only if the restriction of $\cM$ to $\bB$ is quasi-coherent.
 In other words, the quasi-coherence axiom~(i) holds for all pairs
of affine open subschemes $V\subset U$ subordinate to $\bW$ in $X$
if and only if it holds for all pairs of affine open subschemes
$V\subset U$ in~$X$.
 (This is essentially a restatement of the classical theory of
quasi-coherent sheaves on affine
schemes~\cite[Th\'eor\`eme~I.1.4.1]{EGAI},
\cite[Corollary~II.5.5]{HarAG}, \cite[Lemma Tag~01IA]{SP};
cf.~\cite[Proposition~3.4.2]{Pform}.)
 If this is the case, the sheaf of $\cO_X$\+modules $\cM$ is called
\emph{quasi-coherent}.

 We denote the category of quasi-coherent sheaves on $X$ by $X\Qcoh$.
 It is a Grothendieck abelian category~\cite[Proposition Tag~077P]{SP}.

 Now let $\fP$ be a copresheaf of $\cO_X$\+modules on~$\bB$.
 The copresheaf $\fP$ is said to be \emph{contraherent} if
the following two axioms are satisfied:
\begin{enumerate}
\renewcommand{\theenumi}{\roman{enumi}}
\setcounter{enumi}{1}
\item for any pair of affine open subschemes $U\subset V$ subordinate
to $\bW$ in $X$, the corestriction map of $\cO_X(U)$\+modules
$\fP[V]\rarrow\fP[U]$ induces an isomorphism of $\cO_X(V)$\+modules
$$
 \fP[V]\simeq\Hom_{\cO_X(U)}(\cO_X(V),\fP[U]);
$$
\item for any pair of affine open subschemes $U\subset V$ subordinate
to $\bW$ in $X$, one has
$$
 \Ext^1_{\cO_X(U)}(\cO_X(V),\fP[U])=0.
$$
\end{enumerate}
 The projective dimension of the $\cO_X(U)$\+module $\cO_X(V)$ never
exceeds~$1$ \,\cite[Remark~5.1]{Pphil}; that is why no $\Ext^n$
vanishing condition for $n\ge2$ is needed in~(iii).

 Fix an affine open subscheme $U\subset X$ subordinate to~$\bW$.
 Then condition~(iii) holds for all affine open subschemes $V\subset U$
if and only if the $\cO_X(U)$\+module $\fP[U]$ is contraadjusted.
 This is clear from Example~\ref{open-affine-very-flat-example}.
 For this reason, condition~(iii) is called
the \emph{contraadjustedness axiom}.
 Condition~(ii) is called the \emph{contraherence axiom}.

 Any contraherent copresheaf $\fP$ on $\bB$ is a cosheaf on~$\bB$ (as
one can see from the exactness of the rightmost fragment of the \v Cech sequence~\eqref{cta-module-cech-resolution}
together with Remarks~\ref{one-covering-of-the-intersections-is-enough}
and~\ref{finite-coverings-suffice-remarks}(b\+-c)).
 Hence, by Proposition~\ref{cosheaves-determined-on-topology-base}, one
can uniquely extend $\fP$ to a cosheaf of $\cO_X$\+modules on the whole
of~$X$ (i.~e., on all the open subsets of~$X$).

 The contraadjustedness condition for cosheaves of $\cO_X$\+modules
on $X$ is local, but the contaherence condition is \emph{not local}.
 A cosheaf of $\cO_X$\+modules $\fP$ on $X$ is said to be
\emph{$\bW$\+locally contraherent} if the restriction of $\fP$ to
the topology base $\bB$ is contraherent, i.~e., if
the contraadjustedness and contraherence axioms~(ii\+-iii) hold
for all pairs of affine open subschemes $V\subset U$ subordinate
to $\bW$ in~$X$.
 A cosheaf of $\cO_X$\+modules $\fP$ is called \emph{locally
contraherent} if it is $\bW$\+locally contraherent for \emph{some}
open covering $\bW$ of~$X$.

 A cosheaf of $\cO_X$\+modules $\fP$ on $X$ is called
\emph{contraherent} if it is $\bW_X$\+locally contraherent for
the trivial open covering $\bW_X=\{X\}$, i.~e., if the restriction
of $\fP$ to the topology base $\bB_X$ is contraherent.
 In other words, a cosheaf $\fP$ on $X$ is contraherent if and only if
the contraadjustedness and contraherence axioms~(ii\+-iii) hold
for all pairs of affine open subschemes $V\subset U$ in~$X$.

 A counterexample of a locally contraherent, but not contraherent
cosheaf of $\cO_X$\+modules on a scheme $X$ can be found
in~\cite[Example~3.2.1]{Pcosh}.
 This counterexample is quite simple in nature; taking $X$ to be
the spectrum of the ring of integers $\boZ$ or the ring of polynomials
in one variable $k[x]$ over a field~$k$ is enough.
 The locality of the contraadjustedness condition is explained
in~\cite[paragraph after Example~3.2.1]{Pcosh}.

 We denote the additive category of $\bW$\+locally contraherent
cosheaves on $X$ by $X\Lcth_\bW$.
 The additive category of (globally) contraherent cosheaves is
denoted by $X\Ctrh=X\Lcth_{\{X\}}$.
 The additive category of all locally contraherent cosheaves on $X$
is denoted by $X\Lcth=\bigcup_\bW X\Lcth_\bW$.
 So we have inclusions of full subcategories
$$
 X\Ctrh\subset X\Lcth_\bW\subset X\Lcth.
$$

\subsection{Locally contraherent cosheaves}
\label{loc-contrah-subsecn}
 The categories of (locally) contraherent cosheaves are almost never
abelian; but they always carry natural exact category structures.
 Let us define these exact structures.

 Let $X$ be a scheme with an open covering~$\bW$.
 Let $\bW'$ be another open covering of~$X$.
 We say that $\bW'$ is \emph{subordinate to\/~$\bW$} if every open
subscheme belonging to $\bW'$ is subordinate to $\bW$ (in the sense
of the definition in Section~\ref{contraadj-contraher-subsecn}).
 In this case, we have $X\Lcth_\bW\subset X\Lcth_{\bW'}$.

 A composable pair of morphisms of $\bW$\+locally contraherent cosheaves
$\fP\rarrow\fQ\rarrow\fR$ is said to be an (\emph{admissible})
\emph{short exact sequence} $0\rarrow\fP\rarrow\fQ\rarrow\fR\rarrow0$
in $X\Lcth_\bW$ if, for every affine open subscheme $U\subset X$
subordinate to $\bW$, the short sequence of $\cO_X(U)$\+modules
$0\rarrow\fP[U]\rarrow\fQ[U]\rarrow\fR[U]\rarrow0$ is exact.

 In view of Lemma~\ref{short-exact-sequnces-of-cta-are-colocal}, it
suffices to check this condition for affine open subschemes $U$
belonging to any chosen affine open covering of the scheme $X$
subordinate to the open covering~$\bW$.
 In this sense, the property of a short sequence of $\bW$\+locally
contraherent cosheaves on $X$ to be exact is \emph{local}.
 Using Lemma~\ref{Hom-into-pushout-lemma}(b) together with
Example~\ref{open-affine-very-flat-example}, one can easily show
that the class of short exact sequences defined above is indeed
an exact category structure on $X\Lcth_\bW$.
 
 Given an open covering $\bW'$ subordinate to $\bW$, it follows
from the previous paragraph that a short sequence in $X\Lcth_\bW$
is exact in $X\Lcth_\bW$ if and only if it is exact as a short
sequence in $X\Lcth_{\bW'}$.

 It follows from Lemma~\ref{colocality-of-epimorphisms}(a) that
the full subcategory $X\Lcth_\bW$ is closed under kernels of
admissible epimorphisms in $X\Lcth_{\bW'}$.
 So a morphism in $X\Lcth_\bW$ is an admissible epimorphism in
$X\Lcth_\bW$ if and only if it is an admissible epimorphism in
$X\Lcth_{\bW'}$.
 Furthermore, the homological criterion of contraherence of
a locally contraherent cosheaf on an affine
scheme~\cite[Lemma~3.2.2]{Pcosh} (see also~\cite[Corollary~4.6.1]{Pcosh}
and/or~\cite[Proposition~3.6.1]{Pform}) implies that the full
subcategory $X\Lcth_\bW$ is closed under extensions in $X\Lcth_{\bW'}$.

 Let us emphasize, however, that the full subcategory $X\Lcth_\bW$
is usually \emph{not} closed under cokernels of admissible
monomorphisms in $X\Lcth_{\bW'}$.
 See~\cite[Example~3.2.1]{Pcosh}.

 A short sequence in the category $X\Lcth$ is called \emph{exact}
if it is a short exact sequence in $X\Lcth_\bW$ for some open
covering $\bW$ of~$X$.
 So the full subcategory $X\Lcth_\bW$ is closed under extensions
and kernels of admissible epimorphisms in the exact category
$X\Lcth$, and a short sequence in $X\Lcth_\bW$ is exact if and only
if it is exact in $X\Lcth$.

 Specializing the discussion above in this section to the case of
the trivial covering $\bW_X=\{X\}$, we obtain an exact category
structure on the category $X\Ctrh=X\Lcth_{\bW_X}$ of contraherent
cosheaves on~$X$.

\begin{lem} \label{ctrh-cosheaves-on-affine-scheme}
 Let $U$ be an affine scheme.
 Then the functor\/ $\fP\longmapsto\fP[U]$ establishes an exact
category equivalence $U\Ctrh\simeq\cO(U)\Modl^\cta$ between
the exact category of contraherent cosheaves on $U$ and the exact
category of contraadjusted modules over the ring\/~$\cO(U)$.
\end{lem}

\begin{proof}
 Follows immediately from the definitions and the discussion above.
 One also has to use
Lemma~\ref{contraadjusted-change-of-scalars-lemma}(b) with
Example~\ref{open-affine-very-flat-example}.
\end{proof}

\subsection{The local cotorsion and local injectivity conditions}
\label{lct-lin-subsecn}
 Let $X$ be a scheme with an open covering~$\bW$.
 A $\bW$\+locally contraherent cosheaf $\fP$ on $X$ is said to be
\emph{locally cotorsion} if the $\cO_X(U)$\+module $\fP[U]$ is
cotorsion for every affine open subscheme $U\subset X$ subordinate
to~$\bW$.
 In view of
Lemma~\ref{cotorsion-injective-colocal-if-contraadjusted}(a), it
suffices to check this condition for affine open subschemes $U$
belonging to any chosen affine open covering of $X$ subordinate
to~$\bW$.

 In particular, let $\bW'$ be an open covering of $X$ subordinate
to~$\bW$.
 Then it follows from the previous paragraph that a $\bW$\+locally
contraherent cosheaf $\fP$ on $X$ is locally cotorsion if and only
if $\fP$ is a locally cotorsion $\bW'$\+locally contraherent
cosheaf on~$X$.

 We denote the full subcategory of locally cotorsion $\bW$\+locally
contraherent cosheaves on $X$ by $X\Lcth^\lct_\bW\subset X\Lcth_\bW$.
 The notation $X\Lcth^\lct=\bigcup_\bW X\Lcth^\lct_\bW\allowbreak
\subset X\Lcth$ stands for the full subcategory of locally cotorsion
locally contraherent cosheaves in $X\Lcth$.

 In particular, a contraherent cosheaf $\fP$ on $X$ is said to be
\emph{locally cotorsion} if it belongs to the full subcategory
$X\Ctrh^\lct=X\Lcth^\lct_{\{X\}}\subset X\Ctrh$.
 So we have inclusions of full subcategories
$$
  X\Ctrh^\lct\subset X\Lcth^\lct_\bW\subset X\Lcth^\lct.
$$

 Clearly, the full subcategory $X\Lcth^\lct_\bW$ is closed under
extensions and cokernels of admissible monomorphisms in $X\Lcth_\bW$.
 So the additive category $X\Lcth^\lct_\bW$ carries the exact
category structure inherited from $X\Lcth_\bW$.
 In particular, in the case of the trivial open covering
$\bW_X=\{X\}$, we obtain an exact structure on the category of
locally cotorsion contraherent cosheaves $X\Ctrh^\lct$.

 Similarly, the full subcategory $X\Lcth^\lct$ is closed under
extensions and cokernels of admissible monomorphisms in $X\Lcth$,
so we have the inherited exact category structure on $X\Lcth^\lct$.
 It follows that the full exact subcategories $X\Ctrh^\lct$, \
$X\Lcth^\lct_\bW$, and $X\Lcth^\lct$ are closed under extensions
and kernels of admissible epimorphisms in each other (see
also Lemma~\ref{colocality-of-epimorphisms}(b)).

 A $\bW$\+locally contraherent cosheaf $\fJ$ on $X$ is said to be
\emph{locally injective} if the $\cO_X(U)$\+module $\fJ[U]$ is
injective for every affine open subscheme $U\subset X$ subordinate
to~$\bW$.
 In view of
Lemma~\ref{cotorsion-injective-colocal-if-contraadjusted}(b), it
suffices to check this condition for affine open subschemes $U$
belonging to any chosen affine open covering of $X$ subordinate
to~$\bW$.
 In particular, it follows that a $\bW$\+locally contraherent cosheaf
$\fJ$ on $X$ is locally injective if and only if $\fJ$ is a locally
injective $\bW'$\+locally contraherent cosheaf on~$X$.

 We denote the full subcategory of locally injective $\bW$\+locally
contraherent cosheaves on $X$ by $X\Lcth^\lin_\bW\subset X\Lcth_\bW$.
 The notation $X\Lcth^\lin=\bigcup_\bW X\Lcth^\lin_\bW\allowbreak
\subset X\Lcth$ stands for the full subcategory of locally injective
locally contraherent cosheaves in $X\Lcth$.  {\hbadness=1150\par}

 In particular, a contraherent cosheaf $\fJ$ on $X$ is said to be
\emph{locally injective} if it belongs to the full subcategory
$X\Ctrh^\lin=X\Lcth^\lin_{\{X\}}\subset X\Ctrh$.
 So we have inclusions of full subcategories
$$
  X\Ctrh^\lin\subset X\Lcth^\lin_\bW\subset X\Lcth^\lin,
$$
as well as
$$
 X\Ctrh^\lin\subset X\Ctrh^\lct\subset X\Ctrh, \qquad
 X\Lcth^\lin_\bW\subset X\Lcth^\lct_\bW\subset X\Lcth_\bW,
$$
and
$$
 X\Lcth^\lin\subset X\Lcth^\lct\subset X\Lcth.
$$

 Clearly, the full subcategory $X\Lcth^\lin_\bW$ is closed under
extensions and cokernels of admissible monomorphisms in $X\Lcth_\bW$
and in $X\Lcth_\bW^\lct$.
 So the additive category $X\Lcth^\lin_\bW$ carries the exact
category structure inherited from $X\Lcth_\bW$ or $X\Lcth^\lct_\bW$.
 In particular, in the case of the trivial open covering
$\bW_X=\{X\}$, we obtain an exact structure on the category of
locally injective contraherent cosheaves $X\Ctrh^\lin$.

 Similarly, the full subcategory $X\Lcth^\lin$ is closed under
extensions and cokernels of admissible monomorphisms in $X\Lcth$
and in $X\Lcth^\lct$, so we have the inherited exact category
structure on $X\Lcth^\lin$.
 It follows that the full exact subcategories $X\Ctrh^\lin$, \
$X\Lcth^\lin_\bW$, and $X\Lcth^\lin$ are closed under extensions
and kernels of admissible epimorphisms in each other (see
also Lemma~\ref{colocality-of-epimorphisms}(c)).

\begin{lem} \label{ctrh-lct-lin-cosheaves-on-affine-scheme}
 Let $U$ be an affine scheme.
 Then the equivalence of exact categories $U\Ctrh\simeq\cO(U)\Modl^\cta$
from Lemma~\ref{ctrh-cosheaves-on-affine-scheme} restricts to \par
\textup{(a)} an equivalence $U\Ctrh^\lct\simeq\cO(U)\Modl^\cot$ between 
the exact category of locally cotorsion contraherent cosheaves on $U$
and the exact category of cotorsion $\cO(U)$\+modules; \par
\textup{(b)} an equivalence $U\Ctrh^\lin\simeq\cO(U)\Modl^\inj$ between
the split exact category of locally injective contraherent cosheaves
on $U$ and the additive (split exact) category of injective\/
$\cO(U)$\+modules. \qed
\end{lem}

\begin{rem} \label{character-cosheaf-ctrh-lct-remark}
 Let $X$ be a scheme and $\cM$ be a quasi-coherent sheaf on~$X$.
 Then the construction of
Remarks~\ref{finite-coverings-suffice-remarks}(b\+-d) provides
a cosheaf of $\cO_X$\+modules $\cM^+$ on $X$ defined by the rule
$\cM^+[U]=\Hom_\boZ(\cM(U),\boQ/\boZ)$ for all quasi-compact,
quasi-separated open subschemes $U\subset X$.
 Comparing the quasi-coherence axiom~(i) with the contraherence
axiom~(ii) from Section~\ref{contraadj-contraher-subsecn}, and taking
into account Lemma~\ref{flat-cotorsion-injective-tensor-Hom-lemma}(c),
one can see that $\cM^+$ is a locally cotorsion contraherent cosheaf
on $X$, that is, $\cM^+\in X\Ctrh^\lct$.
 For a flat quasi-coherent sheaf $\cF$ on $X$, the contraherent cosheaf
$\cF^+$ on $X$ is locally injective by
Lemma~\ref{flat-cotorsion-injective-tensor-Hom-lemma}(d).
 A further discussion of this construction can be found 
in~\cite[beginning of Section~2.4]{Pcosh}.
\end{rem}

 Before we finish this section, let us recall that, for any scheme $X$,
infinite products exist in the additive category $(X,\cO_X)\Cosh$,
and the functors of cosections $\fP\longmapsto\fP[U]\: 
(X,\cO_X)\Cosh\rarrow\cO_X(U)\Modl$ over quasi-compact, quasi-separated
open subschemes $U\subset X$ preserve infinite products by
Remarks~\ref{finite-coverings-suffice-remarks}(b,c,f).
 In particular, this holds for affine open subschemes $U\subset X$.
 It follows easily that, for any open covering $\bW$ of $X$,
the three full subcategories $X\Lcth^\lin_\bW\subset X\Lcth^\lct_\bW
\subset X\Lcth_\bW$ are preserved by the infinite products in
$(X,\cO_X)\Cosh$.
 So the exact categories of contraherent or locally contraherent
cosheaves (for a fixed covering~$\bW$), as well as their full
subcategories of locally cotorsion and locally injective (locally)
contraherent cosheaves, have infinite direct product functors,
and such direct products are preserved by the identity inclusion
functors between all those categories.

\Section{Direct Images of Locally Contraherent Cosheaves}

\subsection{Direct image of cosheaves of modules}
\label{direct-images-of-cosheaves-of-O-modules}
 Let $f\:(Y,\cO_Y)\rarrow(X,\cO_X)$ be a morphism of ringed spaces.
 Given a presheaf of $\cO_Y$\+modules $\cN$ on $Y$, the presheaf
of $\cO_X$\+modules $f_*\cN$ on $X$ is defined by the rule
$$
 (f_*\cN)(U)=\cN(f^{-1}(U)),
$$
where $f^{-1}(U)\subset Y$ is the preimage of an open subset
$U\subset X$ under the continuous map of topological spaces
$f\:Y\rarrow X$.
 The $\cO_Y(f^{-1}(U))$\+module $\cN(f^{-1}(U))$ is endowed with
an $\cO_X(U)$\+module structure using the ring homomorphism
$\cO_X(U)\rarrow\cO_Y(f^{-1}(U))$ that is a part of the structure of
a morphism of ringed spaces $f\:(Y,\cO_Y)\rarrow(X,\cO_X)$.
{\hbadness=1600\par}

 For any open covering $U=\bigcup_\alpha V_\alpha$ of an open
subset $U\subset X$, the open subsets $f^{-1}(V_\alpha)\subset Y$
form an open covering $f^{-1}(U)=\bigcup_\alpha f^{-1}(V_\alpha)$
of the open subset $f^{-1}(U)\subset Y$.
 Using this observation, one can readily check that the presheaf
$f_*\cN$ on $X$ satisfies the sheaf
axiom~\eqref{sheaf-axiom-for-topology-base} whenever so does
the presheaf $\cN$ on~$Y$.
 So we have the functor of direct image of sheaves of $\cO$\+modules
\begin{equation} \label{direct-image-of-sheaves-of-O-modules}
 f_*\:(Y,\cO_Y)\Sh\lrarrow(X,\cO_X)\Sh.
\end{equation}

 Dual-analogously, given a copresheaf of $\cO_Y$\+modules $\fQ$ on $Y$,
the copresheaf of $\cO_X$\+modules $f_!\fQ$ on $X$ is defined by
the rule
$$
 (f_!\fQ)[U]=\fQ[f^{-1}(U)]
$$
for all open subsets $U\subset X$.
 The $\cO_Y(f^{-1}(U))$\+module $\fQ[f^{-1}(U)]$ is endowed with
an $\cO_X(U)$\+module structure using the same ring homomorphism
$\cO_X(U)\rarrow\cO_Y(f^{-1}(U))$ as above.

 Using the same observation about the preimages of open coverings as
above, one can readily check that the copresheaf $f_!\fQ$ on $X$
satisfies the cosheaf axiom~\eqref{cosheaf-axiom-for-topology-base} 
whenever so does the copresheaf $\fQ$ on~$Y$.
 So we have the functor of direct image of cosheaves of $\cO$\+modules
\begin{equation} \label{direct-image-of-cosheaves-of-O-modules}
 f_!\:(Y,\cO_Y)\Cosh\lrarrow(X,\cO_X)\Cosh.
\end{equation}

 Let $f\:Y\rarrow X$ be a morphism of schemes.
 Assume additionally that the morphism~$f$ is quasi-compact and
quasi-separated.
 Then it is clear from the discussion in
Remarks~\ref{finite-coverings-suffice-remarks}(b,c,f) that
the functor of direct image of cosheaves of $\cO$\+modules $f_!$
\,\eqref{direct-image-of-cosheaves-of-O-modules} preserves infinite
direct products.

\subsection{Direct image and contraherence}
\label{direct-image-and-contraherence-subsecn}
 Let $f\:Y\rarrow X$ be a morphism of schemes.
 Whenever the morphism~$f$ is quasi-compact and quasi-separated,
the direct image functor $f_*\:(Y,\cO_Y)\Sh\rarrow(X,\cO_X)\Sh$
\,\eqref{direct-image-of-sheaves-of-O-modules} takes quasi-coherent
sheaves on $Y$ to quasi-coherent sheaves on~$X$
\,\cite[Proposition~I.6.7.1]{EGAI}, \cite[Lemma Tag~01LC]{SP}
(cf.~\cite[Proposition~3.9.1(b)]{Pform}).

 The theory of direct images of contraherent and locally contraherent
cosheaves is a bit more complicated.
 Generally speaking, for an open immersion $f\:Y\rarrow X$ of smooth
algebraic varieties (i.~e., smooth schemes of finite type) over a field,
even when the complement $X\setminus Y$ is smooth,
the direct image functor $f_!\:(Y,\cO_Y)\Cosh\rarrow(X,\cO_X)\Cosh$
\,\eqref{direct-image-of-cosheaves-of-O-modules} does not even take
contraherent cosheaves on $Y$ to locally contraherent cosheaves on~$X$
\,\cite[Remark~2.3.1 and Example~2.3.2]{Pcosh}.

 In order to make sure that the cosheaf $f_!\fQ$ is (locally)
contraherent, one needs to either assume the morphism~$f$ to be
affine (and more), or impose suitable adjustedness conditions
on a (locally) contraherent cosheaf~$\fQ$.
 In the rest of this section, we discuss affine morphisms~$f$, while
postponing the discussion of adjustedness conditions on $\fQ$ to
Sections~\ref{alf-subsecn}\+-\ref{homology-subsecn}
and~\ref{coflasque-secn}.

 Firstly, let $f\:Y\rarrow X$ be an affine morphism of schemes.
 Then the direct image functor
$f_!\:(Y,\cO_Y)\Cosh\rarrow(X,\cO_X)\Cosh$
\,\eqref{direct-image-of-cosheaves-of-O-modules}
takes contraherent cosheaves on $Y$ to contraherent cosheaves on $X$,
and induces an exact functor between the respective exact categories
\begin{equation} \label{ctrh-direct-image}
 f_!\:Y\Ctrh\lrarrow X\Ctrh.
\end{equation}
 Furthermore, the same direct image functor
$f_!\:(Y,\cO_Y)\Cosh\rarrow(X,\cO_X)\Cosh$ takes locally cotorsion
contraherent cosheaves on $Y$ to locally cotorsion contraherent
cosheaves on $X$, and induces an exact functor between their exact
categories
\begin{equation} \label{ctrh-lct-direct-image}
 f_!\:Y\Ctrh^\lct\lrarrow X\Ctrh^\lct.
\end{equation}
 Finally, if the morphism~$f$ is affine \emph{and} flat,
then the same functor of direct image of cosheaves of
$\cO$\+modules~$f_!$ takes locally injective contraherent cosheaves
on $Y$ to locally injective contraherent cosheaves on $X$, and induces
an exact functor between their exact categories
\begin{equation} \label{ctrh-lin-direct-image}
 f_!\:Y\Ctrh^\lin\lrarrow X\Ctrh^\lin.
\end{equation}
 The proofs of these assertions use
Lemmas~\ref{contraadjusted-change-of-scalars-lemma}(a),
\ref{cotorsion-change-of-scalars-lemma}(a),
and~\ref{injective-change-of-scalars-lemma}(a).
 We refer to~\cite[Section~2.3]{Pcosh} for the details
(see also~\cite[Section~2.4]{Pdomc}).

 Secondly, let $\bW$ be an open covering of the scheme $X$ and $\bT$
be an open covering of the scheme~$Y$.
 One says that a morphism of schemes $f\:Y\rarrow X$ is
\emph{$(\bW,\bT)$\+affine} if, for every affine open subscheme
$U\subset X$ subordinate to $\bW$, the preimage $f^{-1}(U)\subset Y$
is an affine open subscheme subordinate to~$\bT$.
 Every $(\bW,\bT)$\+affine morphism of schemes is
affine~\cite[D\'efinition~II.1.2.1 and Corollaire~II.1.3.2]{EGAII},
\cite[Lemma Tag~01S8]{SP}.

 Let $f\:Y\rarrow X$ be a $(\bW,\bT)$\+affine morphism of schemes.
 Then the direct image functor
$f_!\:(Y,\cO_Y)\Cosh\rarrow(X,\cO_X)\Cosh$
\,\eqref{direct-image-of-cosheaves-of-O-modules}
takes $\bT$\+locally contraherent cosheaves on $Y$ to $\bW$\+locally
contraherent cosheaves on $X$, and induces an exact functor between
the respective exact categories
\begin{equation} \label{lcth-W-T-direct-image}
 f_!\:Y\Lcth_\bT\lrarrow X\Lcth_\bW.
\end{equation}
 Furthermore, the same direct image functor
$f_!\:(Y,\cO_Y)\Cosh\rarrow(X,\cO_X)\Cosh$ takes locally cotorsion
$\bT$\+locally contraherent cosheaves on $Y$ to locally cotorsion
$\bW$\+locally contraherent cosheaves on $X$, and induces an exact
functor between their exact categories
\begin{equation} \label{lcth-W-T-lct-direct-image}
 f_!\:Y\Lcth^\lct_\bT\lrarrow X\Lcth^\lct_\bW.
\end{equation}
 Finally, if the morphism~$f$ is $(\bW,\bT)$\+affine \emph{and} flat,
then the same functor of direct image of cosheaves of
$\cO$\+modules~$f_!$ takes locally injective $\bT$\+locally
contraherent cosheaves on $Y$ to locally injective $\bW$\+locally
contraherent cosheaves on $X$, and induces an exact functor
between their exact categories
\begin{equation} \label{lcth-W-T-lin-direct-image}
 f_!\:Y\Lcth^\lin_\bT\lrarrow X\Lcth^\lin_\bW.
\end{equation}
 We refer to~\cite[Section~3.3]{Pcosh} for the details
(see also~\cite[Section~2.4]{Pdomc}).

\subsection{The direct image-restriction adjunction}
 In the context of the proofs of the main results of the present
paper, the only morphisms of schemes that we will use are open
immersions and the natural morphisms to a scheme from the spectra
of the stalks of its structure sheaf.
 For this reason, we do not go into an (otherwise rather complicated)
discussion of inverse images of locally contraherent cosheaves with
respect to arbitrary morphisms of schemes~\cite[Sections~2.3
and~3.3]{Pcosh}, \cite[Section~2.5]{Pdomc}, \cite[Section~3.12]{Pform}.
 For open immersion morphisms, the construction of the inverse
images of locally contraherent cosheaves (as well as of quasi-coherent
sheaves) is much more straightforward.

 Let $(X,\cO_X)$ be a ringed space and $Y\subset X$ be an open
subset.
 We endow the topological space $Y$ with the sheaf of rings
$\cO_Y=\cO_X|_Y$ obtained by restricting the sheaf of rings $\cO_X$
on $X$ to the open subset $Y\subset X$.
 Then there is a natural morphism of ringed spaces
$j\:(Y,\cO_Y)\rarrow(X,\cO_X)$.

 Given a sheaf of $\cO_X$\+modules $\cM$ on $X$, the restriction
$\cM|_Y$ of the sheaf $\cM$ to the open subset $Y\subset X$ is
the sheaf of $\cO_Y$\+modules on $Y$ defined by the rule
$$
 (\cM|_Y)(V)=\cM(V)
$$
for any open subset $V\subset Y\subset X$.
 This construction defines the restriction functor
\begin{equation} \label{restriction-of-sheaves-of-O-modules}
 j^*\:(X,\cO_X)\Sh\lrarrow(Y,\cO_Y)\Sh, \qquad j^*\cM=\cM|_Y.
\end{equation}

 For any two sheaves of $\cO$\+modules $\cM$ on $X$ and $\cN$ on $Y$,
there is a natural adjunction isomorphism of abelian groups
\begin{equation} \label{sheaves-direct-image-restriction-adjunction}
 \Hom_{\cO_X}(\cM,j_*\cN)\simeq\Hom_{\cO_Y}(j^*\cM,\cN),
\end{equation}
where the notation $\Hom_{\cO_X}$ and $\Hom_{\cO_Y}$ stands for
the groups of morphisms in the categories $(X,\cO_X)\Sh$ and
$(Y,\cO_Y)\Sh$.
 In other words, the restriction functor $j^*\:(X,\cO_X)\Sh\rarrow
(Y,\cO_Y)\Sh$ \,\eqref{restriction-of-sheaves-of-O-modules} is
left adjoint to the direct image functor $j_*\:(Y,\cO_Y)\Sh\rarrow
(X,\cO_X)\Sh$ \,\eqref{direct-image-of-sheaves-of-O-modules}.

 For any scheme $X$ and any open subscheme $Y\subset X$,
the restriction functor $j^*$
\,\eqref{restriction-of-sheaves-of-O-modules} takes quasi-coherent
sheaves on $X$ to quasi-coherent sheaves on~$X$.
 So we have the restriction functor
\begin{equation} \label{restriction-of-qcoh-sheaves}
 j^*\:X\Qcoh\lrarrow Y\Qcoh, \qquad j^*\cM=\cM|_Y.
\end{equation}
 The restriction functor of quasi-coherent sheaves $j^*$
\,\eqref{restriction-of-qcoh-sheaves} is exact as a functor between
Grothendieck abelian categories.

 Dual-analogously, given a cosheaf of $\cO_X$\+modules $\fP$ on $X$,
the restriction $\fP|_Y$ of the cosheaf $\fP$ to the open subset
$Y\subset X$ is the cosheaf of $\cO_Y$\+modules on $Y$ defined by
the rule
$$
 (\fP|_Y)[V]=\fP[V]
$$
for any open subset $V\subset Y\subset X$.
 This construction defines the restriction functor
\begin{equation} \label{restriction-of-cosheaves-of-O-modules}
 j^!\:(X,\cO_X)\Cosh\lrarrow(Y,\cO_Y)\Cosh, \qquad j^!\fP=\fP|_Y.
\end{equation}

 For any two cosheaves of $\cO$\+modules $\fP$ on $X$ and $\fQ$ on $Y$,
there is a natural adjunction isomorphism of abelian groups
\begin{equation} \label{cosheaves-direct-image-restriction-adjunction}
 \Hom^{\cO_X}(j_!\fQ,\fP)\simeq\Hom^{\cO_Y}(\fQ,j^!\fP),
\end{equation}
where the notation $\Hom^{\cO_X}$ and $\Hom^{\cO_Y}$ stands for
the groups of morphisms in the categories $(X,\cO_X)\Cosh$ and
$(Y,\cO_Y)\Cosh$.
 In other words, the restriction functor $j^!\:(X,\cO_X)\Cosh\rarrow
(Y,\cO_Y)\Cosh$ \,\eqref{restriction-of-cosheaves-of-O-modules} is
right adjoint to the direct image functor $j_!\:(Y,\cO_Y)\Cosh\rarrow
(X,\cO_X)\Cosh$ \,\eqref{direct-image-of-cosheaves-of-O-modules}.
 This is~\cite[formula~(2.7) in Section~2.3]{Pcosh},
\cite[formula~(19) in Section~2.4]{Pdomc}, or~\cite[formula~(30)
in Section~3.10]{Pform}.

 Now let $X$ be a scheme with an open covering $\bW$, and let
$Y\subset X$ be an open subscheme.
 Define the restriction $\bW|_Y$ of $\bW$ to $Y$ by the rule
$\bW|_Y=\{Y\cap W\mid W\in\bW\}$.
 Clearly, $\bW|_Y$ is an open covering of~$Y$.

 Then the restriction functor $j^!\:(X,\cO_X)\Cosh\rarrow
(Y,\cO_Y)\Cosh$ \,\eqref{restriction-of-cosheaves-of-O-modules}
takes $\bW$\+locally contraherent cosheaves on $X$ to
$\bW|_Y$\+locally contraherent cosheaves on $Y$, and induces
an exact functor between the respective exact categories
\begin{equation} \label{restriction-of-lcth-cosheaves}
 j^!\:X\Lcth_\bW\lrarrow Y\Lcth_{\bW|_Y}.
\end{equation}
 Furthermore, the same restriction functor $j^!\:(X,\cO_X)\Cosh
\rarrow(Y,\cO_Y)\Cosh$ takes locally cotorsion $\bW$\+locally
contraherent cosheaves on $X$ to locally cotorsion $\bW|_Y$\+locally
contraherent cosheaves on $Y$, and induces an exact functor between
their exact categories
\begin{equation} \label{restriction-of-lcth-lct-cosheaves}
 j^!\:X\Lcth^\lct_\bW\lrarrow Y\Lcth^\lct_{\bW|_Y}.
\end{equation}
 Finally, the same functor of restriction of cosheaves of
$\cO$\+modules to the open subset takes locally injective
$\bW$\+locally contraherent cosheaves on $X$ to locally injective
$\bW|_Y$\+locally contraherent cosheaves on $Y$, and induces
an exact functor between their exact categories
\begin{equation} \label{restriction-of-lcth-lin-cosheaves}
 j^!\:X\Lcth^\lin_\bW\lrarrow Y\Lcth^\lin_{\bW|_Y}.
\end{equation}

 Notice that if, in addition, the open immersion morphism
$j\:Y\rarrow X$ is affine, then it is $(\bW,\bW|_Y)$\+affine.
 In this case, the constructions of the functors of direct image
of (locally) contraherent cosheaves from
Section~\ref{direct-image-and-contraherence-subsecn} are applicable
to~$j$.

 One typical setting in which we will apply the constructions of direct
images and restrictions of locally contraherent cosheaves in the present
paper is that of an open subscheme subordinate to the covering.
 If the open subscheme $Y\subset X$ is subordinate to the open covering
$\bW$ of $X$, then the open covering $\bW|_Y$ of $Y$ contains the whole
scheme $Y$ as one of the open subsets in the covering, i.~e.,
$Y\in\bW|_Y$.
 In this case, all the $\bW|_Y$\+locally contraherent cosheaves on $Y$
are contraherent, that is, $Y\Lcth_{\bW|_Y}=Y\Ctrh$ (and similarly for
the locally cotorsion and locally injective locally contraherent
cosheaves).

\Section{Quasi-Compact Semi-Separated Schemes}  \label{qcss-secn}

\subsection{Admissible monomorphisms into locally injectives}
\label{qcoh-ssep-enough-lin-subsecn}
 The proof of the following proposition is based on the argument
dual-analogous to the proof of the existence of enough flat
quasi-coherent sheaves on a quasi-compact semi-separated scheme given
in~\cite[Lemma~A.1]{EP} (see also~\cite[Lemmas~4.1.1 and~4.1.8]{Pcosh}).

\begin{prop} \label{qcoh-ssep-enough-lin-prop}
 Let $X$ be a quasi-compact semi-separated scheme with an open
covering\/~$\bW$.
 Then, for any\/ $\bW$\+locally contraherent cosheaf\/ $\fP$ on $X$,
there exists a locally injective\/ $\bW$\+locally contraherent cosheaf\/
$\fJ$ on $X$ together with an admissible monomorphism\/ $\fP\rarrow\fJ$
in the exact category $X\Lcth_\bW$ of\/ $\bW$\+locally contraherent
cosheaves on~$X$.
\end{prop}

\begin{proof}
 This is a partial version of~\cite[Lemma~4.3.2(a)]{Pcosh}.
 For the sake of completeness of the exposition, we spell out
the proof below.

 Let $X=\bigcup_{\alpha=1}^N U_\alpha$ be a finite affine open covering
of $X$ subordinate to the open covering~$\bW$.
 Proceeding by induction on $\beta=0$,~\dots, $N$, we will construct
admissible monomorphisms $\fP\rarrow\fJ_\beta$ in $X\Lcth_\bW$ such
that the restriction of $\fJ_\beta$ to the open subset
$V_\beta=\bigcup_{\alpha=1}^\beta U_\alpha\subset X$ is a locally
injective $\bW|_{V_\beta}$\+locally contraherent cosheaf on~$V_\beta$.
 In the case of $\beta=0$ (the induction base), we have
$V_0=\varnothing$, so it suffices to take $\fJ_0=\fP$ with the identity
monomorphism $\fP\rarrow\fJ_0$.

 For $\beta\ge1$, assuming that an admissible monomorphism
$\fP\rarrow\fJ_{\beta-1}$ with the desired property has been
constructed already, we put $\fK=\fJ_{\beta-1}$, \ $V=V_{\beta-1}$,
and $U=U_\beta$.
 Denote the identity open immersion morphisms by $j\:U\rarrow X$
and $h\:V\rarrow X$.

 We need to use the fact that the assertion of the proposition holds
for the affine scheme $U=U_\beta$ with its trivial open covering
$\bW_U=\{U\}$.
 Specifically, there exists an admissible monomorphism
$j^!\fK\rarrow\fI$ in $U\Ctrh$ with a locally injective contraherent
cosheaf $\fI$ on~$U$.

 Indeed, by Lemmas~\ref{ctrh-cosheaves-on-affine-scheme}
and~\ref{ctrh-lct-lin-cosheaves-on-affine-scheme}(b), we have
an equivalence of exact categories $U\Ctrh\simeq\cO_X(U)\Modl^\cta$
identifying the full subcategories $U\Ctrh^\lin\subset U\Ctrh$
and $\cO_X(U)\Modl^\inj\subset\cO_X(U)\Modl$.
 Consider the contraadjusted $\cO_X(U)$\+module $(j^!\fK)[U]=\fK[U]$,
and pick an injective $\cO_X(U)$\+module $I$ together with
an injective morphism of $\cO_X(U)$\+modules $\fK[U]\rarrow I$.
 Then the quotient $\cO_X(U)$\+module $I/\fK[U]$ is contraadjusted,
so $\fK[U]\rarrow I$ is an admissible monomorphism in
$\cO_X(U)\Modl^\cta$.
 Let $j^!\fK\rarrow\fI$ be the related admissible monomorphism in
$U\Ctrh$.

 So we have an (admissible) short exact sequence $0\rarrow j^!\fK
\rarrow\fI\rarrow\fR\rarrow0$ in $U\Ctrh$.
 Applying the exact functor of direct image of contraherent
cosheaves $j_!$ \,\eqref{ctrh-direct-image} for the affine open
immersion morphism $j\:U\rarrow X$ (here we use the assumption that
the scheme $X$ is semi-separated), we obtain a short exact sequence
$0\rarrow j_!j^!\fK\rarrow j_!\fI\rarrow j_!\fR\rarrow0$ in
$X\Ctrh\subset X\Lcth_\bW$.
 The adjunction isomorphism of abelian
groups~\eqref{cosheaves-direct-image-restriction-adjunction} provides
an adjunction morphism $j_!j^!\fK\rarrow\fK$ in $X\Lcth_\bW$.
 Taking the pushout of this short exact sequence with respect to
this morphism, we obtain a short exact sequence
$0\rarrow\fK\rarrow\fL\rarrow j_!\fR\rarrow0$ in $X\Lcth_\bW$.

 Put $\fJ_\beta=\fL$.
 We claim that the $\bW$\+locally contraherent cosheaf $\fL$ on $X$
is locally injective in restriction to the open subscheme
$V_\beta=V\cup U$.
 In view of the discussion in Section~\ref{lct-lin-subsecn}, it
suffices to check that the contraherent cosheaf $j^!\fL$ is
locally injective on $U$ and the $\bW|_V$\+locally contraherent
cosheaf $h^!\fL$ is locally injective on~$V$.

 The restriction functor~$j^!$ takes the adjunction morphism
$j_!j^!\fK\rarrow\fK$ to an isomorphism.
 Hence we have $j^!\fL\simeq j^!j_!\fI\simeq\fI$, and $\fI$ is
a locally injective contraherent cosheaf on $U$ by construction.
 This takes care of the restriction to~$U$.

 The proof of the local injectivity of the locally contraherent
cosheaf $h^!\fL$ on $V$ is a bit more involved.
 Consider the intersection $U\cap V\subset X$, and denote
the identity open immersion morphisms by $j'\:U\cap V\rarrow V$
and $h'\:U\cap V\rarrow U$.

 We have a base change isomorphism $h^!j_!\fR\simeq j'_!h'{}^!\fR$,
which is obvious immediately from the construction of the direct image.
 Restricting the short exact sequence $0\rarrow j^!\fK\rarrow\fI
\rarrow\fR\rarrow0$ in $U\Ctrh$ to the open subscheme
$U\cap V\subset U$, we obtain a short exact sequence
$0\rarrow h'{}^!j^!\fK\rarrow h'{}^!\fI\rarrow h'{}^!\fR\rarrow0$
of contraherent cosheaves on $U\cap V$ (see
formula~\eqref{restriction-of-lcth-cosheaves}).

 Now the $\bW|_V$\+locally contraherent cosheaf $h^!\fK$ on $V$
is locally injective by the induction assumption; hence
the contraherent cosheaf $h'{}^!j^!\fK\simeq j'{}^!h^!\fK$ on
$U\cap V$ is locally injective, too.
 The contraherent cosheaf $h'{}^!\fI$ on $U\cap V$ is also locally
injective as a restriction of the locally injective contraherent
cosheaf $\fI$ to an open subscheme.
 Since the full subcategory $(U\cap\nobreak V)\Ctrh^\lin$ is
closed under cokernels of admissible monomorphisms in
$(U\cap\nobreak V)\Ctrh$ (see Section~\ref{lct-lin-subsecn}),
it follows that the contraherent cosheaf $h'{}^!\fR$ on $U\cap V$
is locally injective as well.
 Since the direct image functor~$j'_!$ with respect to the flat affine
morphism $j'\:U\cap V\rarrow V$ preserves the local injectivity
of contraherent cosheaves by formula~\eqref{ctrh-lin-direct-image},
we arrive to the conclusion that the contraherent cosheaf
$h^!j_!\fR\simeq j'_!h'{}^!\fR$ on $V$ is locally injective.

 Finally, restricting the short exact sequence $0\rarrow\fK\rarrow\fL
\rarrow j_!\fR\rarrow0$ in $X\Lcth_\bW$ to the open subscheme
$V\subset X$, we obtain a short exact sequence $0\rarrow h^!\fK\rarrow
h^!\fL\rarrow h^!j_!\fR\rarrow0$ in $V\Lcth_{\bW|_V}$ with locally
injective $\bW|_V$\+locally contraherent cosheaves $h^!\fK$ and
$h^!j_!\fR$.
 As the full subcategory $V\Lcth^\lin_{\bW|_V}$ is closed under
extensions in $V\Lcth_{\bW|_V}$ (see Section~\ref{lct-lin-subsecn}),
it follows that the $\bW|_V$\+locally contraherent cosheaf $h^!\fL$
on $V$ is also locally injective, as desired.
\end{proof}

\begin{rem}
 For any quasi-compact semi-separated scheme $X$ with an open covering
$\bW$, the full subcategory of locally injective $\bW$\+locally
contraherent cosheaves $X\Lcth^\lin_\bW$ is, actually, the right-hand
class of a complete cotorsion pair in the exact category $X\Lcth_\bW$.
 The objects from the left-hand class are called the \emph{antilocal}
contraherent cosheaves on $X$ (all such objects are globally
contraherent on $X$, and the full subcategory of antilocal contraherent
cosheaves on $X$ does not depend on the open covering~$\bW$).
 This is explained in~\cite[Section~4.3]{Pcosh} (see
also~\cite[Sections~4.4\+-4.5]{Pdomc} and Section~\ref{alf-subsecn}
below).
\end{rem}

\subsection{Projective contraherent cosheaves}
\label{projective-on-qcss-subsecn}
 In this section, which is based on~\cite[Section~4.5]{Pcosh}
(see also~\cite[Lemma~15.21]{Pdomc} or~\cite[Section~4.1]{Pform}),
we describe projective and projective locally cotorsion contraherent
cosheaves on quasi-compact semi-separated schemes.
 For the classification of projective locally cotorsion contraherent
cosheaves on Noetherian schemes, see
Section~\ref{noetherian-lct-prj-secn}, and for a discussion of
projective locally contraadjusted contraherent cosheaves on
Noetherian schemes of finite Krull dimension, see
Section~\ref{main-positive-result-secn}.

\begin{lem} \label{qcss-lcth-prj-lemma}
 Let $X$ be a quasi-compact semi-separated scheme with an open
covering\/~$\bW$.
 Let $X=\bigcup_{\alpha=1}^N U_\alpha$ be a finite affine open
covering of $X$ subordinate to\/~$\bW$.
 Then there are enough projective objects in the exact category
$X\Lcth_\bW$ of\/ $\bW$\+locally contraherent cosheaves on~$X$.
 A\/ $\bW$\+locally contraherent cosheaf on $X$ is projective if and
only if it is a direct summand of a finite direct sum of the direct
images of contraherent cosheaves on $U_\alpha$ corresponding to
very flat contraadjusted $\cO_X(U_\alpha)$\+modules.
\end{lem}

\begin{proof}
 This is~\cite[Lemma~4.5.1]{Pcosh}.
 For every index~$\alpha$, the exact category $U_\alpha\Ctrh$ of
contraherent cosheaves on $U_\alpha$ is equivalent to the exact
category $\cO_X(U_\alpha)\Modl^\cta$ of contraadjusted
$\cO_X(U_\alpha)$\+mod\-ules (see
Lemma~\ref{ctrh-cosheaves-on-affine-scheme}).
 By Theorem~\ref{very-flat-cotorsion-pair} and
Lemma~\ref{cotorsion-pair-injectives-projectives}(b), it follows
that the projective objects of $U_\alpha\Ctrh$ correspond to
very flat contraadjusted $\cO_X(U_\alpha)$\+modules, and there are
enough such projective objects. {\hbadness=1250\par}

 Let $j_\alpha\:U_\alpha\rarrow X$ denote the identity open immersion
morphism.
 As a particular case of the adjunction
formula~\eqref{cosheaves-direct-image-restriction-adjunction},
the direct image functor $j_\alpha{}_!\:U_\alpha\Ctrh\rarrow X\Lcth_\bW$
\,\eqref{ctrh-direct-image}, \eqref{lcth-W-T-direct-image} is left
adjoint to the restriction functor $j_\alpha^!\:X\Lcth_\bW\rarrow
U_\alpha\Ctrh$ \,\eqref{restriction-of-lcth-cosheaves}.
 Since the functor~$j_\alpha^!$ is exact, it follows that
the functor~$j_\alpha{}_!$ takes projective objects to projective
objects.

 In order to prove both assertions of the lemma, it remains to show
that for every $\bW$\+locally contraherent cosheaf $\fP$ on $X$ there
is an admissible epimorphism $\bigoplus_{\alpha=1}^N
j_\alpha{}_!\fF_\alpha\rarrow\fP$ in $X\Lcth_\bW$, where $\fF_\alpha$
are projective objects of $U_\alpha\Ctrh$.
 Indeed, given $\fP\in X\Lcth_\bW$, choose for every index~$\alpha$
an admissible epimorphism $\fF_\alpha\rarrow j_\alpha^!\fP$ in
$U_\alpha\Ctrh$ with a projective contraherent cosheaf $\fF_\alpha$
on~$U_\alpha$.
 By adjunction~\eqref{cosheaves-direct-image-restriction-adjunction},
we obtain the related morphism $f\:\bigoplus_{\alpha=1}^N
j_\alpha{}_!\fF_\alpha\rarrow\fP$ in $X\Lcth_\bW$.
 
 We still need to check that the morphism~$f$ is an admissible
epimorphism in $X\Lcth_\bW$.
 Here the point is that the property of a morphism of (locally)
contraherent cosheaves to be an admissible epimorphism is local
by Lemma~\ref{colocality-of-epimorphisms}(a).
 So it suffices to observe that the morphism $j_\alpha^!(f)$ is
an admissible epimorphism in $U_\alpha\Ctrh$ for every index~$\alpha$.
\end{proof}

\begin{cor} \label{qcss-lcth-prj-cor}
 Let $X$ be a quasi-compact semi-separated scheme with an open
covering\/~$\bW$.
 Then the classes of projective objects in the three exact categories
$X\Ctrh\subset X\Lcth_\bW\subset X\Lcth$ coincide with each other,
and there are enough of them in all the three categories.
 In particular, all the projective objects of $X\Lcth_\bW$ belong to
$X\Ctrh$, i.~e., they are contraherent cosheaves on the whole of~$X$.
\end{cor}

\begin{proof}
 Follows from Lemma~\ref{qcss-lcth-prj-lemma}.
\end{proof}

\begin{lem} \label{qcss-lcth-lct-prj-lemma}
 Let $X$ be a quasi-compact semi-separated scheme with an open
covering\/~$\bW$.
 Let $X=\bigcup_{\alpha=1}^N U_\alpha$ be a finite affine open
covering of $X$ subordinate to\/~$\bW$.
 Then there are enough projective objects in the exact category
$X\Lcth^\lct_\bW$ of locally cotorsion\/ $\bW$\+locally contraherent
cosheaves on~$X$.
 A locally cotorsion\/ $\bW$\+locally contraherent cosheaf on $X$ is
projective if and only if it is a direct summand of a finite direct
sum of the direct images of contraherent cosheaves on $U_\alpha$
corresponding to flat cotorsion $\cO_X(U_\alpha)$\+modules.
\end{lem}

\begin{proof}
 This is~\cite[Lemma~4.5.3]{Pcosh}.
 The argument is very similar to the proof of
Lemma~\ref{qcss-lcth-prj-lemma}.
 One needs to use
Lemma~\ref{ctrh-lct-lin-cosheaves-on-affine-scheme}(a),
Theorem~\ref{flat-cotorsion-pair} (for commutative rings),
and the fact that the direct image functor
$j_\alpha{}_!\:U_\alpha\Ctrh^\lct\rarrow X\Lcth^\lct_\bW$
\,\eqref{ctrh-lct-direct-image}, \eqref{lcth-W-T-lct-direct-image}
is left adjoint to the restriction functor
$j_\alpha^!\:X\Lcth^\lct_\bW\rarrow U_\alpha\Ctrh^\lct$
\,\eqref{restriction-of-lcth-lct-cosheaves} by
formula~\eqref{cosheaves-direct-image-restriction-adjunction}.
 Then one also needs to use Lemma~\ref{colocality-of-epimorphisms}(b).
\end{proof}

\begin{cor} \label{qcss-lcth-lct-prj-cor}
 Let $X$ be a quasi-compact semi-separated scheme with an open
covering\/~$\bW$.
 Then the classes of projective objects in the three exact categories
$X\Ctrh^\lct\subset X\Lcth^\lct_\bW\subset X\Lcth^\lct$ coincide with
each other, and there are enough of them in all the three categories.
 In particular, all the projective objects of $X\Lcth^\lct_\bW$ belong
to $X\Ctrh^\lct$, i.~e., they are (locally cotorsion) contraherent
cosheaves on the whole of~$X$.
\end{cor}

\begin{proof}
 Follows from Lemma~\ref{qcss-lcth-lct-prj-lemma}.
\end{proof}

 Given a quasi-compact semi-separated scheme $X$, we denote by
$X\Ctrh_\prj\subset X\Ctrh$ the full subcategory of projective
contraherent cosheaves on $X$, and by $X\Ctrh^\lct_\prj\subset
X\Ctrh^\lct$ the full subcategory of projective locally cotorsion
contraherent cosheaves on~$X$.
 Let us \emph{warn} the reader that our terminology and notation
are misleading: the projective objects of $X\Ctrh^\lct$ are usually
\emph{not} projective in $X\Ctrh$; so $X\Ctrh^\lct_\prj\varsubsetneq
X\Ctrh_\prj$.

\subsection{Antilocally flat contraherent cosheaves}
\label{alf-subsecn}
 Given a scheme $X$, we denote by $\Hom_X({-},{-})$ the groups of
morphisms in the abelian category $X\Qcoh$ and by $\Hom^X({-},{-})$
the groups of morphisms in the exact category $X\Lcth$.
 The $\Ext$ groups in the abelian category $X\Qcoh$ are denoted by
$\Ext_X^*({-},{-})$.
 We avoid references to $\Ext$ groups in the exact categories of
locally contraherent cosheaves before it has been established that
such $\Ext$ groups agree between various such exact (sub)categories
defined in Sections~\ref{loc-contrah-subsecn}\+-\ref{lct-lin-subsecn}
and the assumptions needed for that have been made.

 Let $X$ be a scheme with an open covering~$\bW$.
 Following the terminology of~\cite[Section~4.4]{Pcosh},
\cite[Section~15.3]{Pdomc}, and~\cite[Section~4.2]{Pform},
we will say that a $\bW$\+locally contraherent cosheaf $\fF$ on $X$
is \emph{antilocally flat} if, for every (admissible) short exact
sequence $0\rarrow\fP\rarrow\fQ\rarrow\fR\rarrow0$ of locally
cotorsion $\bW$\+locally contraherent cosheaves on $X$, the short
sequence of abelian groups
$$
 0\lrarrow\Hom^X(\fF,\fP)\lrarrow\Hom^X(\fF,\fQ)\lrarrow
 \Hom^X(\fF,\fR)\lrarrow0
$$
is exact.
 Notice that, according to this definition, the property of
a locally contraherent cosheaf on $X$ to be antilocally flat may be
affected by refinements of the covering~$\bW$ (so one has to
assume a fixed open covering~$\bW$).
 We will see below in this section that, over a quasi-compact
semi-separated scheme $X$, all antilocally flat $\bW$\+locally
contraherent cosheaves on $X$ are actually contraherent on
the whole of $X$, and the class of all such cosheaves does not
depend on~$\bW$.

\begin{lem} \label{antilocally-flat-basic-properties}
\textup{(a)} Let $U$ be an affine scheme with the trivial open
covering\/ $\bW_U=\{U\}$.
 Then a contraherent cosheaf\/ $\fF$ on $U$ is antilocally flat if
and only if the corresponding contraadjusted $\cO(U)$\+module\/
$\fF[U]$ is flat. \par
\textup{(b)} Let $X$ be a scheme and $Y\subset X$ be an open
subscheme such that the open immersion morphism $j\:Y\rarrow X$
is affine.
 Let\/ $\bW$ be an open covering of $X$ and\/ $\bW|_Y$ be its
restriction to~$Y$.
 Then the direct image functor $j_!\:Y\Lcth_{\bW|_Y}\rarrow
X\Lcth_\bW$ \,\eqref{lcth-W-T-direct-image} takes antilocally flat\/
$\bW|_Y$\+locally contraherent cosheaves on $Y$ to antilocally flat\/
$\bW$\+locally contraherent cosheaves on~$X$.
\end{lem}

\begin{proof}
 Part~(a): the ``if'' assertion follows immediately from
the definition of a cotorsion module (one still needs
to keep in mind Lemmas~\ref{ctrh-cosheaves-on-affine-scheme}
and~\ref{ctrh-lct-lin-cosheaves-on-affine-scheme}(a)).
 The ``only if'' follows from the easy part of
Theorem~\ref{flat-cotorsion-pair}; a simple version
of~\cite[Lemma~A.2.2(b)]{Pform} is also relevant.

 Part~(b): this assertion actually holds in greater generality,
as mentioned in~\cite[beginning of Section~4.4]{Pcosh}.
 We restrict ourselves to open immersions~$j$ here because we have
not defined inverse images of locally contraherent cosheaves
with respect to any other morphisms of schemes in this paper.

 Let $\fG$ be an antilocally flat $\bW|_Y$\+locally contraherent
cosheaf on~$Y$, and let $0\rarrow\fP\rarrow\fQ\rarrow\fR\rarrow0$
be a short exact sequence of locally cotorsion $\bW$\+locally
contraherent cosheaves on~$X$.
 Then $0\rarrow j^!\fP\rarrow j^!\fQ\rarrow j^!\fR\rarrow0$ is
a short exact sequence of locally cotorsion $\bW|_Y$\+locally
contraherent cosheaves on $Y$,
by formula~\eqref{restriction-of-lcth-lct-cosheaves}.
 So the short sequence of abelian groups $0\rarrow
\Hom^Y(\fG,j^!\fP)\rarrow\Hom^Y(\fG,j^!\fQ)\rarrow\Hom^Y(\fG,j^!\fR)
\rarrow0$ is exact.
 Now, by
adjunction~\eqref{cosheaves-direct-image-restriction-adjunction},
this short sequence is isomorphic to the short sequence $0\rarrow
\Hom^X(j_!\fG,\fP)\rarrow\Hom^X(j_!\fG,\fQ)\rarrow\Hom^X(j_!\fG,\fR)
\rarrow0$, which is consequently exact as well.
\end{proof}

 Given an exact category $\sE$, a class of objects $\sS\subset\sE$,
and an object $E\in\sE$, we say that $E$ is a \emph{finitely iterated
extension of objects from\/ $\sS$ in\/~$\sE$} if the morphism
$0\rarrow E$ can be presented as the composition of a finite sequence
of admissible monomorphisms with the cokernels belonging to~$\sS$.
 Equivalently, $E$ is a finitely iterated extension of objects from
$\sS$ if and only if the morphism $E\rarrow0$ can be presented as
the composition of a finite sequence of admissible epimorphisms
with the kernels belonging to~$\sS$.

\begin{prop} \label{qcoh-ssep-lct-preenvelope-prop}
 Let $X$ be a quasi-compact semi-separated scheme with an open
covering\/~$\bW$.
 Let $X=\bigcup_{\alpha=1}^N U_\alpha$ be a finite affine open
covering of $X$ subordinate to\/~$\bW$.
 Then, for any\/ $\bW$\+locally contraherent cosheaf\/ $\fM$ on $X$,
there exists an (admissible) short exact sequence\/ $0\rarrow\fM
\rarrow\fQ\rarrow\fF\rarrow0$ in the exact category $X\Lcth_\bW$ with
a locally cotorsion\/ $\bW$\+locally contraherent cosheaf\/ $\fQ$ on
$X$ and a contraherent cosheaf\/ $\fF$ on $X$ that is a finitely
iterated extension of the direct images of contraherent cosheaves on
$U_\alpha$ corresponding to flat contraadjusted
$\cO_X(U_\alpha)$\+modules.
\end{prop}

\begin{proof}
 This is~\cite[Lemma~4.4.1]{Pcosh}.
 The proof is based on the same construction as the proof of
Proposition~\ref{qcoh-ssep-enough-lin-prop} above; one just has to
look into it a bit more carefully.
 Proceeding by induction on~$\beta$, one constructs admissible
monomorphisms $\fM\rarrow\fN_\beta$ in $X\Lcth_\bW$ such that
the restriction of $\fN_\beta$ to the open subset
$V_\beta=\bigcup_{\alpha=1}^\beta U_\alpha\subset X$ is a locally
cotorsion $\bW|_{V_\beta}$\+locally contrahent cosheaf on $V_\beta$,
while the cokernel of $\fM\rarrow\fN_\beta$ is a finitely iterated
extension of the direct images of contraherent cosheaves on $U_\alpha$,
\ $\alpha\le\beta$, corresponding to flat contraadjusted
$\cO_X(U_\alpha)$\+modules.

 In particular, for an affine scheme $U=U_\beta$ with the open
immersion morphism $j\:U\rarrow X$, one has to use the fact that
there exists an admissible monomorphism $j^!\fN_{\beta-1}\rarrow\fR$
in $U\Ctrh$ with a locally cotorsion contraherent cosheaf $\fR$ on $U$
such that the cokernel of $j^!\fN_{\beta-1}\rarrow\fR$ is
a contraherent cosheaf $\fG$ on $U$ corresponding to a flat
contraadjusted $\cO_X(U)$\+module $\fG[U]$.
 Indeed, by Lemmas~\ref{ctrh-cosheaves-on-affine-scheme}
and~\ref{ctrh-lct-lin-cosheaves-on-affine-scheme}(a), we have
an equivalence of exact categories $U\Ctrh\simeq\cO_X(U)\Modl^\cta$
identifying the full subcategories $U\Ctrh^\lct\subset U\Ctrh$
and $\cO_X(U)\Modl^\cot\subset\cO_X(U)\Modl$.
 Then it remains to apply Theorem~\ref{flat-cotorsion-pair}(b)
to the $\cO_X(U)$\+module $M=\fN_{\beta-1}[U]$, and recall that
any quotient module of a contraadjusted (in particular, cotorsion)
module is contraadjusted (Section~\ref{cta-modules-subsecn}).

 Then one has to use the locality of the locally cotorsion property
of locally contraherent cosheaves, the fact that the full subcategory
they form is closed under extensions and cokernels of admissible
monomorphisms (all of this as per Section~\ref{lct-lin-subsecn}),
and finally, the fact that the local cotorsion property of
contraherent cosheaves is preserved by the direct images with respect
to affine morphisms (formula~\eqref{ctrh-lct-direct-image}).

 The construction produces a sequence of admissible monomorphisms
$\fM=\fN_0\rarrow\fN_1\rarrow\dotsb\rarrow\fN_N=\fQ$ in $X\Lcth_\bW$
such that the cokernel of the admissible monomorphism
$\fN_{\beta-1}\rarrow\fN_\beta$ is isomorphic to the direct image
$j_\beta{}_!\fG_\beta$ for the identity open immersion morphism
$j_\beta\:U_\beta\rarrow X$ and a contraherent cosheaf $\fG_\beta$
on $U_\beta$ corresponding to a flat contraadjusted module over
$\cO_X(U_\beta)$.
\end{proof}

 For a quasi-compact semi-separated scheme $X$, we can observe that
the $\Ext$ groups computed in the exact categories $X\Ctrh$, \
$X\Lcth_\bW$, and $X\Lcth$ agree.
 Indeed, the $\Ext$ groups in an exact category and its
(co)resolving subcategory always agree by~\cite[Lemma~6.3]{Pfltp}
(which is based on~\cite[beginning of Section~1]{Pfltp} and~\cite[Theorem~12.1(b)]{Kel},
\cite[Proposition~13.2.2(i)]{KS},
or~\cite[Proposition~A.2.1 or~A.3.1(a)]{Pcosh}).
 It remains to point out that the full subcategories $X\Ctrh\subset
X\Lcth_\bW\subset X\Lcth$ are resolving in each other, since they
are closed under extensions and kernels of admissible epimorphisms
in each other as per Section~\ref{loc-contrah-subsecn}, and
the full subcategory $X\Ctrh$ is generating in $X\Lcth_\bW$ and
$X\Lcth$ by Corollary~\ref{qcss-lcth-prj-cor}.

 On the other hand, under the same assumptions on $X$, the full
subcategories $X\Lcth^\lin_\bW\subset X\Lcth^\lct_\bW\subset X\Lcth_\bW$
are coresolving in each other.
 Indeed, they are closed under extensions and cokernels of
admissible monomorphisms in each other as per
Section~\ref{lct-lin-subsecn}, and the full subcategory
$X\Lcth^\lin_\bW$ is cogenerating in $X\Lcth^\lct_\bW$ and $X\Lcth_\bW$
by Proposition~\ref{qcoh-ssep-enough-lin-prop}.
 So the $\Ext$ groups in these three exact categories also agree,
and the $\Ext$ gropus in the exact categories $X\Lcth^\lin\subset
X\Lcth^\lct\subset X\Lcth$ agree as well.
 Based on the observations in this and the previous paragraph, we
will denote the $\Ext$ groups computed in any one of the mentioned
exact categories of cosheaves by $\Ext^{X,*}({-},{-})$.

\begin{cor} \label{qcss-alf-characterizations}
 Let $X$ be a quasi-compact semi-separated scheme with an open
covering\/~$\bW$.
 In this setting: \par
\textup{(a)} The following three conditions on a\/ $\bW$\+locally
contraherent cosheaf\/ $\fF$ on $X$ are equivalent:
\begin{enumerate}
\item $\fF$ is antilocally flat (for the given open covering\/~$\bW$);
\item $\Ext^{X,1}(\fF,\fQ)=0$ for all locally cotorsion\/
$\bW$\+locally contraherent cosheaves\/ $\fQ$ on~$X$;
\item $\Ext^{X,n}(\fF,\fQ)=0$ for all locally cotorsion\/
$\bW$\+locally contraherent cosheaves\/ $\fQ$ on $X$ and all $n\ge1$.
\end{enumerate} \par
\textup{(b)} The full subcategory of antilocally flat\/
$\bW$\+locally contraherent cosheaves is closed under extensions
and kernels of admissible epimorphisms in $X\Lcth_\bW$.
\end{cor}

\begin{proof}
 This is~\cite[Corollary~4.4.2]{Pcosh}.
 Part~(a) can be obtained by applying~\cite[Lemma~A.2.2(b)]{Pform}
to the full subcategory $X\Lcth_\bW^\lct\subset X\Lcth_\bW$
(which is cogenerating by Proposition~\ref{qcoh-ssep-enough-lin-prop}
or~\ref{qcoh-ssep-lct-preenvelope-prop} and closed under cokernels
of admissible monomorphisms by Section~\ref{lct-lin-subsecn}).
 Part~(b) follows from part~(a).
\end{proof}

\begin{lem} \label{qcoh-ssep-alf-precover-lem}
 Let $X$ be a quasi-compact semi-separated scheme with an open
covering\/~$\bW$.
 Let $X=\bigcup_{\alpha=1}^N U_\alpha$ be a finite affine open
covering of $X$ subordinate to\/~$\bW$.
 Then, for any\/ $\bW$\+locally contraherent cosheaf\/ $\fM$ on $X$,
there exists an (admissible) short exact sequence\/ $0\rarrow\fQ
\rarrow\fF\rarrow\fM\rarrow0$ in the exact category $X\Lcth_\bW$ with
a locally cotorsion\/ $\bW$\+locally contraherent cosheaf\/ $\fQ$ on
$X$ and a contraherent cosheaf\/ $\fF$ on $X$ that is a finitely
iterated extension of the direct images of contraherent cosheaves on
$U_\alpha$ corresponding to flat contraadjusted
$\cO_X(U_\alpha)$\+modules.
\end{lem}

\begin{proof}
 This is~\cite[Lemma~4.4.3]{Pcosh}.
 The assertion follows from
Proposition~\ref{qcoh-ssep-lct-preenvelope-prop} and
Lemma~\ref{qcss-lcth-prj-lemma} by virtue of Lemma~\ref{salce-lemma}(a).
\end{proof}

\begin{cor} \label{antilocally-flat-cor}
 Let $X$ be a quasi-compact semi-separated scheme with an open
covering\/~$\bW$.
 Let $X=\bigcup_{\alpha=1}^N U_\alpha$ be a finite affine open
covering of $X$ subordinate to\/~$\bW$.
 Then a\/ $\bW$\+locally contraherent cosheaf on $X$ is antilocally
flat if and only if it is a direct summand of a finitely iterated
extension of the direct images of contraherent cosheaves on $U_\alpha$
corresponding to flat contraadjusted
$\cO_X(U_\alpha)$\+modules.
\end{cor}

\begin{proof}
 This is~\cite[Corollary~4.4.4(c)]{Pcosh}.
 
 ``If'': by Lemma~\ref{antilocally-flat-basic-properties}(a\+-b),
the direct images to $X$ of contraherent cosheavers on $U_\alpha$
corresponding to flat contraadjusted $\cO_X(U_\alpha)$\+modules are
antilocally flat $\bW$\+locally contraherent cosheaves on~$X$.
 By Corollary~\ref{qcss-alf-characterizations}(b), the full subcategory
of antilocally flat $\bW$\+locally contraherent cosheaves is closed
under extensions in $X\Lcth_\bW$; it is also obviously closed under
direct summands.

 ``Only if'': let $\fG$ be an antilocally flat $\bW$\+locally
contraherent cosheaf on~$X$.
 Consider a short exact sequence $0\rarrow\fQ\rarrow\fF\rarrow\fG
\rarrow0$ from Lemma~\ref{qcoh-ssep-alf-precover-lem}.
 By Corollary~\ref{qcss-alf-characterizations}(a), we have
$\Ext^{X,1}(\fG,\fQ)=0$.
 Hence our short exact sequence splits, and $\fG$ is a direct summand
of~$\fF$.
\end{proof}

\begin{cor}
 Let $X$ be a quasi-compact semi-separated scheme.
 Then the class of antilocally flat\/ $\bW$\+locally contraherent
cosheaves on $X$ does not depend on an open covering\/ $\bW$ of
the scheme~$X$.
 All such cosheaves are (globally) contraherent on~$X$.
\end{cor}

\begin{proof}
 This is~\cite[Corollary~4.4.5]{Pcosh}.
 The assertion follows from Corollary~\ref{antilocally-flat-cor};
the point is that, for any two open coverings $\bW'$ and $\bW''$ of
$X$, there exists a finite affine open covering
$X=\bigcup_\alpha U_\alpha$ subordinate to both $\bW'$ and~$\bW''$.
 Furthermore, the direct images of contraherent cosheaves from
$U_\alpha$ are globally contraherent on $X$ by
formula~\eqref{ctrh-direct-image}, and the full subcategory
$X\Lcth_\bW$ (in particular, $X\Ctrh$) is closed under
extensions in $X\Lcth$ by Section~\ref{loc-contrah-subsecn}.
\end{proof}

 We will denote the full subcategory of antilocally flat
contraherent cosheaves by $X\Ctrh_\alf\subset X\Ctrh$.
 By the definition, both the full subcategories $X\Ctrh_\prj$
and $X\Ctrh^\lct_\prj$ are contained in $X\Ctrh_\alf$.
 In fact, one can see that the exact category $X\Ctrh_\alf$ has
\emph{both} enough projective and enough injective objects.
 In $X\Ctrh_\alf$, the full subcategory of projective objects
is $X\Ctrh_\prj$, and the full subcategory of injective objects
is $X\Ctrh^\lct_\prj=X\Ctrh_\alf\cap X\Ctrh^\lct$ (cf.\
Lemma~\ref{cotorsion-pair-injectives-projectives}(a) and
the following Corollary~\ref{antilocally-flat-cotorsion-pair}).

\begin{cor} \label{antilocally-flat-cotorsion-pair}
 Let $X$ be a quasi-compact semi-separated scheme with an open
covering\/~$\bW$.
 Then the pair of classes (antilocally flat contraherent cosheaves
on $X$, locally cotorsion\/ $\bW$\+locally contraherent cosheaves
on~$X$) is a complete cotorsion pair in the exact category $X\Lcth_\bW$.
 In other words: \par
\textup{(a)} for any\/ $\bW$\+locally contraherent cosheaf\/ $\fM$
on $X$, there exists a short exact sequence\/ $0\rarrow\fM\rarrow
\fQ\rarrow\fF\rarrow0$ in $X\Lcth_\bW$ with locally cotorsion\/
$\bW$\+locally contraherent cosheaf\/ $\fQ$ and an antilocally flat
contraherent cosheaf\/~$\fF$; \par
\textup{(b)} for any\/ $\bW$\+locally contraherent cosheaf\/ $\fM$
on $X$, there exists a short exact sequence\/ $0\rarrow\fQ\rarrow
\fF\rarrow\fM\rarrow0$ in $X\Lcth_\bW$ with an antilocally flat
contraherent cosheaf\/ $\fF$ and a locally cotorsion\/ $\bW$\+locally
contraherent cosheaf\/~$\fQ$.
\end{cor}

\begin{proof}
 This is~\cite[Corollary~4.4.4(a\+-b)]{Pcosh}.
 In view of the ``if'' assertion of
Corollary~\ref{antilocally-flat-cor}, part~(a) follows from
Proposition~\ref{qcoh-ssep-lct-preenvelope-prop} and part~(b) from
Lemma~\ref{qcoh-ssep-alf-precover-lem}.
 One has $\Ext^{X,1}(\fF,\fQ)=0$ for all $\fF\in X\Ctrh_\alf$ and
$\fQ\in X\Lcth^\lct_\bW$ by
Corollary~\ref{qcss-alf-characterizations}(a).
 By the same corollary, the class $X\Ctrh_\alf$ is the maximal class
with this property with respect to $X\Lcth^\lct_\bW$.
 To prove that $X\Lcth^\lct_\bW\subset X\Lcth_\bW$ is the maximal
class with the $\Ext^1$\+orthogonality property with respect to
$X\Ctrh_\alf$, one needs to use part~(a); see the ``direct summand
lemma'', \cite[Lemma~B.1.2]{Pcosh} or~\cite[Lemma~A.2.3]{Pform}.
\end{proof}

\begin{cor} \label{alf-open-restriction-cor}
 Let $X$ be a semi-separated Noetherian scheme and $Y\subset X$ be
an open subscheme such that the open immersion morphism $h\:Y\rarrow X$
is affine.
 Then the restriction functor $h^!\:X\Ctrh\rarrow Y\Ctrh$ takes
antilocally flat contraherent cosheaves on $X$ to antilocally flat
contraherent cosheaves on~$Y$.
\end{cor}

\begin{proof}
 This is~\cite[Corollary~4.4.9]{Pcosh}.
 The proof is based on Corollary~\ref{antilocally-flat-cor} together
with Proposition~\ref{flat-contraadjusted-colocalization}.
 The point is that if $X=\bigcup_\alpha U_\alpha$ is a finite affine
open covering of $X$, then $Y=\bigcup_\alpha (Y\cap U_\alpha)$ is
a finite affine open covering of~$Y$.
 The functor~$h^!$ is exact, so it preserves finitely iterated
extensions; and there is an obvious base change isomorphism
$h^!j_\alpha{}_!\simeq j'_{\alpha!}h_\alpha^{\prime\,!}$ for
the open immersion morphisms $j_\alpha\:U_\alpha\rarrow X$, \
$j_\alpha'\:Y\cap U_\alpha\rarrow Y$, and $h_\alpha'\:Y\cap U_\alpha
\rarrow U_\alpha$.
 The Noetherianity assumption and
Proposition~\ref{flat-contraadjusted-colocalization} are needed here
in order to claim that the restriction functor $h_\alpha^{\prime\,!}\:
U_\alpha\Ctrh\rarrow(Y\cap\nobreak U_\alpha)\Ctrh$ takes contraherent
cosheaves corresponding to flat contraadjusted
$\cO_X(U_\alpha)$\+modules to contraherent cosheaves corresponding
to flat contraadjusted $\cO_X(Y\cap\nobreak U_\alpha)$\+modules.
\end{proof}

\subsection{Homology of locally contraherent cosheaves}
\label{homology-subsecn}
 We only discuss the homology of locally contraherent cosheaves on
quasi-compact semi-separated schemes in this paper, and only for
the auxiliary purpose of proving Proposition~\ref{alf-direct-images}
below.
 A more substantial and detailed discussion can be found
in~\cite[Sections~4.6 and~6.3]{Pcosh}; see also~\cite[Sections~4.3
and~4.8]{Pform}.

 Similarly to the notation $\Gamma(X,\cM)=\cM(X)$ for the global
sections of a sheaf $\cM$ on a topological space $X$, we will use
the notation $\Delta(X,\fP)=\fP[X]$ for the global cosections of
a cosheaf $\fP$ on~$X$.

 For any scheme $X$, the functor $\Delta(X,{-})\:X\Lcth\rarrow\Ab$
is right exact on the exact category $X\Lcth$.
 In other words, for any (admissible) short exact sequence
$0\rarrow\fP\rarrow\fQ\rarrow\fR\rarrow0$ of locally contraherent
cosheaves on $X$, one has a right exact sequence
$$
 \Delta(X,\fP)\lrarrow\Delta(X,\fQ)\lrarrow\Delta(X,\fR)\lrarrow0.
$$
 This follows from the fact that the functor $\Delta(X,{-})\:
X\Lcth_\bW\rarrow\Ab$ can be computed as the cokernel of a map
of exact functors (of direct sums of cosections over affine open
subschemes subordinate to~$\bW$) using the cosheaf
axiom~\eqref{cosheaf-axiom-for-topology-base}.

 Let $X$ be a quasi-compact semi-separated scheme with an open
covering~$\bW$.
 Then the left derived functor $\boL_*\Delta(X,{-})$ of the functor of
global cosections $\Delta(X,{-}\:X\Lcth_\bW\rarrow\Ab$ is constructed,
as usually, using projective resolutions in the exact category
$X\Lcth_\bW$.
 The fact that enough projective objects exist in $X\Lcth_\bW$
(Lemma~\ref{qcss-lcth-prj-lemma}) is used here.

 By Corollary~\ref{qcss-lcth-prj-cor}, the class of projective
objects in $X\Lcth_\bW$ does not change when the open covering $\bW$
varies.
 According to Section~\ref{loc-contrah-subsecn}, the property of
a morphism in $X\Lcth_\bW$ to be an admissible epimorphism does not
change as $\bW$ varies, either; and the kernels of admissible
epimorphisms do not change.
 It follows that the left derived functors $\boL_*\Delta(X,{-})$
computed in all the exact categories $X\Ctrh\subset X\Lcth_\bW\subset
X\Lcth$ agree with each other.

 To any short exact sequence $0\rarrow\fP\rarrow\fQ\rarrow\fR\rarrow0$
in $X\Lcth$, the derived functor $\boL_*\Delta(X,{-})$ assigns
the usual long exact sequence of homology groups.

 On an affine scheme $U$, the functor $\Delta(U,{-})\:U\Ctrh\rarrow
\Ab$ is exact.
 Hence one has $\boL_n\Delta(U,\fP)=0$ for all contraherent cosheaves
$\fP$ on $U$ and all integers $n\ge1$ (this only applies to contraherent
and \emph{not} to locally contraherent cosheaves on~$U$;
cf.~\cite[Corollary~4.6.1]{Pcosh}).

 Now let $U\subset X$ be an affine open subscheme with the open
immersion morphism $j\:U\rarrow X$.
 Then the direct image functor $j_!\:U\Ctrh\rarrow X\Lcth_\bW$ is
exact and takes projective objects to projective objects (by
the proof of Lemma~\ref{qcss-lcth-prj-lemma}).
 Therefore, one has $\boL_n\Delta(X,j_!\fP)=0$ for all contraherent
cosheaves $\fP$ on $U$ and all $n\ge1$.

 Consequently, one has $\boL_n\Delta(X,\fE)=0$ for all $n\ge1$ whenever
a contraherent cosheaf $\fE$ on $X$ is (a direct summand of) a finitely
iterated extension of the direct images of contraherent cosheaves from
affine open subschemes of~$X$.
 In particular, one has $\boL_n\Delta(X,\fF)=0$ for all $n\ge1$ and
all antilocally flat contraherent cosheaves $\fF$ on~$X$ (in view of
Corollary~\ref{antilocally-flat-cor}).

\begin{prop} \label{alf-direct-images}
 Let $X$ be a quasi-compact semi-separated scheme and $Y\subset X$
be a quasi-compact open subscheme with the open immersion morphism
$j\:Y\rarrow X$.
 Then the functor of direct image of cosheaves of $\cO$\+modules
$j_!\:(Y,\cO_Y)\Cosh\rarrow(X,\cO_X)\Cosh$ takes antilocally flat
contraherent cosheaves on $Y$ to antilocally flat contraherent
cosheaves on $X$, and induces an exact functor between the respective
exact categories $j_!\:Y\Ctrh_\alf\rarrow X\Ctrh_\alf$.
\end{prop}

\begin{proof}
 This is a special case of the more general result
of~\cite[Corollary~4.6.3(c)]{Pcosh}, which is applicable to all
flat morphisms of quasi-compact semi-separated schemes, and not only
to open immersions (see also~\cite[Corollary~4.3.2]{Pform}).
 Notice that, given an antilocally flat contraherent cosheaf~$\fG$
on $Y$, the assertion that the cosheaf of $\cO_X$\+modules
$j_!\fG$ on $X$ is contraherent (or locally contraherent) is
\emph{not} covered by the results of
Section~\ref{direct-image-and-contraherence-subsecn}, and already
requires a proof (cf.\ a counterexample in~\cite[Remark~2.3.1 and
Example~2.3.2]{Pcosh}).

 By Corollary~\ref{antilocally-flat-cor}, \,$\fG$ is a direct summand
of a finitely iterated extension of the direct images of antilocally
flat contraherent cosheaves from affine open subschemes of~$Y$.
 Let $T\subset Y$ be an affine open subscheme with the open
immersion morphism $l\:T\rarrow Y$.
 According to Lemma~\ref{antilocally-flat-basic-properties}(b)
or Corollary~\ref{antilocally-flat-cor}, the functor
$j_!\circ l_!=(j\circ l)_!\:T\Ctrh\rarrow X\Ctrh$ takes antilocally
flat contraherent cosheaves on $T$ to antilocally flat contraherent
cosheaves on~$X$.

 Given an affine open subscheme $U\subset X$, consider the open
subscheme $V=Y\cap U\subset Y$.
 Then the open immersion morphism $k\:V\rarrow Y$ is affine, since
the open immersion morphism $h\:U\rarrow X$ is affine (even though
the scheme $V$ need not be affine).
 If the scheme $Y$ is Noetherian, then the contraherent cosheaf
$k^!\fG$ on $V$ is antilocally flat by
Corollary~\ref{alf-open-restriction-cor}.
 In the general case, following the proof of that corollary,
the contraherent cosheaf $k^!\fG$ on $V$ is a direct summand of
a finitely iterated extension of the direct images of contraherent
cosheaves from affine open subschemes of~$V$ (such contraherent
cosheaves are called \emph{antilocal} in~\cite[Section~4.3]{Pcosh}).
 In any case, according to the discussion above, it follows that
$\boL_n\Delta(V,k^!\fG)=0$ for all $n\ge1$.
 Therefore, the functor $\fG\longmapsto(j_!\fG)[U]=\fG[V]$ is
exact on the exact category of antilocally flat contraherent
cosheaves $Y\Ctrh_\alf$.

 We still need to show that the functor $j_!\:(Y,\cO_Y)\Cosh\rarrow
(X,\cO_X)\Cosh$ takes $Y\Ctrh_\alf$ into $X\Ctrh$ and extensions in
$Y\Ctrh_\alf$ to extensions in $X\Ctrh$.
 Arguing by induction on the length of a finitely iterated extension,
given a short exact sequence $0\rarrow\fF\rarrow\fG\rarrow\fH
\rarrow0$ in $Y\Ctrh_\alf$ such that $j_!\fF$ and $j_!\fH\in X\Ctrh$,
we need to prove that $j_!\fG\in X\Ctrh$ and $0\rarrow j_!\fF
\rarrow j_!\fG\rarrow j_!\fH\rarrow0$ is a short exact sequence
in $X\Ctrh$.
 On the level of cosections over affine open subschemes $U\subset X$,
we have already shown that in the previous paragraph.
 It remains to point out that the contraadjustedness axiom~(iii)
and the contraherence axiom~(ii) from
Section~\ref{contraadj-contraher-subsecn} hold for the cosheaf $j_!\fG$
whenever they hold for the cosheaves $j_!\fF$ and~$j_!\fH$.
 The facts that the full subcategory of contraadjusted $R$\+modules
is closed under extensions in $R\Modl$ and the functor $\Hom_R(F,{-})$
preserves exactness of short exact sequences of contraadjusted
$R$\+modules for any commutative ring $R$ and any very flat
$R$\+module~$F$ (together with
Example~\ref{open-affine-very-flat-example}) need to be used here.
\end{proof}

\subsection{Flat contraherent cosheaves} \label{qcss-flat-subsecn}
 Let $X$ be a scheme with an open covering~$\bW$.
 Following~\cite[Sections~3.4 and~3.7]{Pcosh}, we say that a cosheaf
of $\cO_X$\+modules $\fF$ on $X$ is \emph{$\bW$\+flat} if, for
any affine open subscheme $U\subset X$ subordinate to $\bW$,
the $\cO_X(U)$\+module $\fF[U]$ is flat.
 A cosheaf $\fF$ on $X$ is said to be \emph{flat} if it is
$\bW_X$\+flat for the trivial open covering $\bW_X=\{X\}$.

 We denote the full subcategory of $\bW$\+flat $\bW$\+locally
contraherent cosheaves on $X$ by $X\Lcth_\bW^\fl\subset X\Lcth_\bW$.
 It is clear that the full subcategory $X\Lcth_\bW^\fl$ is closed under
extensions and kernels of admissible epimorphisms in $X\Lcth_\bW$.
 So the full subcategory $X\Lcth_\bW^\fl$ inherits an exact category
structure from $X\Lcth_\bW$.
 In particular, the full exact category of flat contraherent cosheaves
on $X$ is denoted by $X\Ctrh^\fl\subset X\Ctrh$.

 The notion of flatness or $\bW$\+flatness of locally contraherent
cosheaves is mostly useful for Noetherian (or at least, coherent)
schemes.
 Given a Noetherian affine scheme $U$ and a contraherent cosheaf $\fG$
on $U$, the cosheaf $\fG$ on $U$ is flat if and only if
the corresponding $\cO(U)$\+module $\fG[U]$ is flat.
 This follows from Proposition~\ref{flat-contraadjusted-colocalization}.

 Let $Y$ be a scheme with an open covering~$\bT$.
 It is clear from the definitions that if $f\:Y\rarrow X$ is
a flat $(\bW,\bT)$\+affine morphism of schemes, then the direct image
functor~$j_!$ takes $\bT$\+flat cosheaves of $\cO_Y$\+modules to
$\bW$\+flat cosheaves of $\cO_X$\+modules.
 In particular, flatness of cosheaves of $\cO$\+modules is preserved
by the direct images with respect to flat affine morphisms of schemes.

\begin{cor} \label{antilocally-flat-are-flat}
 Let $X$ be a semi-separated Noetherian scheme.
 Then all antilocally flat contraherent cosheaves on $X$ are flat.
\end{cor}

\begin{proof}
 This is~\cite[Corollary~4.4.7]{Pcosh} or a particular case
of~\cite[Corollary~4.2.9]{Pform}.
 The assertion follows from Corollary~\ref{antilocally-flat-cor}
together with the discussion above in this section.
 Let us emphasize once again the crucial role of
Proposition~\ref{flat-contraadjusted-colocalization}.
\end{proof}

\Section{Main Negative Result}

 We start with a fairly standard lemma about quasi-coherent sheaves.

\begin{lem} \label{quasi-compact-intersected-with-affine}
 Let $X$ be an affine scheme, $V\subset X$ be a quasi-compact open
subscheme, and $U\subset X$ be an affine open subscheme.
 Let $\cM$ be a quasi-coherent sheaf on~$X$.
 Then the restriction map of $\cO(X)$\+modules
$\cM(V)\rarrow\cM(U\cap V)$ induces an isomorphism of
$\cO_X(U)$\+modules
$$
 \cO_X(U)\ot_{\cO(X)}\cM(V)\simeq\cM(U\cap V).
$$
\end{lem}

\begin{proof}
 Let $V=\bigcup_{\alpha=1}^N V_\alpha$ be a finite affine open covering
of the scheme~$V$.
 Then the sheaf axiom~\eqref{sheaf-axiom-for-topology-base} tells us
that the $\cO(X)$\+module $\cM(V)$ can be computed as the kernel of
the $\cO(X)$\+module map
\begin{equation} \label{M-V-sheaf-axiom-computing}
 \bigoplus\nolimits_{\alpha=1}^N \cM(V_\alpha)\lrarrow
 \bigoplus\nolimits_{1\le\alpha<\beta\le N} \cM(V_\alpha\cap V_\beta).
\end{equation}
 Now $U\cap V=\bigcup_{\alpha=1}^N U\cap V_\alpha$ is a finite affine
open covering of the scheme $U\cap V$.
 Hence the $\cO_X(U)$\+module $\cM(U\cap V)$ can be similarly computed
as the kernel of the $\cO_X(U)$\+module map
\begin{equation} \label{M-U-cap-V-sheaf-axiom-computing}
 \bigoplus\nolimits_{\alpha=1}^N \cM(U\cap V_\alpha)\lrarrow
 \bigoplus\nolimits_{1\le\alpha<\beta\le N}
 \cM(U\cap V_\alpha\cap V_\beta).
\end{equation}
 It remains to observe that
the map~\eqref{M-U-cap-V-sheaf-axiom-computing} can be obtained by
applying the functor $\cO_X(U)\ot_{\cO(X)}{-}$ to
the map~\eqref{M-V-sheaf-axiom-computing}, since
for any affine open subschemes $U$, $W\subset X$ one has
$\cO_X(U\cap W)\simeq\cO_X(U)\ot_{\cO(X)}\cO_X(W)$ and
\begin{multline*}
 \cM(U\cap W)\simeq\cO_X(U\cap W)\ot_{\cO_X(W)}\cM(W) \\
 \simeq\cO_X(U)\ot_{\cO(X)}\cO_X(W)\ot_{\cO_X(W)}\cM(W)\simeq
 \cO_X(U)\ot_{\cO(X)}\cM(W).
\end{multline*}
 Finally, this tensor product functor preserves kernels (since
the $\cO(X)$\+module $\cO_X(U)$ is flat, and in fact, very flat).
\end{proof}

\begin{lem} \label{co-generation-square-lemma}
 Let $R\rarrow S$ be a homomorphism of associative rings. \par
\textup{(a)} Suppose given a commutative square diagram of
left $R$\+module maps
\begin{equation} \label{generation-square-diagram}
\begin{gathered}
 \xymatrix{
  F \ar[r]^-q \ar@{->>}[d]_f & F' \ar[d]^{f'} \\
  M \ar[r]_-r & S\ot_RM 
 }
\end{gathered}
\end{equation}
where the $R$\+module map\/ $r\:M\rarrow S\ot_RM$ is induced
by the ring homomorphism $R\rarrow S$, the $R$\+module map
$f\:F\rarrow M$ is surjective, $F'$ is the underlying $R$\+module of
an $S$\+module, and\/ $f'\:F'\rarrow S\ot_RM$ is an $S$\+module map.
 Then the map $f'\:F'\rarrow S\ot_RM$ is surjective. \par
\textup{(b)} Suppose given a commutative square diagram of
left $R$\+module maps
\begin{equation} \label{cogeneration-square-diagram}
\begin{gathered}
 \xymatrix{
  J & J' \ar[l]_-d \\
  G \ar@{>->}[u]^g & \Hom_R(S,G) \ar[l]^-c \ar[u]_{g'}
 }
\end{gathered}
\end{equation}
where the $R$\+module map\/ $c\:\Hom_R(S,G)\rarrow G$ is induced
by the ring homomorphism $R\rarrow S$, the $R$\+module map
$g\:G\rarrow J$ is injective, $J'$ is the underlying $R$\+module of
an $S$\+module, and\/ $g'\:\Hom_R(S,G)\rarrow J'$ is an $S$\+module map.
 Then the map $g'\:\Hom_R(S,G)\rarrow J'$ is injective.
\end{lem}

\begin{proof}
 Part~(a), which is included here as an illustration for part~(b),
is clear and intuitive.
 The point is that the $S$\+module $S\ot_RM$ is generated by
the image of the map~$r$, which is contained in the image of
the $S$\+module map~$f'$.
 Let us prove part~(b), which we will really use.

 Let $v\:S\rarrow G$ be a left $R$\+module map; so $v\in\Hom_R(S,G)$.
 Assume that $g'(v)=0$.
 Then we also have $g'(sv)=0$ for all $s\in S$, since $g'$~is
an $S$\+module map.
 Since the map~$g$ is injective, it follows that $c(sv)=0$.
 By construction, we have $c(sv)=(sv)(1)=v(s)\in G$.
 Thus $v(s)=0$ for all $s\in S$; so $v=0$.
\end{proof}

 The following lemma is a dual analogue of~\cite[Lemma~2.1]{SS}.

\begin{lem} \label{slavik-stovicek-lemma}
 Let $X$ be an affine scheme, $U\subset X$ be a quasi-compact open
subscheme, and\/ $\fJ\in X\Ctrh^\lin$ be a locally injective
contraherent cosheaf on~$X$.
 Consider the homomorphism of $\cO_X(U)$\+modules
\begin{equation} \label{qcomp-in-affine-ctrh-module-map}
 \fJ[U]\lrarrow\Hom_{\cO(X)}(\cO_X(U),\fJ[X])
\end{equation}
induced by the corestriction map of $\cO(X)$\+modules\/
$\fJ[U]\rarrow\fJ[X]$.
 Then the map~\eqref{qcomp-in-affine-ctrh-module-map} is
an isomorphism.
\end{lem}

\begin{proof}
 Put $R=\cO(X)$.
 By Lemma~\ref{ctrh-lct-lin-cosheaves-on-affine-scheme}(b),
locally injective contraherent cosheaves $\fJ$ on $X$ correspond to
injective $R$\+modules $J=\fJ[X]$.
 For any injective $R$\+module $J$, there exists a flat (and even
free/projective) $R$\+module $F$ such that $J$ is a direct summand of
the $R$\+module $F^+=\Hom_\boZ(F,\boQ/\boZ)$.
 Denoting by $\cF$ the flat/projective/free quasi-coherent sheaf on
$X$ corresponding to the $R$\+module $F$, we conclude that $\fJ$ is
a direct summand of the locally injective contraherent cosheaf $\cF^+$
on $X$, in the notation of
Remarks~\ref{finite-coverings-suffice-remarks}(b\+-d)
and~\ref{character-cosheaf-ctrh-lct-remark}.
 So it suffices to consider the case of $\fJ=\cF^+$.
 In this case, the desired
isomorphism~\eqref{qcomp-in-affine-ctrh-module-map} can be obtained
by applying the character module functor
$({-})^+=\Hom_\boZ({-},\boQ/\boZ)$ to the isomorphism of
$\cO_X(U)$\+modules $\cO_X(U)\ot_{\cO(X)}\cF(X)\simeq\cF(U)$
from~\cite[Lemma~2.1]{SS}.
\end{proof}

 The following theorem is the first main result of this paper.
 It is a dual analogue of the theorem of Sl\'avik
and \v St\!'ov\'\i\v cek~\cite[Theorem~2.2]{SS}.

\begin{thm} \label{non-ssep-not-enough-lin-theorem}
 Let $X$ be a quasi-compact, quasi-separated scheme that is \emph{not}
semi-separated.
 Then there is a locally cotorsion contraherent cosheaf\/
$\fG\in X\Ctrh^\lct$ on $X$ for which there \emph{does not exist}
an admissible monomorphism\/ $\fG\rarrow\fJ$ in the exact category
$X\Lcth$ or $X\Lcth^\lct$ from\/ $\fG$ into any locally injective
locally contraherent cosheaf\/ $\fJ\in X\Lcth^\lin$.
\end{thm}

\begin{proof}
 We follow the argument from~\cite[proof of Theorem~2.2]{SS}, with
suitable enhancements.
 Let $U$, $V\subset X$ be two affine open subschemes whose intersection
$T=U\cap V$ is not affine.
 Since $X$ is quasi-separated, the scheme $T$ is quasi-compact.
 Put $R=\cO_X(U)$, so $U=\Spec R$.
 Then there is a finite collection of elements $f_1$,~\dots, $f_n\in R$
such that $T=\Spec R[f_1^{-1}]\cup\dotsb\cup\Spec R[f_n^{-1}]\subset U$.

 Let $I=(f_1,\dotsc,f_n)\subset R$ be the ideal spanned by $f_1$,
\dots, $f_n$ in $R$, and let $\cI$ be the quasi-coherent sheaf
(of ideals) on $U$ corresponding to the $R$\+module~$I$.
 Denote the identity open immersion morphism by $j\:U\rarrow X$,
and consider the locally cotorsion contraherent cosheaf
$\fG=(j_*\cI)^+\simeq j_!(\cI^+)$ on~$X$ (in the notation of
Remark~\ref{character-cosheaf-ctrh-lct-remark} and
Section~\ref{direct-images-of-cosheaves-of-O-modules}).
 Assuming that there exists an open covering $\bW$ of $X$,
a locally injective $\bW$\+locally contraherent cosheaf $\fJ$ on $X$,
and an admissible monomorphism $\fG\rarrow\fJ$ in $X\Lcth_\bW$,
we will come to a contradiction.

 Let $U=\bigcup_\alpha U_\alpha$ be a covering of $U$ by affine open
subschemes $U_\alpha$ subordinate to $\bW$, and let $V=\bigcup_\beta
V_\beta$ be a covering of $V$ by affine open subschemes $V_\beta$
subordinate to~$\bW$.
 Our next aim is to show that there exists a pair of indices
$(\alpha,\beta)$ such that the open subscheme $U_\alpha\cap V_\beta
\subset X$ is not affine.

 We will consider Cartesian products of schemes over $\Spec\boZ$;
so, throughout the rest of this proof, the Cartesian product
sign~$\times$ without subindex means $\times_{\Spec\boZ}$.
 The pair of open immersion morphisms $T\rarrow U$ and $T\rarrow V$
induces a morphism of schemes $\delta\:T\rarrow U\times V$.
 The affine scheme $U\times V$ is covered by its affine open
subschemes $U_\alpha\times V_\beta$, that is, $U\times V=
\bigcup_{\alpha,\beta}(U_\alpha\times V_\beta)$.
 For any pair of indices $\alpha$ and~$\beta$, we have
$\delta^{-1}(U_\alpha\times V_\beta)=U_\alpha\cap V_\beta\subset T$.
 Since the scheme $T$ is not affine, the scheme morphism~$\delta$
cannot be affine.
 By~\cite[D\'efinition~II.1.2.1 and Corollaire~II.1.3.2]{EGAII}
or~\cite[Lemma Tag~01S8]{SP}, it follows that a desired pair
of indices $\alpha$ and~$\beta$ with a nonaffine intersection
$U_\alpha\cap V_\beta$ exists.

 For the chosen pair of indices $\alpha$ and~$\beta$, we put
$U'=U_\alpha$, \ $V'=V_\beta$, and $T'=U_\alpha\cap V_\beta$.
 Then the open immersion morphism $T'\rarrow T$ is affine as
a base change of the morphism of affine schemes $U'\times V'
\rarrow U\times V$; specifically,
$T'=(U'\times V')\times_{(U\times V)}T$.
 Hence, restricting to $T'$ the affine open covering
$T=\bigcup_{i=1}^n\Spec R[f_i^{-1}]$, we obtain
an affine open covering $T'=\bigcup_{i=1}^n(T'\cap\Spec R[f_i^{-1}])$
of the scheme~$T'$.
 To repeat, the point is that the open subschemes
$T'\cap\Spec R[f_i^{-1}]$ are affine.
 Hence, according to~\cite[Chapter~II, Section~II.2,
Exercise~2.17(b)]{HarAG}, the fact that the scheme $T'$ is not affine
implies that the restrictions of the elements $f_1$,~\dots,
$f_n\in\cO_X(U)$ to the open subscheme $T'\subset U$ \emph{do not}
generate the unit ideal of the ring $\cO_X(T')$.

 The following commutative square diagram is a dual analogue
of the leftmost square of the main commutative diagram
in~\cite[proof of Theorem~2.2]{SS}:
\begin{equation} \label{slavik-stovicek-leftmost-square}
\begin{gathered}
 \xymatrix{
  \fJ[V'] & \fJ[U\cap V'] \ar[l] \\
  \fG[V'] \ar@{>->}[u] \ar@{=}[r] & \fG[U\cap V'] \ar[u]
 }
\end{gathered}
\end{equation}
 Here the upper horizontal map $\fJ[U\cap V']\rarrow\fJ[V']$
is the corestriction map in the locally injective locally
contraherent cosheaf~$\fJ$.
 The corestriction map $\fG[U\cap V']\rarrow\fG[V']$ in
the contraherent cosheaf $\fG=(j_*\cI)^+$ is an isomorphism because
the restriction map $(j_*\cI)(V')\rarrow(j_*\cI)(U\cap V')$ is
an isomorphism by the construction of the direct image functor
$j_*\:U\Qcoh\rarrow X\Qcoh$.
 We are using the fact that the scheme $U\cap V'$ is quasi-compact
and quasi-separated here in order to compute
$(j_*\cI)^+[U\cap\nobreak V']$.

 The map of cosections $\fG[V']\rarrow\fJ[V']$ induced by
the admissible monomorphism of $\bW$\+locally contraherent cosheaves
$\fG\rarrow\fJ$ on $X$ is injective because $V'\subset X$ is
an affine open subscheme subordinate to~$\bW$.
 It follows from the commutativity of
the diagram~\eqref{slavik-stovicek-leftmost-square} that the map of
cosections $\fG[U\cap V']\rarrow\fJ[U\cap V']$ induced by the morphism
of cosheaves $\fG\rarrow\fJ$ is injective, too.

 Now let us consider the commutative square diagram
\begin{equation} \label{inserted-middle-square}
\begin{gathered}
 \xymatrix{
  \fJ[U\cap V'] & \fJ[U'\cap V'] \ar[l] \\
  \fG[U\cap V'] \ar@{>->}[u] & \fG[U'\cap V'] \ar[l] \ar[u]
 }
\end{gathered}
\end{equation}
 Here the horizontal maps are the corestriction maps in the cosheaves
$\fJ$ and $\fG$, while the vertical maps are the maps of cosections
induced by the morphism of cosheaves $\fG\rarrow\fJ$.
 The leftmost vertical map in~\eqref{inserted-middle-square} is
injective according to the argument above.

 Put $\cM=j_*\cI$; so $\fG=\cM^+$.
 The corestriction map $\fG[U'\cap V']\rarrow\fG[U\cap V']$ is
obtained by applying the character module functor $({-})^+$
to the restriction map $\cM(U\cap V')\rarrow\cM(U'\cap V')$.
 Now we use Lemma~\ref{quasi-compact-intersected-with-affine} for
the affine scheme $U$, its quasi-compact open subscheme $U\cap V'
\subset U$, and the affine open subscheme $U'\subset U$.
 According to the lemma, we have an isomorphism of $\cO_X(U')$\+modules
$\cO_X(U')\ot_{\cO_X(U)}\cM(U\cap V')\simeq\cM(U'\cap V')$.
 Applying the functor $({-})^+$, we obtain an isomorphism
$$
 \fG[U'\cap V']\simeq\Hom_{\cO_X(U)}(\cO_X(U'),\>\fG[U\cap V']).
$$

 Put $R=\cO_X(U)$ and $S=\cO_X(U')$.
 Now the diagram~\eqref{inserted-middle-square} is a special case
of the diagram~\eqref{cogeneration-square-diagram}
from Lemma~\ref{co-generation-square-lemma}(b).
 According to that lemma, it follows that the rightmost vertical
map $\fG[U'\cap V']\rarrow\fJ[U'\cap V']$
in~\eqref{inserted-middle-square} is injective as well.

 Finally, the following commutative square diagram is a dual analogue
of the rightmost square of the main commutative diagram
in~\cite[proof of Theorem~2.2]{SS}:
\begin{equation} \label{slavik-stovicek-rightmost-square}
\begin{gathered}
 \xymatrix{
  \fJ[U'\cap V'] \ar[r] & \fJ[U'] \\
  \fG[U'\cap V'] \ar@{>->}[u] \ar[r] & \fG[U'] \ar@{>->}[u]
 }
\end{gathered}
\end{equation}
 Once again, the horizontal maps are the corestriction maps in
the cosheaves $\fJ$ and $\fG$, while the vertical maps are the maps of
cosections induced by the morphism of cosheaves $\fG\rarrow\fJ$.
 The leftmost vertical map
in~\eqref{slavik-stovicek-rightmost-square} is injective according to
the argument above, while the rightmost vertical map
in~\eqref{slavik-stovicek-rightmost-square} is injective because
$\fG\rarrow\fJ$ is an admissible monomorphism in $X\Lcth_\bW$ and
$U'\subset X$ is an affine open subscheme subordinate to~$\bW$.

 The argument finishes similarly to the proof in~\cite{SS}.
 The commutative diagram of $\cO_X(U')$\+module
maps~\eqref{slavik-stovicek-rightmost-square} induces a commutative
diagram of $\cO_X(U'\cap\nobreak V')$\+module maps
\begin{equation} \label{slavik-stovicek-subsequent-diagram}
\begin{gathered}
 \xymatrix{
  \fJ[U'\cap V'] \ar[r]
   & \Hom_{\cO_X(U')}(\cO_X(U'\cap V'),\>\fJ[U']) \\
  \fG[U'\cap V'] \ar@{>->}[u] \ar[r]
   & \Hom_{\cO_X(U')}(\cO_X(U'\cap V'),\>\fG[U']) \ar@{>->}[u]
 }
\end{gathered}
\end{equation}
 The upper horizontal map in~\eqref{slavik-stovicek-subsequent-diagram}
is an isomorphism by Lemma~\ref{slavik-stovicek-lemma} applied to
the quasi-compact open subscheme $U'\cap V'$ in the affine scheme $U'$
and the locally injective contraherent cosheaf $\fJ|_{U'}$ on~$U'$.
 Since the leftmost vertical map
in~\eqref{slavik-stovicek-subsequent-diagram} is injective, it follows
that the lower horizontal map is injective, too.

 The lower horizontal map $\fG[U'\cap V']\rarrow
\Hom_{\cO_X(U')}(\cO_X(U'\cap V'),\>\fG[U'])$
in~\eqref{slavik-stovicek-subsequent-diagram} can be obtained by
applying the functor $({-})^+$ to the natural map
\begin{equation} \label{slavik-stovicek-surjective-map}
 \cO_X(U'\cap V')\ot_{\cO_X(U')}\cI(U')\lrarrow\cI(U'\cap V').
\end{equation}
 Consequently, the map~\eqref{slavik-stovicek-surjective-map} is
surjective.

 Now we recall that $\cI$ is a quasi-coherent sheaf of ideals on
$U=\Spec R$.
 By construction, the restriction of $\cI$ to $\Spec R[f_i^{-1}]
\subset U$ coincides with the structure sheaf of $\Spec R[f_i^{-1}]$
for every $1\le i\le n$.
 Therefore, the restriction of the sheaf of ideals $\cI$ to
the open subscheme $U\cap V=\bigcup_{i=1}^n\Spec R[f_i^{-1}]$
coincides with the structure sheaf of $U\cap V$.
 So we have $\cI(U'\cap V')=\cO_X(U'\cap V')$.

 On the other hand, we have $\cI(U')=\cO_X(U')\ot_{\cO_X(U)}\cI(U)$.
 Hence the map~\eqref{slavik-stovicek-surjective-map} is isomorphic
to the natural map
\begin{equation} \label{surjective-map-rewritten}
 \cO_X(U'\cap V')\ot_{\cO_X(U)}\cI(U)\lrarrow\cO_X(U'\cap V').
\end{equation}
 Recall also the notation $\cI(U)=I=(f_1,\dotsc,f_n)\subset R$
and $T'=U'\cap V'$.
 It remains to point out that surjectivity of
the map~\eqref{surjective-map-rewritten} contradicts our previous
observation that restrictions of the elements $f_1$,~\dots,
$f_n\in\cO_X(U)$ to the open subscheme $T'\subset U$ do not
generate the unit ideal of the ring $\cO_X(T')$.
 We have arrived at a contradiction, proving the theorem.
\end{proof}

 The following corollary summarizes the results of
Section~\ref{qcoh-ssep-enough-lin-subsecn} and the present
section, together with~\cite[Theorem~2.2]{SS}.

\begin{cor} \label{slavik-stovicek-style-main-corollary}
 Let $X$ be a quasi-compact, quasi-separated scheme with an open
covering\/~$\bW$.
 Then the following conditions are equivalent:
\begin{enumerate}
\item the scheme $X$ is semi-separated;
\item every quasi-coherent sheaf on $X$ is a quotient sheaf of
a flat quasi-coherent sheaf;
\item for every contraherent cosheaf\/ $\fP\in X\Ctrh$, there exists
an admissible monomorphism\/ $\fP\rarrow\fJ$ in the exact category
$X\Ctrh$ from\/ $\fP$ to a locally injective contraherent cosheaf\/
$\fJ\in X\Ctrh^\lin$;
\item for every locally cotorsion contraherent cosheaf\/
$\fP\in X\Ctrh^\lct$, there exists an admissible monomorphism\/
$\fP\rarrow\fJ$ in the exact category $X\Ctrh^\lct$ (or equivalently,
in $X\Ctrh$) from\/ $\fP$ to a locally injective contraherent
cosheaf\/ $\fJ\in X\Ctrh^\lin$;
\item for every\/ $\bW$\+locally contraherent cosheaf\/
$\fP\in X\Lcth_\bW$, there exists an admissible monomorphism\/
$\fP\rarrow\fJ$ in the exact category $X\Lcth_\bW$ from\/ $\fP$
to a locally injective\/ $\bW$\+locally contraherent cosheaf\/
$\fJ\in X\Lcth^\lin_\bW$;
\item for every locally cotorsion\/ $\bW$\+locally contraherent
cosheaf\/ $\fP\in X\Lcth^\lct_\bW$, there exists an admissible
monomorphism\/ $\fP\rarrow\fJ$ in the exact category $X\Lcth^\lct_\bW$
(or equivalently, in $X\Lcth_\bW$) from\/ $\fP$ to a locally injective\/
$\bW$\+locally contraherent cosheaf\/ $\fJ\in X\Lcth^\lin_\bW$;
\item for every locally contraherent cosheaf\/ $\fP\in X\Lcth$, there
exists an admissible monomorphism\/ $\fP\rarrow\fJ$ in the exact
category $X\Lcth$ from\/ $\fP$ to a locally injective locally
contraherent cosheaf\/ $\fJ\in X\Lcth^\lin$;
\item for every locally cotorsion locally contraherent cosheaf\/
$\fP\in X\Lcth^\lct$, there exists an admissible monomorphism\/
$\fP\rarrow\fJ$ in the exact category $X\Lcth^\lct$ (or equivalently,
in $X\Lcth$) from\/ $\fP$ to a locally injective locally contraherent
cosheaf\/ $\fJ\in X\Lcth^\lin$.
\item for every contraherent cosheaf\/ $\fP\in X\Ctrh$, there exists
an admissible monomorphism\/ $\fP\rarrow\fJ$ in the exact category
$X\Lcth$ from\/ $\fP$ to a locally injective locally contraherent
cosheaf\/ $\fJ\in X\Lcth^\lin$;
\item for every locally cotorsion contraherent cosheaf\/
$\fP\in X\Ctrh^\lct$, there exists an admissible monomorphism\/
$\fP\rarrow\fJ$ in the exact category $X\Lcth^\lct$ (or equivalently,
in $X\Lcth$) from\/ $\fP$ to a locally injective locally contraherent
cosheaf\/ $\fJ\in X\Lcth^\lin$.
\end{enumerate}
\end{cor}

\begin{proof}
 (1)~$\Longrightarrow$~(2) This is~\cite[Section~2.4]{M-n}, 
\cite[Section~3.2]{M-th}, \cite[Lemma~A.1]{EP},
or~\cite[Lemma~4.1.8 or Corollary~4.1.11(a)]{Pcosh}
(see~\cite[Lemma~4.1.1 or Corollary~4.1.4(a)]{Pcosh} for
a stronger very flat version).

 (2)~$\Longrightarrow$~(1) This is~\cite[Theorem~2.2]{SS}.

 (3)~$\Longrightarrow$~(4) Holds because the full exact subcategory
$X\Ctrh^\lct$ is closed under cokernels of admissible monomorphisms
in $X\Ctrh$.

 (5)~$\Longrightarrow$~(6) Holds because the full exact subcategory
$X\Lcth^\lct_\bW$ is closed under cokernels of admissible monomorphisms
in $X\Lcth_\bW$.

 (7)~$\Longrightarrow$~(8) Holds because the full exact subcategory
$X\Lcth^\lct$ is closed under cokernels of admissible monomorphisms
in $X\Lcth$.

 (9)~$\Longrightarrow$~(10) Holds because the full exact subcategory
$X\Lcth^\lct$ is closed under cokernels of admissible monomorphisms
in $X\Lcth$.

%

 (3)~$\Longrightarrow$~(9) Holds due to the inclusions of exact
categories $X\Ctrh\subset X\Lcth$ and $X\Ctrh^\lin\subset X\Lcth^\lin$.

 (4)~$\Longrightarrow$~(10) Holds due to the inclusions of exact
categories $X\Ctrh^\lct\subset X\Lcth^\lct$
and $X\Ctrh^\lin\subset X\Lcth^\lin$.

 (5)~$\Longrightarrow$~(9) Holds due to the inclusions of exact
categories $X\Ctrh\subset X\Lcth_\bW\subset X\Lcth$
and $X\Lcth^\lin_\bW\subset X\Lcth^\lin$.

 (6)~$\Longrightarrow$~(10) Holds due to the inclusions of exact
categories $X\Ctrh^\lct\subset X\Lcth^\lct_\bW\subset X\Lcth^\lct$
and $X\Lcth^\lin_\bW\subset X\Lcth^\lin$.

 (7)~$\Longrightarrow$~(9) Holds due to the inclusion of
categories $X\Ctrh\subset X\Lcth$.

 (8)~$\Longrightarrow$~(10) Holds due to the inclusion of
categories $X\Ctrh^\lct\subset X\Lcth^\lct$.

 (1)~$\Longrightarrow$~(5) This is
Proposition~\ref{qcoh-ssep-enough-lin-prop}.

 (1)~$\Longrightarrow$~(3) This is
Proposition~\ref{qcoh-ssep-enough-lin-prop} for $\bW=\{X\}$.

 (1)~$\Longrightarrow$~(7) Follows from
Proposition~\ref{qcoh-ssep-enough-lin-prop}.

 (10)~$\Longrightarrow$~(1) This is
Theorem~\ref{non-ssep-not-enough-lin-theorem}.
\end{proof}

\Section{Coflasque Contraherent Cosheaves}  \label{coflasque-secn}

 We recall that a sheaf of abelian groups $\cM$ on a topological
space $X$ is said to be \emph{flasque}~\cite[Section~II.3.1]{God}
if the restriction map $\cM(U)\rarrow\cM(V)$ is surjective for any pair
of open subsets $V\subset U\subset X$.
 A cosheaf of abelian groups $\fP$ on $X$ is said to be
\emph{coflasque}~\cite[Section~3.4]{Pcosh}, \cite[Section~4.4]{Pform}
if the corestriction map $\fP[V]\rarrow\fP[U]$ is injective for all
$V\subset U\subset X$.

 The property of a sheaf of abelian groups $\cM$ on $X$ to be
flasque is \emph{local}, i.~e., it suffices to check it for
the restrictions of $\cM$ to open subsets forming any chosen open
covering of~$X$ \,\cite[Section~II.3.1]{God}.
 Using the construction of
Remark~\ref{finite-coverings-suffice-remarks}(a) in order to pass from
the cosheaves to the sheaves, one deduces the fact that the flasqueness
property of cosheaves of abelian groups on $X$ is similarly
local~\cite[Lemma~3.4.1(a)]{Pcosh}.

 Let us emphasize that, in the context of a scheme $X$, even for
quasi-coherent sheaves or contraherent cosheaves, it is \emph{not}
sufficient to check the surjectivity/injectivity condition for pairs
of \emph{affine} open subschemes $V\subset U\subset X$ in order to
establish the (co)flasqueness.
 See the discussion in~\cite[Example~3.4.3 with the preceding
paragraphs, and Remark~3.4.5]{Pcosh}.

\begin{lem} \label{coflasque-lcth-are-contraherent}
 Let $X$ be a scheme with an open covering\/~$\bW$.
 Then any coflasque\/ $\bW$\+locally contraherent cosheaf on $X$
is (globally) contraherent.
\end{lem}

\begin{proof}
 This is~\cite[Corollary~3.4.2]{Pcosh} or~\cite[Lemma~4.4.1]{Pform}.
 The assertion follows from the homological criterion of contraherence
of a locally contraherent cosheaf on an affine
scheme~\cite[Lemma~3.2.2]{Pcosh}, \cite[Proposition~3.6.1]{Pform}
together with the fact that the higher \v Cech homology of
a coflasque cosheaf with respect to any open covering
vanish~\cite[Th\'eor\`eme~II.5.2.3(a)]{God},
\cite[Lemma~3.4.1(b)]{Pcosh}.
 Once again, one can use the construction of the character sheaf from
Remark~\ref{finite-coverings-suffice-remarks}(a) above in order to pass
from a cosheaf to a sheaf in the latter assertion.
\end{proof}

\begin{lem} \label{coflasque-closedness-properties-and-cosections}
 Let $X$ be a scheme with an open covering\/~$\bW$.
 Then \par
\textup{(a)} the full subcategory of coflasque contraherent cosheaves
is closed under extensions in $X\Lcth_\bW$; \par
\textup{(b)} the full subcategory of coflasque contraherent cosheaves
is closed under kernels of admissible epimorphisms in $X\Lcth_\bW$; \par
\textup{(c)} for any (admissible) short exact sequence $0\rarrow\fP
\rarrow\fQ\rarrow\fE\rarrow0$ in $X\Lcth_\bW$ with a coflasque
contraherent cosheaf\/ $\fE$, and for any open subscheme $Y\subset X$,
the short sequence of abelian groups\/ $0\rarrow\fP[Y]\rarrow\fQ[Y]
\rarrow\fE[Y]\rarrow0$ is exact.
\end{lem}

\begin{proof}
 This is~\cite[Corollary~3.4.4]{Pcosh} or~\cite[Lemma~4.4.2]{Pform}.
 See~\cite{Pcosh} for a direct proof.
 Alternatively, one can notice that the construction of the character
sheaves from Remark~\ref{finite-coverings-suffice-remarks}(a) takes
any short exact sequence in $X\Lcth_\bW$ to a short exact sequence of
sheaves of $\cO_X$\+modules (since, for sheaves of $\cO_X$\+modules,
exactness of the short sequences of sections over affine open subschemes
subordinate to $\bW$ implies exactness of the short sequence of stalks)
in order to pass to the similar assertions for sheaves~\cite[Theorem
and Corollary~II.3.1.2]{God}.
\end{proof}

\begin{lem} \label{coflasque-resol-dim}
 Let $X$ be a scheme with an open covering\/~$\bW$.
 Assume that the underlying topological space of $X$ is Noetherian
of finite Krull dimension~$D$.
 Let\/ $0\rarrow\fE\rarrow\fE_{D-1}\rarrow\dotsb\rarrow\fE_1\rarrow
\fE_0\rarrow\fP\rarrow0$ be an exact sequence in the exact category
$X\Lcth_\bW$ such that the cosheaves\/ $\fE_i$ are coflasque for
all\/ $0\le i\le D-1$.
 Then the cosheaf\/ $\fE$ is coflasque as well.
\end{lem}

\begin{proof}
 This is~\cite[Lemma~3.4.7(b)]{Pcosh} or~\cite[Lemma~4.4.4(b)]{Pform}.
 The assertion is provable by passing from the cosheaves to
the sheaves using the construction of the character sheaves from
Remark~\ref{finite-coverings-suffice-remarks}(a), and then using
a suitable version of Grothendieck's vanishing
theorem~\cite[Th\'eor\`eme~3.6.5]{Toh}, \cite[Theorem~III.2.7]{HarAG},
\cite[Proposition Tag~02UZ]{SP}, which can be found in the proof
in~\cite{Pcosh}.
\end{proof}

 Given a scheme $X$, we denote by $X\Ctrh_\cfq\subset X\Ctrh$
the full subcategory of coflasque contraherent cosheaves on~$X$.
 We also put $X\Ctrh^\lct_\cfq=X\Ctrh_\cfq\cap X\Ctrh^\lct$; so
$X\Ctrh^\lct_\cfq$ is the full subcategory of coflasque locally
cotorsion contraherent cosheaves on~$X$.
 It follows from
Lemma~\ref{coflasque-closedness-properties-and-cosections}(a\+-b)
that the full subcategories $X\Ctrh_\cfq$ and $X\Ctrh^\lct_\cfq$
inherit exact category structures from the ambient exact categories
$X\Ctrh$ and $X\Ctrh^\lct$.

\begin{prop} \label{coflasque-direct-images}
 Let $f\:Y\rarrow X$ be a quasi-compact quasi-separated morphism of
schemes.
 Then \par
\textup{(a)} the functor of direct image of cosheaves of $\cO$\+modules
$f_!\:(Y,\cO_Y)\Cosh\rarrow(X,\cO_X)\Cosh$ takes coflasque
contraherent cosheaves on $Y$ to coflasque contraherent cosheaves on
$X$, and induces an exact functor between the respective exact 
categories $f_!\:Y\Ctrh_\cfq\rarrow X\Ctrh_\cfq$; \par
\textup{(b)} the functor of direct image of cosheaves of $\cO$\+modules
$f_!\:(Y,\cO_Y)\Cosh\rarrow(X,\cO_X)\Cosh$ takes coflasque locally
cotorsion contraherent cosheaves on $Y$ to coflasque locally cotorsion
contraherent cosheaves on $X$, and induces an exact functor between
the respective exact categories $f_!\:Y\Ctrh^\lct_\cfq\rarrow
X\Ctrh^\lct_\cfq$.
\end{prop}

\begin{proof}
 This is~\cite[Corollary~3.4.8]{Pcosh}
or~\cite[Corollary~4.4.5(b\+-c)]{Pform}.
 Similarly to Proposition~\ref{alf-direct-images} above, given
a coflasque contraherent cosheaf $\fE$ on $Y$, the assertion that
the cosheaf of $\cO_X$\+modules $f_!\fE$ on $X$ is contraherent
(or locally contraherent) is \emph{not} covered by the results
of Section~\ref{direct-image-and-contraherence-subsecn}, and
already requires a proof.

 All the conditions and properties assumed or claimed in the present
proposition are local in $X$, so one can assume the scheme $X$ to be
affine; then the scheme $Y$ is quasi-compact and quasi-separated.
 One first considers the case of a semi-separated scheme $Y$ before
passing to the general case when $Y$ is quasi-separated.
 The argument is based on the fact of exactness of the \v Cech
complexes of abelian groups/modules for coflasque
cosheaves~\cite[Lemma~3.4.1(b)]{Pcosh}, which was already mentioned
above in the proof of Lemma~\ref{coflasque-lcth-are-contraherent}.
 One also needs to observe that any module admitting a finite
(left) resolution by contraadjusted/cotorsion modules is
contraadjusted/cotorsion, and that the functor $\Hom_R(F,{-})$
for a very flat $R$\+module $F$ preserves exactness of sequences
of contraadjusted modules (over any commutative ring~$R$).
\end{proof}

\Section{Projective Locally Cotorsion Contraherent Cosheaves}
\label{noetherian-lct-prj-secn}

 We state the following theorem for Noetherian schemes only, rather
than for locally Noetherian schemes, as it is stated
in~\cite[Theorem~6.1.1]{Pcosh}.
 The reason is that the proof in the Noetherian (i.~e., quasi-compact)
case is simpler than in the general case of a locally Noetherian
scheme, and we only need quasi-compact schemes for the purposes of
the present paper, anyway.

 This is the contraherent dual analogue of Hartshorne's classification
of injective quasi-coherent sheaves on locally Noetherian
schemes~\cite[Proposition~II.7.17]{HarRD}.
 Just as the proof of Hartshorne's theorem is based on Matlis'
classification of injective modules over Noetherian commutative
rings~\cite[Section~3]{Mat}, the contraherent version of the theorem
is based on Enochs' classification of flat cotorsion modules over
Noetherian commutative rings~\cite[Section~2]{En2} (see
Proposition~\ref{flat-cotorsion-over-Noetherian-classified} above).

 The generalizations of these results to locally Noetherian formal
schemes can be found in~\cite[Section~4.5]{Pform} (for Hartshorne's
theorem) and~\cite[Section~4.6]{Pform} (for the contraherent
dual analogue).

 Recall the standard notation $\widetilde M$ for the quasi-coherent
sheaf on an affine scheme $\Spec R$ corresponding to an $R$\+module~$M$
(for any commutative ring~$R$).
 Similarly, we will denote by $\widecheck P$ the contraherent cosheaf
on $\Spec R$ corresponding to a contraadjusted $R$\+module $P$ under
the equivalence of categories from
Lemma~\ref{ctrh-cosheaves-on-affine-scheme}.

 Given a scheme $X$ and a scheme point $x\in X$, we denote by
$\cO_{x,X}$ the stalk of the structure sheaf $\cO_X$ at the point~$x$.
 So $\cO_{x,X}$ is a commutative local ring; let
$\m_{x,X}\subset\cO_{x,X}$ denote its maximal ideal.
 Furthermore, there is a natural morphism of schemes $\iota_x\:
\Spec\cO_{x,X}\rarrow X$.
 If the scheme $X$ is Noetherian, then so are the rings~$\cO_{x,X}$.
 The exposition in this section is based on the commutative algebra
preliminaries from Section~\ref{noetherian-commutative-rings-subsecn};
the definition and discussion of $I$\+contramodule (or
$\m$\+contramodule) $R$\+modules can be found there.

\begin{thm} \label{noetherian-lcth-lct-prj-theorem}
 Let $X$ be a (not necessarily semi-separated) Noetherian scheme with
an open covering\/~$\bW$.
 In this setting: \par
\textup{(a)} There are enough projective objects in the exact
categories of locally cotorsion locally contraherent cosheaves
$X\Ctrh^\lct\subset X\Lcth^\lct_\bW\subset X\Lcth^\lct$.
 All these projective objects belong to the full subcategory of
locally cotorsion contraherent cosheaves $X\Ctrh^\lct$, and
the classes of projective objects in the three exact categories
coincide. \par
\textup{(b)} A locally cotorsion contraherent cosheaf on $X$ is
projective if and only if it is isomorphic to the infinite direct
product\/ $\prod_{x\in X}\iota_x{}_!\widecheck F_x$ of the direct
images $\iota_x{}_!\widecheck F_x$ of contraherent cosheaves
$\widecheck F_x$ on\/ $\Spec\cO_{x,X}$ corresponding to some
free/projective/flat\/ $\m_{x,X}$\+contramodule
$\cO_{x,X}$\+modules.
\end{thm}

\begin{proof}
 This is a particular case of~\cite[Theorem~6.1.1]{Pcosh}; an even
more general result can be found in~\cite[Theorem~4.6.1]{Pform}.
 For a quasi-compact semi-separated scheme $X$, the assertions of
part~(a) were proved in Section~\ref{projective-on-qcss-subsecn};
see Lemma~\ref{qcss-lcth-lct-prj-lemma} and
Corollary~\ref{qcss-lcth-lct-prj-cor}.
 The argument below, based on
Proposition~\ref{flat-cotorsion-over-Noetherian-classified}, proves
parts~(a) and~(b) simultaneously for a Noetherian scheme~$X$.

 Recall that the classes of free, projective, and flat
$\m_{x,X}$\+contramodule $\cO_{x,X}$\+modules $F_x$ coincide by
Lemma~\ref{free-contramodules-are-flat-cotorsion}(b) and
the preceding discussion in
Section~\ref{noetherian-commutative-rings-subsecn} (where ``free''
and ``projective means ``free/projective as an $\m_{x,X}$\+contramodule
$\cO_{x,X}$\+module'', while ``flat'' means ``flat as
an $\cO_{x,X}$\+module'').
 Furthermore, all $\m_{x,X}$\+contramodule $\cO_{x,X}$\+modules are
cotorsion $\cO_{x,X}$\+modules by
Lemma~\ref{free-contramodules-are-flat-cotorsion}(a).
 Consequently, $\widecheck F_x$ is a projective locally cotorsion
contraherent cosheaf on $\Spec\cO_{x,X}$ by
Lemmas~\ref{ctrh-lct-lin-cosheaves-on-affine-scheme}(a)
and~\ref{qcss-lcth-lct-prj-lemma}.  {\hbadness=1400\par}

 What does the notation in the formulation of part~(b) of the theorem
mean?
 The infinite product over the points $x\in X$ is taken in the exact
category $X\Ctrh^\lct$ or $X\Lcth^\lct_\bW$, or equivalently, in
the exact category $X\Ctrh$ or $X\Lcth_\bW$, or equivalently, in
the additive category $(X,\cO_X)\Cosh$.
 Infinite direct products exist in all these categories and agree with
each other according to the discussion at the end of
Section~\ref{lct-lin-subsecn}.

 Notice, however, that any morphism from an affine scheme to
a semi-separated scheme is affine; but for a non-semi-separated
scheme this need not be the case.
 When the scheme $X$ is not semi-separated, the morphism
$\iota_x\:\Spec\cO_{x,X}\rarrow X$ is \emph{not} affine in general,
and the construction of the direct image functor acting between
the exact categories of contraherent
cosheaves~\eqref{ctrh-direct-image} or locally cotorsion contraherent
cosheaves~\eqref{ctrh-lct-direct-image} is \emph{not} applicable to it.
 So the assertion that the cosheaf of $\cO_X$\+modules
$\iota_x{}_!\widecheck F_x$ on $X$ belongs to $X\Ctrh^\lct$
requires a proof.

\begin{lem} \label{iota-direct-image-of-contamodules-lemma}
 Let $X$ be a Noetherian scheme, $x\in X$ be a point, and $P_x$
be an\/ $\m_{x,X}$\+contramodule $\cO_{x,X}$\+module.
 Then \par
\textup{(a)} for any affine open subscheme $U\subset X$ such that
$x\notin U$, one has $(\iota_x{}_!\widecheck P_x)|_U=0$; \par
\textup{(b)} the cosheaf of $\cO_X$\+modules
$\iota_x{}_!\widecheck P_x$ on $X$ is a locally cotorsion
contraherent cosheaf.
\end{lem}

\begin{proof}
 Part~(a): the point is that the module of cosections
$(\widecheck P_x)[V]$ vanishes for any open subscheme $V\subset
\Spec\cO_{x,X}$ that does not contain the closed point of
$\Spec\cO_{x,X}$.
 Indeed, it suffices to consider the case of a principal affine
open subscheme $V=\Spec\cO_{x,X}[s^{-1}]$, where $s\in\cO_{x,X}$
and the closed point of $\Spec\cO_{x,X}$ does not belong to~$V$.
 The latter condition means that $s\in\m_{x,X}$.
 So we need to check that the $\cO_{x,X}[s^{-1}]$\+module
$(\widecheck P_x)[V]=\Hom_{\cO_{x,X}}(\cO_{x,X}[s^{-1}],P_x)$
vanishes.
 By the definition, this holds for any
$\m_{x,X}$\+contramodule $\cO_{x,X}$\+module~$P_x$;
see Section~\ref{noetherian-commutative-rings-subsecn}.

 Part~(b): by the definition, it suffices to check that
$(\iota_x{}_!\widecheck P_x)|_U$ is a locally cotorsion contraherent
cosheaf on $U$ for any affine open subscheme $U\subset X$.
 Now the case when $x\notin U$ is covered by part~(a).
 When $x\in U$, the morphism~$\iota_x$ factorizes as
$\Spec\cO_{x,X}\rarrow U\rarrow X$.
 Denote by~$\kappa_x$ the morphism $\Spec\cO_{x,X}\rarrow U$.
 Then we have $(\iota_x{}_!\widecheck P_x)|_U=\kappa_x{}_!
\widecheck P_x$.
 Finally, $\kappa_x$~is an affine morphism (as any morphism of affine
schemes), and we can refer to formula~\eqref{ctrh-lct-direct-image}
for the assertion that the functor~$\kappa_x{}_!$ takes locally
cotorsion contraherent cosheaves on $\Spec\cO_{x,X}$ to locally
cotorsion contraherent cosheaves on~$U$.
\end{proof}

 We have shown that $\fF=\prod_{x\in X}\iota_x{}_!\widecheck F_x$ is
a locally cotorsion contraherent cosheaf on~$X$.
 Let us prove that it is a projective object of the exact category
$X\Lcth^\lct_\bW$.
 Let $X=\bigcup_{\alpha=1}^N U_\alpha$ be a finite affine open
covering of $X$ subordinate to~$\bW$.
 For every index $1\le\alpha\le N$, put $S_\alpha=U_\alpha\setminus
\bigcup_{\gamma=1}^{\alpha-1} U_\gamma$ (so $S_\alpha$ is a locally
closed subset in $X$, and the underlying topological space of $X$
is a disjoint union of its subsets~$S_\alpha$).
 Put $\fF_\alpha=\prod_{z\in S_\alpha}\iota_z{}_!\widecheck F_z$.
 Then we have $\fF=\bigoplus_{\alpha=1}^N\fF_\alpha$ in
$X\Ctrh^\lct$.

 Introduce the notation $j_\alpha\:U_\alpha\rarrow X$ for the identity
open immersion morphisms and $\kappa_{\alpha,z}\:\Spec\cO_{z,X}\rarrow
U_\alpha$ for the natural morphisms of schemes such that
$j_\alpha\circ\kappa_{\alpha,z}=\iota_z$.
 Then we have $\fF_\alpha=j_\alpha{}_!\fG_\alpha$, where
$\fG_\alpha=\prod_{z\in S_\alpha}\kappa_{\alpha,z}{}_!\widecheck F_z$.
 The observation that the functors of direct image of cosheaves
of $\cO$\+modules~$j_\alpha{}_!$ preserve infinite direct products
(since the morphisms~$j_\alpha$ are quasi-compact and quasi-separated)
is used here; see the end of
Section~\ref{direct-images-of-cosheaves-of-O-modules}.

 By Lemma~\ref{free-contramodules-are-flat-cotorsion} or
Proposition~\ref{flat-cotorsion-over-Noetherian-classified},
the $\cO_X(U_\alpha)$\+module $\fG_\alpha[U_\alpha]=
\prod_{z\in S_\alpha}F_z$ corresponding to the contraherent cosheaf
$\fG_\alpha$ on $U_\alpha$ is flat and cotorsion.
 By Lemma~\ref{qcss-lcth-lct-prj-lemma}, it follows that $\fG_\alpha$
is a projective locally cotorsion contraherent cosheaf on~$U_\alpha$.
 Now one can see from
the adjunction~\eqref{cosheaves-direct-image-restriction-adjunction}
together with the existence and exactness of the restriction
functor $j_\alpha^!\:X\Lcth^\lct_\bW\rarrow U_\alpha\Ctrh$
\,\eqref{restriction-of-lcth-lct-cosheaves} that
$\fF_\alpha=j_\alpha{}_!\fG_\alpha$ is a projective object of
the exact category $X\Lcth^\lct_\bW$.
 Thus $\fF=\bigoplus_{\alpha=1}^N\fF_\alpha$ is a projective locally
cotorsion $\bW$\+locally contraherent cosheaf on $X$, too.

 Now let us construct, for every locally cotorsion $\bW$\+locally
contraherent cosheaf $\fQ$ on $X$, an admissible epimorphism
$\fF\rarrow\fQ$ in $X\Lcth^\lct_\bW$, where
$\fF=\prod_{x\in X}\iota_x{}_!\widecheck F_x$ for some
free $\m_{x,X}$\+contramodule $\cO_{x,X}$\+modules~$F_x$.
 For this purpose, we consider a finite affine open covering
$X=\bigcup_{\alpha=1}^N U_\alpha$ as above; and for every~$\alpha$,
pick an admissible epimorphism $\fH_\alpha\rarrow j_\alpha^!\fQ$
in $U_\alpha\Ctrh^\lct$ with a projective locally cotorsion
contraherent cosheaf $\fH_\alpha$ on~$U_\alpha$.
 Such an admissible epimorphism exists by
Lemma~\ref{qcss-lcth-lct-prj-lemma} (for the affine scheme~$U_\alpha$).
 By adjunction~\eqref{cosheaves-direct-image-restriction-adjunction},
we have the related morphism of cosheaves of $\cO_X$\+modules
$\bigoplus_{\alpha=1}^N j_\alpha{}_!\fH_\alpha\rarrow\fQ$ on~$X$.

 We need to show that the cosheaf $\fF=\bigoplus_{\alpha=1}^N
j_\alpha{}_!\fH_\alpha$ has the desired form
$\fF\simeq\prod_{x\in X}\iota_x{}_!\widecheck F_x$ for some
free $\m_{x,X}$\+contramodule $\cO_{x,X}$\+modules~$F_x$.
 Indeed, by Proposition~\ref{flat-cotorsion-over-Noetherian-classified},
for every index~$\alpha$ we have
$\fH_\alpha\simeq\prod_{z\in U_\alpha}\kappa_{\alpha,z}{}_!
\widecheck H_{\alpha,z}$, where $H_{\alpha,z}$ are some free
$\m_{z,X}$\+contramodule $\cO_{z,X}$\+modules.
 According to the discussion above, it follows that
$j_\alpha{}_!\fH_\alpha\simeq\prod_{z\in U_\alpha}\iota_z{}_!
\widecheck H_{\alpha,z}$, hence
$\bigoplus_{\alpha=1}^N j_\alpha{}_!\fH_\alpha\simeq
\prod_{x\in X}\iota_x{}_!\widecheck F_x$, where
$F_x=\bigoplus_{1\le\alpha\le N}^{x\in U_\alpha} H_{\alpha,x}$
for every $x\in X$.

 Hence, in particular, $\fF$ is a (projective) locally cotorsion
contraherent cosheaf on $X$, as we have shown in the beginning of
this proof.
 Finally, similarly to the proof of
Lemma~\ref{qcss-lcth-lct-prj-lemma}, the morphism
$\bigoplus_{\alpha=1}^N j_\alpha{}_!\fH_\alpha\rarrow\fQ$ is
an admissible epimorphism in $X\Lcth^\lct_\bW$, since its restriction
to $U_\alpha$ is an admissible epimorphism in $U_\alpha\Ctrh^\lct$
for every~$\alpha$.
 The result of Lemma~\ref{colocality-of-epimorphisms}(b)
is relevant here.

 The preceding part of this proof is sufficient to establish
part~(a) of the theorem, and it also ``almost'' proves part~(b).
 In order to finish the proof of part~(b), it remains to show
that the class of cosheaves $\fF$ described in it is closed under
direct summands in $X\Ctrh^\lct$ (or equivalently, in
$(X,\cO_X)\Cosh$).
 For ease of reference, we formulate this part of the proof as
a separate lemma, but skip its proof, referring the reader to
the book manuscripts~\cite{Pcosh,Pform} instead.

\begin{lem} \label{proj-lct-direct-summand-lemma}
 Let $X$ be a Noetherian scheme, and\/ $\fP$ be a (locally cotorsion
contraherent) cosheaf of $\cO_X$\+modules on $X$ of the form\/
$\fP=\prod_{x\in X}\iota_x{}_!\widecheck P_x$, where $P_x$ are
some\/ $m_{x,X}$\+contramodule $\cO_{x,X}$\+modules.
 Then any direct summand\/ $\fQ$ of the cosheaf\/ $\fP$ is
isomorphic to the infinite product\/
$\prod_{x\in X}\iota_x{}_!\widecheck Q_x$, where, for every point
$x\in X$, the $\cO_{x,X}$\+module $Q_x$ is a direct summand of~$P_x$.
\end{lem}

\begin{proof}
 This is explained in~\cite[proof of Theorem~6.1.1]{Pcosh}, and,
in a greater generality, in~\cite[proof of Theorem~4.6.1]{Pform}.
 These proofs are based on the $\Hom$ semiorthogonality
lemmas, \cite[Lemma~6.1.2]{Pcosh} and~\cite[Lemma~4.6.2]{Pform}.
 Notice that the arguments in~\cite{Pcosh,Pform} are more complicated
than needed for the purposes of our present lemma, as those arguments
are applicable to arbitrary locally Noetherian (formal) schemes.
 For a Noetherian (i.~e., quasi-compact) scheme $X$, one application
of~\cite[Lemma~6.1.3]{Pcosh} or~\cite[Lemma~4.5.6]{Pform} is
sufficient, while in the locally Noetherian case, this abstract
category-theoretic lemma has to be applied twice.
\end{proof}

 Skipping the result of Lemma~\ref{proj-lct-direct-summand-lemma}
altogether, based on what is shown in the proof of the present theorem
above in this paper, one can assert that part~(a) of the theorem holds,
and in part~(b), a locally cotorsion contraherent cosheaf on $X$ is
projective if and only if it is a \emph{direct summand of} a cosheaf of
the desired form $\prod_{x\in X}\iota_x{}_!\widecheck F_x$ with
free $\m_{x,X}$\+contramodule $\cO_{x,X}$\+modules~$F_x$.
 Following the proofs below in this paper, the reader can easily see
that such a weak form of part~(b) is sufficient for our purposes in
the rest of this paper.
 (The only partial exception is
Corollary~\ref{lct-prj-local-coflasque}(c), whose proof we do not
spell out in detail in this paper, either.)
\end{proof}

 Following the notation of Section~\ref{projective-on-qcss-subsecn},
for any Noetherian scheme $X$, we denote by $X\Ctrh^\lct_\prj\subset
X\Ctrh^\lct$ the full subcategory of projective objects in $X\Ctrh^\lct$
(or equivalently, in $X\Lcth^\lct_\bW$, or in $X\Lcth^\lct$).
 The following corollary claims that being a projective locally
cotorsion contraherent cosheaf on a Noetherian scheme is a local
property.
 Furthermore, all such cosheaves are coflasque.

\begin{cor} \label{lct-prj-local-coflasque}
 Let $X$ be a Noetherian scheme. \par
\textup{(a)} Let $Y\subset X$ be an open subscheme.
 Then, for any cosheaf\/ $\fF\in X\Ctrh^\lct_\prj$, the cosheaf\/
$\fF|_Y$ belongs to $Y\Ctrh^\lct_\prj$. \par
\textup{(b)} In the context of part~\textup{(a)}, the corestriction
map\/ $\fF[Y]\rarrow\fF[X]$ is injective. \par
\textup{(c)} Let $X=\bigcup_\alpha Y_\alpha$ be an open covering
of~$X$.
 Then a locally contraherent cosheaf on $X$ belongs to
$X\Ctrh^\lct_\prj$ if and only if its restrictions to $Y_\alpha$
belong to $Y_\alpha\Ctrh^\lct_\prj$ for all indices\/~$\alpha$.
\end{cor}

\begin{proof}
 This is a partial version of~\cite[Corollary~6.1.4]{Pcosh}
or~\cite[Corollary~4.6.3]{Pform}.

 Part~(a): denote by $j\:Y\rarrow X$ the open immersion morphism.
 It is clear from Remarks~\ref{finite-coverings-suffice-remarks}(b,c,f)
(or from
the adjunction~\eqref{cosheaves-direct-image-restriction-adjunction})
that the functors of restriction of cosheaves of $\cO$\+modules to open
subschemes preserve infinite products.
 So, by Theorem~\ref{noetherian-lcth-lct-prj-theorem}(b), we have
$j^!\fF\simeq j^!\prod_{x\in X}\iota_x{}_!\widecheck F_x\simeq
\prod_{x\in X}j^!\iota_x{}_!\widecheck F_x$.
 By Lemma~\ref{iota-direct-image-of-contamodules-lemma}(a),
\,$j^!\iota_x{}_!\widecheck F_x=0$ for $x\notin Y$.
 Denote by~$\kappa_y$ the natural morphism of schemes
$\Spec\cO_{y,Y}\rarrow Y$ for a point $y\in Y$.
 Then we have $j^!\iota_y{}_!\widecheck F_y=\kappa_y{}_!\widecheck F_y$.
 Thus $\prod_{x\in X}j^!\iota_x{}_!\widecheck F_x\simeq
\prod_{y\in Y}\kappa_y{}_!\widecheck F_y$ is a projective locally
cotorsion contraherent cosheaf on $Y$ by
Theorem~\ref{noetherian-lcth-lct-prj-theorem}(b).
 
 Part~(b): by the definition, we have
$(\iota_x{}_!\widecheck F_x)[X]=F_x$.
 Since the cosection functors over quasi-compact quasi-separated
open subschemes preserve the infinite products of cosheaves of
$\cO$\+modules, we have $\fF[X]=\prod_{x\in X}F_x$ in the notation
above, and similarly $\fF[Y]=(j^!\fF)[Y]=\prod_{y\in Y}F_y$.
 So the $\cO(X)$\+module map $\fF[Y]\rarrow\fF[X]$ is actually
the inclusion of a direct summand.

 Part~(c): the ``only if'' implication holds by part~(a).
 To prove the ``if'', we assume without loss of generality that
the set of indices~$\{\alpha\}$ is finite (as the scheme $X$ is
quasi-compact in our setting), and proceed by induction.
 Then it suffices to prove the following assertion.

 Suppose $X$ is represented as the union of two open subschemes,
$X=W\cup Y$.
 Let $\fP$ be a locally contraherent cosheaf on~$X$.
 Suppose further that the restriction of $\fP$ onto $W$ is isomorphic
to the direct product $\prod_{w\in W}\lambda_w{}_!\widecheck P_w$,
where $P_w$ are some $\m_{w,X}$\+contramodule $\cO_{w,X}$\+modules,
while $\lambda_w\:\Spec\cO_{w,X}\rarrow W$ are the natural morphisms
of schemes.
 Similarly, suppose that the restriction of $\fP$ onto $Y$ is
isomorphic to the direct product $\prod_{y\in Y}\kappa_y{}_!
\widecheck Q_y$, where $Q_y$ are some $\m_{y,X}$\+contramodule
$\cO_{y,X}$\+modules.
 Then it is claimed that, for every point $v\in W\cap Y$, there is
an isomorphism of $\m_{v,X}$\+contramodule $\cO_{v,X}$\+modules
$P_v\simeq Q_v$.
 Furthermore, put $P_y=Q_y$ for all $y\in X\setminus W\subset Y$.
 Then the cosheaf $\fP$ on $X$ is isomorphic to
$\prod_{x\in X}\iota_x{}_!\widecheck P_x$.

 We suppress the rest of the technical argument proving
the assertion in question.
 It is based on the same Hom semiorthogonality lemma on which
the proof of Lemma~\ref{proj-lct-direct-summand-lemma} is based.
 A stronger version of argument can be found in~\cite[proof of
Corollary~6.1.4(c)]{Pcosh} or~\cite[proof of Corollary~4.6.3(d)]{Pform}
(see~\cite[proof of Corollary~4.5.7(d)]{Pform} for some additional
explanation).
\end{proof}

 We refer to Section~\ref{qcss-flat-subsecn} for the definitions
of flat and $\bW$\+flat cosheaves of $\cO$\+modules on schemes.

\begin{cor} \label{lct-prj=flat}
 Let $X$ be a Noetherian scheme with an open covering\/~$\bW$.
 Then the classes of projective locally cotorsion contraherent
cosheaves and\/ $\bW$\+flat locally cotorsion\/ $\bW$\+locally
contraherent cosheaves on $X$ coincide.
 Consequently, any\/ $\bW$\+flat locally cotorsion\/ $\bW$\+locally
contraherent cosheaf on $X$ is flat and contraherent.
\end{cor}

\begin{proof}
 This is a partial version of~\cite[Corollary~6.1.5]{Pcosh}
or~\cite[Corollary~4.6.4]{Pform}.
 All projective locally cotorsion contraherent cosheaves on $X$ are
flat by Corollary~\ref{lct-prj-local-coflasque}(a) (see also
Lemma~\ref{qcss-lcth-lct-prj-lemma}
and/or Corollary~\ref{antilocally-flat-are-flat}; recall that
all projective locally cotorsion contraherent cosheaves on
quasi-compact semi-separated schemes are antilocally flat as per
the discussion in Section~\ref{alf-subsecn}).
 To prove that every $\bW$\+flat locally cotorsion $\bW$\+locally
contraherent cosheaf $\fF$ on $X$ belongs to $X\Ctrh^\lct_\prj$,
choose an affine open covering $X=\bigcup_\alpha U_\alpha$ of
the scheme $X$ subordinate to~$\bW$.
 Then the $\cO_X(U_\alpha)$\+modules $\fF[U_\alpha]$ are flat
(since $\fF$ is $\bW$\+locally flat) and cotorsion (since $\fF$ is
$\bW$\+locally cotorsion).
 Hence the locally cotorsion contraherent cosheaves $\fF|_{U_\alpha}$
are projective by Lemma~\ref{qcss-lcth-lct-prj-lemma}, and it remains
to refer to Corollary~\ref{lct-prj-local-coflasque}(c).
\end{proof}

\Section{Main Positive Result}
\label{main-positive-result-secn}

 The exposition in this section largely
follows~\cite[Section~6.2]{Pcosh}, with an important improvement
achieved in~\cite[Section~4.7]{Pform} in the context of formal schemes.
 In the context of schemes, this improvement constitutes the proof
of the main positive result of this paper.

 We recall the notation $X\Lcth_\bW^\fl\subset X\Lcth_\bW$ and
$X\Ctrh^\fl\subset X\Ctrh$ for the full exact subcategories of
\emph{$\bW$\+flat\/ $\bW$\+locally contraherent cosheaves} and
\emph{flat contraherent cosheaves}
(see Section~\ref{qcss-flat-subsecn}).

\begin{lem} \label{ssep-lct-lct-prj-coresol-dim}
 Let $X$ be a semi-separated Noetherian scheme of Krull dimension $D$
with an open covering\/~$\bW$.
 In this context: \par
\textup{(a)} The full subcategory of locally cotorsion\/ $\bW$\+locally
contraherent cosheaves $X\Lcth^\lct_\bW$ is coresolving in the exact
category of\/ $\bW$\+locally contraherent cosheaves $X\Lcth_\bW$.
 The coresolution dimension of any object of $X\Lcth_\bW$ with respect
to the coresolving subcategory $X\Lcth^\lct_\bW$ does not exceed~$D$.
\par
\textup{(b)} The full subcategory of projective locally cotorsion
contraherent cosheaves $X\Ctrh^\lct_\prj$ is coresolving in
the exact category of\/ $\bW$\+flat\/ $\bW$\+locally contraherent
cosheaves $X\Lcth_\bW^\fl$.
 The coresolution dimension of any object of $X\Lcth_\bW^\fl$ with
respect to the coresolving subcategory $X\Ctrh^\lct_\prj$ does not
exceed~$D$.
\end{lem}

\begin{proof}
 This is~\cite[Lemma~6.2.1]{Pcosh} and a particular case
of~\cite[Lemma~4.7.1]{Pform}.

 Part~(a): for any scheme $X$ with an open covering $\bW$, the full
subcategory $X\Lcth^\lct_\bW\subset X\Lcth_\bW$ is closed under
extensions and cokernels of admissible monomorphisms according to
Section~\ref{lct-lin-subsecn}.
 For a quasi-compact semi-separated scheme $X$, the same full
subcategory is cogenerating by
Proposition~\ref{qcoh-ssep-enough-lin-prop},
Proposition~\ref{qcoh-ssep-lct-preenvelope-prop},
or Corollary~\ref{antilocally-flat-cotorsion-pair}(a).
 For a semi-separated Noetherian scheme, the coresolution dimension
bound follows from Theorem~\ref{raynaud-gruson-theorem}(b).

 Part~(b): one has $X\Ctrh^\lct_\prj=X\Lcth_\bW^\fl\cap
X\Lcth^\lct_\bW$ for any Noetherian scheme $X$ by
Corollary~\ref{lct-prj=flat}.
 Let us show that the full subcategory $X\Lcth^\lct_\prj$ is
cogenerating in $X\Lcth_\bW^\fl$ when the scheme $X$ is semi-separated
and Noetherian.

 Let $\fG$ be a $\bW$\+flat $\bW$\+locally contraherent cosheaf on~$X$.
 By Corollary~\ref{antilocally-flat-cotorsion-pair}(a), there exists
an (admissible) short exact sequence $0\rarrow\fG\rarrow\fP
\rarrow\fF\rarrow0$ in $X\Lcth_\bW$ with a locally cotorsion
$\bW$\+locally contraherent cosheaf $\fP$ and an antilocally flat
contraherent cosheaf $\fF$ on~$X$.
 By Corollary~\ref{antilocally-flat-are-flat}, the cosheaf $\fF$ is
flat.
 As the full subcategory $X\Lcth_\bW^\fl$ is closed under extensions
in $X\Lcth_\bW$, it follows that $\fP\in X\Lcth_\bW^\fl\cap
X\Lcth^\lct_\bW=X\Ctrh^\lct_\prj$.
 So $0\rarrow\fG\rarrow\fP\rarrow\fF\rarrow0$ is a short exact
sequence in $X\Lcth_\bW^\fl$ with $\fP\in X\Ctrh^\lct_\prj$,
as desired.

 The rest of the assertions of part~(b) follows purely formally
from the respective assertions of part~(a) by virtue
of~\cite[Lemma~A.1.1(b) and Corollary~A.1.3(b)]{Pform} or the dual
version of~\cite[Corollary~A.5.5]{Pcosh}.
\end{proof}

 A detailed discussion of \emph{antilocally flat} ($\bW$\+locally)
contraherent cosheaves was presented in Section~\ref{alf-subsecn}.
 For the definitions and discussion of the full exact subcategories
of \emph{coflasque} contraherent cosheaves $X\Ctrh_\cfq\subset X\Ctrh$
and coflasque locally cotorsion contraherent cosheaves 
$X\Ctrh^\lct_\cfq\subset X\Ctrh^\lct$ we refer to
Section~\ref{coflasque-secn}.

\begin{cor} \label{flat-are-alf-contraherent-and-coflasque}
\textup{(a)} On a semi-separated Noetherian scheme of finite Krull
dimension with an open covering\/ $\bW$, the classes of\/
$\bW$\+flat\/ $\bW$\+locally contraherent cosheaves and antilocally
flat contraherent cosheaves coincide. \par
\textup{(b)} On a Noetherian scheme of finite Krull dimension with
an open covering\/ $\bW$, any\/ $\bW$\+flat\/ $\bW$\+locally
contraherent cosheaf is flat and contraherent. \par
\textup{(c)} On a Noetherian scheme of finite Krull dimension,
all flat contraherent cosheaves are coflasque.
\end{cor}

\begin{proof}
 This is~\cite[Corollary~6.2.2]{Pcosh} and a particular case
of~\cite[Corollary~4.7.2]{Pform}.

 Part~(a): all antilocally flat contraherent cosheaves on
a semi-separated Noetherian scheme $X$ are flat by
Corollary~\ref{antilocally-flat-are-flat}.
 Conversely, if the Krull dimension of $X$ is finite, then
Lemma~\ref{ssep-lct-lct-prj-coresol-dim}(b) tells us that any
$\bW$\+flat $\bW$\+locally contraherent cosheaf $\fF$ on $X$ has
a finite coresolution by projective locally cotorsion contraherent
cosheaves in the exact category $X\Lcth_\bW^\fl$.
 All projective locally cotorsion contraherent cosheaves are
antilocally flat by definition.
 Applying Corollary~\ref{qcss-alf-characterizations}(b) iteratively,
we conclude that the cosheaf $\fF$ is antilocally flat.

 Part~(b): by the definition, both the flatness and contraherence
of locally contraherent cosheaves are checked over affine open
subschemes.
 For a $\bW$\+flat $\bW$\+locally contraherent cosheaf $\fF$ on $X$
and an affine open subscheme $U\subset X$, the cosheaf $\fF|_U$
is $\bW|_U$\+flat and $\bW|_U$\+locally contraherent, where
$\bW|_U=\{U\cap W\mid W\in\bW\}$ is the restriction to $U$ of
the open covering $\bW$ of~$X$.
 This reduces the question to the case of an affine Noetherian
scheme~$X$ (with a nontrivial open covering!), and this case
is covered by part~(a).

 Part~(c): coflasqueness of cosheaves is a local property
by~\cite[Lemma~3.4.1(a)]{Pcosh} (see some details in the beginning of
Section~\ref{coflasque-secn}), and flatness of cosheaves is preserved
by restrictions to open subschemes by the definition.
 So once again it suffices to consider the case of an affine
Noetherian scheme~$X$.
 Then we know that any flat contraherent cosheaf $\fF$ on $X$ has
a finite coresolution by projective locally cotorsion contraherent
cosheaves.
 All projective locally cotorsion contraherent cosheaves on
a Noetherian scheme are coflasque by
Corollary~\ref{lct-prj-local-coflasque}(b).
 Applying Lemma~\ref{coflasque-closedness-properties-and-cosections}(b)
iteratively, we see that the cosheaf $\fF$ is coflasque.
\end{proof}

 In the next lemma, the semi-sepatedness assumption on $Y$ can be
dropped, and an open immersion~$j$ can be replaced by an arbitrary
flat morphism~$f$; see~\cite[Corollary~6.2.12(b)]{Pcosh} or
the even more general~\cite[Corollary~4.7.16(b)]{Pform}.
 We only state the very special case that is needed for our purposes
in this section.

\begin{lem} \label{flat-direct-images}
 Let $X$ be a Noetherian scheme, $Y$ be a semi-separated Noetherian
scheme of finite Krull dimension, and $j\:Y\rarrow X$ be an open
immersion of schemes.
 Then the functor of direct image of cosheaves of $\cO$\+modules
$j_!\:(Y,\cO_Y)\Cosh\rarrow(X,\cO_X)\Cosh$ takes flat contraherent
cosheaves on $Y$ to flat contraherent cosheaves on $X$, and induces
an exact functor between the respective exact categories
$j_!\:Y\Ctrh^\fl\rarrow X\Ctrh^\fl$.
\end{lem}

\begin{proof}
 This is~\cite[Lemma~6.2.3]{Pcosh} or~\cite[Lemma~4.7.3]{Pform}
(with the generality level restricted to open immersions of schemes).
 Since contraherence and flatness are checked over affine open
subschemes, it suffices to consider the case of an affine Noetherian
scheme~$X$.
 By Corollary~\ref{flat-are-alf-contraherent-and-coflasque}(a),
all flat contraherent cosheaves on $Y$ are antilocally flat.
 According to Proposition~\ref{alf-direct-images}, we have an exact
direct image functor $j_!\:Y\Ctrh_\alf\rarrow X\Ctrh_\alf$.
 By Corollary~\ref{antilocally-flat-are-flat} (or
Lemma~\ref{antilocally-flat-basic-properties}(a) and
Proposition~\ref{flat-contraadjusted-colocalization}), all antilocally
flat contraherent cosheaves on $X$ are flat.
\end{proof}

\begin{cor} \label{noetherian-lcth-prj-cor}
 Let $X$ be a Noetherian scheme with an open covering\/~$\bW$.
 Let $X=\bigcup_{\alpha=1}^N U_\alpha$ be a finite affine open
covering of $X$ subordinate to\/~$\bW$. \par
\textup{(a)} There are enough projective objects in the exact
categories of locally contraherent cosheaves $X\Ctrh\subset
X\Lcth_\bW\subset X\Lcth$.
 All these projective objects belong to the full subcategory of
contraherent cosheaves $X\Ctrh$, and the classes of projective objects
in the three exact categories coincide. \par
\textup{(b)} A\/ $\bW$\+locally contraherent cosheaf on $X$ is
projective if and only if it is a direct summand of a finite direct
sum of the direct images of projective contraherent cosheaves
from~$U_\alpha$. \par
\textup{(c)} All projective contraherent cosheaves on $X$ are flat
(and consequently, coflasque). \hbadness=1450
\end{cor}

\begin{proof}
 This is~\cite[Corollary~6.2.4]{Pcosh} and a particular case
of~\cite[Corollary~4.7.4]{Pform}.
 The argument is similar to the proofs of
Lemma~\ref{qcss-lcth-prj-lemma} and
Corollary~\ref{qcss-lcth-prj-cor}.
 The difference is that the scheme $X$ need not be semi-separated now,
so the open immersion morphisms $j_\alpha\:U_\alpha\rarrow X$ need not
be affine.
 That is why we need Lemma~\ref{flat-direct-images}.

 Let us elaborate a little bit on this difference.
 \emph{Unlike} in Lemma~\ref{qcss-lcth-prj-lemma}, the restriction
functors $j_\alpha^!\:X\Lcth_\bW\rarrow U_\alpha\Ctrh$ \emph{do not}
seem to have left adjoint functors in the context of the present
corollary.
 See a counterexample in~\cite[Remark~2.3.1 and Example~2.3.2]{Pcosh}.
 Rather, Lemma~\ref{flat-direct-images} and the adjunction
formula~\eqref{cosheaves-direct-image-restriction-adjunction}
provide a \emph{partially defined} functor $U_\alpha\Ctrh\supset
U_\alpha\Ctrh^\fl\overset{j_\alpha{}_!}\rarrow X\Ctrh^\fl\subset
X\Lcth_\bW$ that is \emph{partially left adjoint} to
the functor~$j_\alpha^!$, in an obvious sense.
 Cf.\ the discussion of partially defined functors between exact
categories in~\cite[Section~5.6]{Pphil} and~\cite[Section~1.10
of the Introduction]{Pcosh}.
 Nevertheless, it follows from
the adjunction~\eqref{cosheaves-direct-image-restriction-adjunction}
that the functor $j_\alpha{}_!\:U_\alpha\Ctrh^\fl\rarrow X\Ctrh^\fl$
from Lemma~\ref{flat-direct-images} takes the projective objects
of $U_\alpha\Ctrh$ (which all belong to $U_\alpha\Ctrh^\fl$ by
Corollary~\ref{antilocally-flat-are-flat} or
Lemma~\ref{qcss-lcth-prj-lemma} with
Proposition~\ref{flat-contraadjusted-colocalization}) to projective
objects in $X\Lcth_\bW$.

 This suffices to establish all the claims in~(a\+-c).
 The parenthetical assertion in part~(c) follows by
Corollary~\ref{flat-are-alf-contraherent-and-coflasque}(c).
\end{proof}

 Following the notation of Section~\ref{projective-on-qcss-subsecn},
for any Noetherian scheme $X$ of finite Krull dimension, we denote by 
$X\Ctrh_\prj\subset X\Ctrh$ the full subcategory of projective objects
in $X\Ctrh$ (or equivalently, in $X\Lcth_\bW$, or in $X\Lcth$).

\begin{prop} \label{coflasque-lct-preenvelope}
 Let $X$ be a Noetherian scheme of finite Krull dimension.
 Let $X=\bigcup_{\alpha=1}^N U_\alpha$ be a finite affine open
covering of~$X$.
 Then, for any coflasque contraherent cosheaf\/ $\fE$ on $X$,
there exists an (admissible) short exact sequence\/ $0\rarrow\fE
\rarrow\fQ\rarrow\fF\rarrow0$ in the exact category $X\Ctrh_\cfq$
with a coflasque locally cotorsion contraherent cosheaf\/ $\fQ$
on $X$ and a contraherent cosheaf\/ $\fF$ on $X$ that is a finitely
iterated extension of the direct images of flat contraherent
cosheaves from~$U_\alpha$.
\end{prop}

\begin{proof}
 This is~\cite[Lemma~6.2.5(a)]{Pcosh} and a particular case
of~\cite[Lemma~4.7.5(a)]{Pform}.
 The argument is similar to the proofs of
Propositions~\ref{qcoh-ssep-enough-lin-prop}
and~\ref{qcoh-ssep-lct-preenvelope-prop}.
 Once again, the only difference is that the scheme $X$ need not be
semi-separated now, so the open immersion morphisms
$j_\alpha\:U_\alpha\rarrow X$ need not be affine.
 For this reason, one needs to perform the whole construction within
the realm of coflasque contraherent cosheaves, and use
Proposition~\ref{coflasque-direct-images}(a\+-b).
 It is important that all flat contraherent cosheaves on $X$ are
coflasque by Corollary~\ref{flat-are-alf-contraherent-and-coflasque}(c).

 In particular, for an affine scheme $U=U_\beta$ and a coflasque
contraherent cosheaf $\fN$ on $U$, one uses
Lemmas~\ref{ctrh-cosheaves-on-affine-scheme}
and~\ref{ctrh-lct-lin-cosheaves-on-affine-scheme}(a) together with
Theorem~\ref{flat-cotorsion-pair}(b) in order to construct
a short exact sequence $0\rarrow\fN\rarrow\fR\rarrow\fG\rarrow0$
in $U\Ctrh$ with a flat contraherent cosheaf~$\fG$
(Proposition~\ref{flat-contraadjusted-colocalization} or Corollary~\ref{antilocally-flat-are-flat} is also relevant here).
 Then the contraherent cosheaf $\fG$ on $U$ is coflasque by
Corollary~\ref{flat-are-alf-contraherent-and-coflasque}(c), and
it follows that the contraherent cosheaf $\fR$ is also coflasque
by Lemma~\ref{coflasque-closedness-properties-and-cosections}(a).
\end{proof}

\begin{cor} \label{flat-lct-prj-preenvelope}
 Let $X$ be a Noetherian scheme of finite Krull dimension.
 Let $X=\bigcup_{\alpha=1}^N U_\alpha$ be a finite affine open
covering of~$X$.
 Then, for any flat contraherent cosheaf\/ $\fG$ on $X$, there exists
a short exact sequence\/ $0\rarrow\fG\rarrow\fQ\rarrow\fF\rarrow0$ in
the exact category $X\Ctrh^\fl$ with a projective locally cotorsion
contraherent cosheaf\/ $\fQ$ on $X$ and a contraherent cosheaf\/ $\fF$
on $X$ that is a finitely iterated extension of the direct images of
flat contraherent cosheaves from~$U_\alpha$.
\end{cor}

\begin{proof}
 This is~\cite[Corollary~6.2.6(a)]{Pcosh} and a particular case
of~\cite[Corollary~4.7.6(a)]{Pform}.
 By Corollary~\ref{flat-are-alf-contraherent-and-coflasque}(c),
the cosheaf $\fG$ is coflasque.
 Hence Proposition~\ref{coflasque-lct-preenvelope} provides
a short exact sequence $0\rarrow\fG\rarrow\fQ\rarrow\fF\rarrow0$
with a coflasque locally cotorsion contraherent cosheaf $\fQ$
and a contraherent cosheaf $\fF$ of the desired form.
 Now the cosheaf $\fF$ is flat by Lemma~\ref{flat-direct-images},
and it follows that the cosheaf $\fQ$ is flat as an extension of
two flat contraherent cosheaves.
 Using Corollary~\ref{lct-prj=flat}, we can conclude that $\fQ$ is
a projective locally cotorsion contraherent cosheaf.
\end{proof}

 We recall that an exact category $\sE$ is said to have
\emph{homological dimension\/~$\le d$} (where $d\ge-1$ is an integer)
if $\Ext_\sE^{d+1}(X,Y)=0$ for all $X$, $Y\in\sE$.

\begin{cor} \label{exact-category-of-flats-cor}
 Let $X$ be a Noetherian scheme of finite Krull dimension~$D$.
 Then the homological dimension of the exact category of flat
contraherent cosheaves $X\Ctrh^\fl$ does not exceed~$D$.
 There are both enough projective and enough injective objects in
the exact category $X\Ctrh^\fl$; the full subcategory of projective
objects is $X\Ctrh_\prj$, and the full subcategory of injective objects
is $X\Ctrh^\lct_\prj$.
\end{cor}

\begin{proof}
 This is~\cite[Corollary~6.2.6(b)]{Pcosh} and a particular case
of~\cite[Corollary~4.7.6(b)]{Pform}.
 The assertions that $X\Ctrh_\prj$ is the full subcategory of
projective objects in $X\Ctrh^\fl$ and there are enough such
projective objects follow from
Corollary~\ref{noetherian-lcth-prj-cor}(a,c).

 The full subcategory $X\Ctrh^\lct_\prj$ is cogenerating in
$X\Ctrh^\fl$ by Corollary~\ref{flat-lct-prj-preenvelope}.
 This full subcategory is also obviously closed under extensions;
and it is closed under cokernels of admissible monomorphisms
since $X\Ctrh^\lct_\prj=X\Ctrh^\fl\cap X\Ctrh^\lct$ by
Corollary~\ref{lct-prj=flat} and the full subcategory $X\Ctrh^\lct$
is closed under cokernels of admissible monomorphisms in $X\Ctrh$
by Section~\ref{lct-lin-subsecn}.
 So $X\Ctrh^\lct_\prj$ is a coresolving subcategory in $X\Ctrh^\fl$.
 Since $X\Ctrh^\lct_\prj$ is a spit exact category, it follows
by virtue of Lemma~\ref{split-co-resolving-lemma}(b) that 
$X\Ctrh^\lct_\prj$ is the full subcategory of injective objects
in $X\Ctrh^\fl$ and there are enough such injective objects.

 Finally, the coresolution dimensions of all the objects of
$X\Ctrh^\fl$ with respect to $X\Ctrh^\lct_\prj$ do not exceed~$D$
by Corollary~\ref{lct-prj=flat} and
Theorem~\ref{raynaud-gruson-theorem}(b)
(cf.\ the proof of Lemma~\ref{ssep-lct-lct-prj-coresol-dim}).
 As $X\Ctrh^\lct_\prj$ is the full subcategory of injective objects
and there are enough of them, this means that the homological
dimension of the exact category $X\Ctrh^\fl$ does not exceed~$D$.
\end{proof}

\begin{lem} \label{flat-explicitly-described-precover}
 Let $X$ be a Noetherian scheme of finite Krull dimension.
 Let $X=\bigcup_{\alpha=1}^N U_\alpha$ be a finite affine open
covering of~$X$.
 Then, for any flat contraherent cosheaf\/ $\fG$ on $X$, there exists
a short exact sequence\/ $0\rarrow\fQ\rarrow\fF\rarrow\fG\rarrow0$ in
the exact category $X\Ctrh^\fl$ with a projective locally cotorsion
contraherent cosheaf\/ $\fQ$ on $X$ and a contraherent cosheaf\/ $\fF$
on $X$ that is a finitely iterated extension of the direct images of
flat contraherent cosheaves from~$U_\alpha$.
\end{lem}

\begin{proof}
 This is~\cite[Lemma~6.2.7(b)]{Pcosh} and a particular case
of~\cite[Lemma~4.7.7(b)]{Pform}.
 The proof is similar to that of Lemma~\ref{qcoh-ssep-alf-precover-lem}.
 The assertion follows from Corollaries~\ref{flat-lct-prj-preenvelope}
and~\ref{noetherian-lcth-prj-cor}(a\+-c) by virtue of
Lemma~\ref{salce-lemma}(a).
\end{proof}

 Let $X$ be a Noetherian scheme of finite Krull dimension.
 The next corollary mentions the $\Ext$ groups in the exact categories
of locally contraherent cosheaves on~$X$.
 One would like to claim that all such categories we are interested in
are either resolving or coresolving in each other, and therefore
the $\Ext$ groups computed in them agree, similarly to the discussion
in Section~\ref{alf-subsecn}.
 However, we do not know that yet, as the assertion that the full
subcategory $X\Lcth_\bW^\lct$ is coresolving (or cogenerating) in
$X\Lcth_\bW$ will be only proved later in this section.
 (While, on the other hand, the assertion that the full subcategories
$X\Ctrh\subset X\Lcth_\bW\subset X\Lcth$ are resolving in each other
is provable using Corollary~\ref{noetherian-lcth-prj-cor}(a).)

 One possible way to tackle this difficulty is to use the derived
categories in order to prove that all the $\Ext$ groups in the exact
categories in question agree, as in~\cite[Corollary~6.2.8]{Pcosh}
or~\cite[Corollary~4.7.8]{Pform}.
 In order to avoid referring to the derived categories in this paper,
we settle on the following compromise.

 All the exact categories of (locally) contraherent cosheaves that
we are interested in are obviously closed under extensions in
each other.
 Hence the groups $\Ext^1$ computed in them agree.
 Based on this observation, let us use the notation
$\Ext^{X,1}({-},{-})$ for the groups $\Ext^1$ in any one of
the exact categories $X\Lcth$, $X\Lcth^\lct$, $X\Lcth_\bW$, 
$X\Lcth_\bW^\lct$, $X\Ctrh_\cfq$, $X\Ctrh^\lct_\cfq$,
or $X\Ctrh^\fl$.

\begin{cor} \label{flat-are-antilocally-flat}
 Let $X$ be a Noetherian scheme of finite Krull dimension with an open
covering\/~$\bW$.
 Let $X=\bigcup_{\alpha=1}^N U_\alpha$ be a finite affine open
covering of $X$ subordinate to\/~$\bW$. \par
\textup{(a)} One has\/ $\Ext^{X,1}(\fG,\fQ)$ for any flat
contraherent cosheaf\/ $\fG$ and any locally cotorsion\/
$\bW$\+locally contraherent cosheaf\/ $\fQ$ on~$X$.
 Consequently, all flat contraherent cosheaves on $X$ are
antilocally flat (with respect to any open covering\/~$\bW$). \par
\textup{(b)} A contraherent cosheaf\/ $\fG$ on $X$ is flat if and
only if\/ $\fG$ is a direct summand of a finitely iterated extension
of the direct images of flat contraherent cosheaves from~$U_\alpha$.
\end{cor}

\begin{proof}
 This is a partial version of~\cite[Corollary~7.2.9]{Pcosh} or of
the even more general result of~\cite[Corollary~4.7.9]{Pform}.

 Part~(b): this is similar to Corollary~\ref{antilocally-flat-cor}.
 The easy ``if'' implication holds by
Lemma~\ref{flat-direct-images}.
 To prove the ``only if'', consider a short exact sequence
$0\rarrow\fQ\rarrow\fF\rarrow\fG\rarrow0$ from
Lemma~\ref{flat-explicitly-described-precover}.
 This is a short exact sequence in the exact category $X\Ctrh^\fl$,
and $\fQ$ is an injective object of $X\Ctrh^\fl$ by
Corollary~\ref{exact-category-of-flats-cor}.
 Hence the short exact sequence splits and $\fG$ is a direct summand
of~$\fF$.

 Part~(a): we follow the idea of~\cite[alternative proof of
Corollary~4.7.9(a)]{Pform}.
 In view of part~(b), it suffices to show that
$\Ext^{X,1}(j_!\fH,\fQ)=0$ for any affine open subscheme
$U\subset X$ subordinate to $\bW$ with the open immersion morphism
$j\:U\rarrow X$, and any flat contraherent cosheaf $\fH$ on~$U$.
 The argument is based on $\Ext^1$\+adjunction considerations
similar to~\cite[Lemma~1.7(d) or the opposite version of
Lemma~1.7(b)]{Pal}.

 Let $0\rarrow\fQ\rarrow\fM\rarrow j_!\fH\rarrow0$ be an (admissible)
short exact sequence in $X\Lcth_\bW$.
 We need to show that this short exact sequence splits.
 Applying the exact functor $j^!$
\,\eqref{restriction-of-lcth-cosheaves}, we obtain a short exact
sequence $0\rarrow j^!\fQ\rarrow j^!\fM\rarrow\fH\rarrow0$
in $U\Ctrh$ with $j^!\fQ\in X\Ctrh^\lct$
(by~\eqref{restriction-of-lcth-lct-cosheaves})
and $\fH\in X\Ctrh^\fl$.

 The equivalence of exact categories from
Lemma~\ref{ctrh-cosheaves-on-affine-scheme} transforms the latter
short exact sequence into a short exact sequence of contraadjusted
$\cO_X(U)$\+modules $0\rarrow\fQ[U]\rarrow\fM[U]\rarrow\fH[U]\rarrow0$
with a flat $\cO_X(U)$\+module $\fH[U]$ and a cotorsion
$\cO(U)$\+module~$\fQ[U]$.
 By the definition of a cotorsion module (see
Section~\ref{cotorsion-modules-subsecn}), this short exact sequence
of $\cO_X(U)$\+modules splits.
 Hence the short exact sequence $0\rarrow j^!\fQ\rarrow j^!\fM
\rarrow\fH\rarrow0$ in $U\Ctrh$ splits as well, and we can choose
a section $\fH\rarrow j^!\fM$ for the split epimorphism
$j^!\fM\rarrow\fH$.

 By the adjunction isomorphism of abelian
groups~\eqref{cosheaves-direct-image-restriction-adjunction},
the morphism $\fH\rarrow j^!\fM$ in $U\Ctrh$ corresponds to
a morphism $j_!\fH\rarrow\fM$ in $X\Lcth_\bW$.
 We leave it to the reader to check that the latter morphism
is a section for the admissible epimorphism $\fM\rarrow j_!\fH$
in $X\Lcth_\bW$.
\end{proof}

\begin{rem}
 In the following lemma, proposition, and corollary, our argument
is a version of~\cite[Theorem~1.1]{AB}.
 Before spelling out our own argument, let us point out that, in
the notation of~\cite{AB}, we need to consider an exact (rather
than abelian) category $\bC=X\Lcth_\bW$, the full subcategory
$\bX=X\Ctrh_\cfq\subset X\Lcth_\bW$, and the full subcategory
$\bomega=X\Ctrh^\lct_\cfq\subset X\Ctrh_\cfq=\bX$ (which is
cogenerating in $\bX$ by Proposition~\ref{coflasque-lct-preenvelope}).
 Then, in the notation of~\cite{AB}, we have $\widehat\bX=\bC$
(by Corollary~\ref{noetherian-lcth-prj-cor}(a,c) and
Lemma~\ref{coflasque-resol-dim}) and $\widehat\bomega=X\Lcth^\lct_\bW$
(by Theorem~\ref{noetherian-lcth-lct-prj-theorem}(a),
Corollary~\ref{lct-prj-local-coflasque}(b), and
Lemma~\ref{coflasque-resol-dim}; one also needs to use the fact
that the full subcategory $X\Lcth^\lct_\bW$ is closed under
cokernels of admissible monomorphisms in $X\Lcth_\bW$).

 Our argument below proves a bit more than the theorem from~\cite{AB}
in our context, in that it provides approximation
sequences~(\ref{special-precover-sequence}\+-%
\ref{special-preenvelope-sequence}) with flat (rather than
merely coflasque) objects $F$ and~$F'$.
 See Proposition~\ref{noetherian-lcth-lct-preenvelope} and
Corollary~\ref{noetherian-lcth-flat-precover}.
\end{rem}

\begin{lem} \label{approximable-closed-under-cokers-of-admmonos}
 Let $X$ be a Noetherian scheme of finite Krull dimension with
an open covering\/~$\bW$.
 Consider the class of all\/ $\bW$\+locally contraherent cosheaves\/
$\fM$ on $X$ for which there exists an (admissible) short exact
sequence\/ $0\rarrow\fM\rarrow\fQ\rarrow\fF\rarrow0$ in $X\Lcth_\bW$
with a locally cotorsion\/ $\bW$\+locally contraherent cosheaf\/
$\fQ\in X\Lcth^\lct_\bW$ and a flat contraherent cosheaf\/
$\fF\in X\Ctrh^\fl$.
 Then the class of all such objects\/ $\fM$ is closed under cokernels
of admissible monomorphisms in $X\Lcth_\bW$.
\end{lem}

\begin{proof}
 We skip the details of the argument, which is essentially the same
as the proof of the more general~\cite[Lemma~4.7.10]{Pform}.
 One needs to use Corollary~\ref{flat-are-antilocally-flat}(a)
for constructing a morphism of the given short exact sequences, and
Corollary~\ref{flat-lct-prj-preenvelope} (with the first assertion of
Lemma~\ref{flat-direct-images}) for improving such a morphism to
make it a termwise admissible monomorphism with the termwise
cokernel staying in the same desired exact subcategory $X\Ctrh^\fl$.
\end{proof}

\begin{prop} \label{noetherian-lcth-lct-preenvelope}
 Let $X$ be a Noetherian scheme of finite Krull dimension with
an open covering\/~$\bW$.
 Then, for any\/ $\bW$\+locally contraherent cosheaf\/ $\fM$ on $X$,
there exists a short exact sequence\/ $0\rarrow\fM\rarrow\fQ
\rarrow\fF\rarrow0$ in the exact category $X\Lcth_\bW$ with
a locally cotorsion\/ $\bW$\+locally contraherent cosheaf\/ $\fQ$
and a flat contraherent cosheaf\/ $\fF$ on~$X$.
\end{prop}

\begin{proof}
 This is a particular case of~\cite[Proposition~4.7.11]{Pform}.
 By Proposition~\ref{coflasque-lct-preenvelope}, the assertion of
the present proposition holds for any coflasque contraherent
cosheaf $\fM=\fE\in X\Ctrh_\cfq$.

 Furthermore, the full subcategory $X\Ctrh_\cfq$ is closed under
extensions and kernels of admissible epimorphisms in $X\Lcth_\bW$ by
Lemma~\ref{coflasque-closedness-properties-and-cosections}(a\+-b),
and it is generating by Corollary~\ref{noetherian-lcth-prj-cor}(a,c).
 So the full subcategory $X\Ctrh_\cfq$ is resolving in $X\Lcth_\bW$.
 By Lemma~\ref{coflasque-resol-dim}, the resolution dimensions of
all the objects of $X\Lcth_\bW$ with respect to the resolving
subcategory $X\Ctrh_\cfq$ do not exceed the Krull dimension of~$X$.
 Thus the $\bW$\+locally contraherent cosheaf $\fM$ has a finite
resolution by coflasque contraherent cosheaves in the exact
category $X\Lcth_\bW$.

 Applying Lemma~\ref{approximable-closed-under-cokers-of-admmonos}
iteratively, we obtain the desired short exact sequence for
any cosheaf $\fM\in X\Lcth_\bW$.
\end{proof}

 As a part of Proposition~\ref{noetherian-lcth-lct-preenvelope},
we obtain the second main result of this paper.

\begin{thm} \label{noetherian-findim-enough-lct-theorem}
 Let $X$ be a Noetherian scheme of finite Krull dimension with
an open covering\/~$\bW$.
 Then, for any\/ $\bW$\+locally contraherent cosheaf\/ $\fM$ on $X$,
there exists an admissible monomorphism\/ $\fM\rarrow\fQ$ in
the exact category $X\Lcth_\bW$ from\/ $\fM$ into a locally
cotorsion\/ $\bW$\+locally contraherent cosheaf\/
$\fQ\in X\Lcth^\lct_\bW$.  \qed
\end{thm}

 For completeness of the exposition, we deduce a couple more
corollaries before this paper is finished.

\begin{cor} \label{noetherian-lcth-flat-precover}
 Let $X$ be a Noetherian scheme of finite Krull dimension with
an open covering\/~$\bW$.
 Then, for any\/ $\bW$\+locally contraherent cosheaf\/ $\fM$ on $X$,
there exists a short exact sequence\/ $0\rarrow\fQ\rarrow\fF
\rarrow\fM\rarrow0$ in the exact category $X\Lcth_\bW$ with
a locally cotorsion\/ $\bW$\+locally contraherent cosheaf\/ $\fQ$
and a flat contraherent cosheaf\/ $\fF$ on~$X$.
\end{cor}

\begin{proof}
 This is a particular case of~\cite[Corollary~4.7.12]{Pform}.
 The proof is similar to those of
Lemmas~\ref{qcoh-ssep-alf-precover-lem}
and~\ref{flat-explicitly-described-precover}.
 The assertion follows from
Proposition~\ref{noetherian-lcth-lct-preenvelope} and
Corollary~\ref{noetherian-lcth-prj-cor}(a,c) by virtue of
Lemma~\ref{salce-lemma}(a).
\end{proof}

\begin{cor} \label{noetherian-findim-scheme-flat-cotorsion-pair}
 Let $X$ be a Noetherian scheme of finite Krull dimension with an open
covering\/~$\bW$.
 Then the pair of classes (flat contraherent cosheaves on $X$, locally
cotorsion\/ $\bW$\+locally contraherent cosheaves on~$X$) is a complete
cotorsion pair in the exact category of\/ $\bW$\+locally contraherent
cosheaves $X\Lcth_\bW$.
\end{cor}

\begin{proof}
 This is an improvement upon~\cite[Corollary~6.2.11]{Pcosh} and
a particular case of an assertion from~\cite[Section~4.7]{Pform}.
 This is also a non-semi-separated Noetherian version of
Corollary~\ref{antilocally-flat-cotorsion-pair} above, and the proof
is similar.
 One has $\Ext^{X,1}(\fF,\fQ)=0$ for all $\fF\in X\Ctrh^\fl$ and
$\fQ\in X\Lcth_\bW^\lct$ by
Corollary~\ref{flat-are-antilocally-flat}(a).
 The special preenvelope and special precover sequences exist
by Proposition~\ref{noetherian-lcth-lct-preenvelope} and
Corollary~\ref{noetherian-lcth-flat-precover}.
 Using the approximation sequences provided by these proposition
and corollary, one can show that $X\Ctrh^\fl\subset X\Lcth_\bW$ is
the maximal class with the $\Ext^1$\+orthogonality property with
respect to $X\Lcth_\bW^\lct$, and $X\Lcth_\bW^\lct\subset X\Lcth_\bW$
is the maximal class with the $\Ext^1$\+orthogonality property with
respect to $X\Ctrh^\fl$.
 This is an instance of the ``direct summand lemma'',
\cite[Lemma~B.1.2]{Pcosh} or~\cite[Lemma~A.2.3]{Pform}.
\end{proof}

\begin{cor} \label{noetherian-findim-flat=alf}
 Let $X$ be a Noetherian scheme of finite Krull dimension with an open
covering\/~$\bW$.
 Then a\/ $\bW$\+locally contraherent cosheaf on $X$ is flat if
and only if it is antilocaly flat (with respect to the given open
covering\/~$\bW$).
 Consequently, the class of all antilocally flat\/ $\bW$\+locally
contraherent cosheaves on $X$ does not depend on the open covering\/
$\bW$, and all such cosheaves are contraherent on the whole of~$X$.
\end{cor}

\begin{proof}
 This is a particular case of~\cite[Corollary~4.7.13(b)]{Pform}.
 Recall that the class of all flat $\bW$\+locally contraherent
cosheaves on $X$ does not depend on $\bW$ by
Corollary~\ref{flat-are-alf-contraherent-and-coflasque}(b),
and all flat contraherent cosheaves on $X$ are antilocally flat by
Corollary~\ref{flat-are-antilocally-flat}(a).
 As a part of
Corollary~\ref{noetherian-findim-scheme-flat-cotorsion-pair},
we have the equality $X\Ctrh^\fl={}^{\perp_1}(X\Lcth^\lct_\bW)
\subset X\Lcth_\bW$.
 The assertion that all antilocally flat $\bW$\+locally contraherent
cosheaves on $X$ are flat follows from that equality by virtue of
the result of~\cite[Lemma~A.2.2(b)]{Pform}, which is applicable since
the full subcategory $X\Lcth_\bW^\lct$ is cogenerating in $X\Lcth_\bW$
by Theorem~\ref{noetherian-findim-enough-lct-theorem}.
\end{proof}

\end{document}